\documentclass[10pt,reqno,final]{amsart}
\usepackage{epsfig,amssymb,amsmath,version}
\usepackage{amssymb,version,graphicx,fancybox,mathrsfs,multirow}
\usepackage{url,hyperref}
\usepackage[notcite,notref]{showkeys}

\usepackage{subfigure}%,mathabx
\usepackage{color,xcolor}
\usepackage{cases}
\usepackage{mathtools}

\catcode`\@=11 \theoremstyle{plain}
\@addtoreset{equation}{section}   % Makes \section reset 'equation' counter.

\@addtoreset{figure}{section}
\renewcommand\thefigure{\thesection.\@arabic\c@figure}
\renewcommand{\thefigure}{\arabic{section}.\arabic{figure}}
\newtheorem{thm}{\bf Theorem}

\newenvironment{theorem}{\begin{thm}} {\end{thm}}
\newtheorem{cor}{\bf Corollary}
\newtheorem{prop}{Proposition}[section]

\newtheorem{lmm}{\bf Lemma}

\newenvironment{lemma}{\begin{lmm}}{\end{lmm}}
\theoremstyle{remark}
\newtheorem{rem}{\bf Remark}[section]
\theoremstyle{definition}
\newtheorem{defn}{\bf Definition}[section]

 \numberwithin{table}{section}

\renewcommand \wedge \times

\graphicspath{{../Figures/}}

\begin{document}
\bibliographystyle{plain}
%\bibliographystyle{unsrt}

%\graphicspath{{../Figures/}}
%\baselineskip 14pt

\title[Approximation by Chebyshev expansions] {Optimal error estimates for Chebyshev  approximations of functions with limited regularity in fractional Sobolev-type spaces}
\author[
	W. Liu,\,    L. Wang\,  $\&$\,  H. Li
	]{
		\;\; Wenjie Liu${}^{1,2}$,   \;\;  Li-Lian Wang${}^{2}$ \;\; and\;\; Huiyuan Li${}^{3}$
		}
	
	\thanks{${}^{1}$Department of Mathematics, Harbin  Institute of Technology, 150001, China. Email: liu.wenjie@ntu.edu.sg. The research of this author is partially  supported by the China Postdoctoral Science Foundation Funded Project (No. 2017M620113), the National Natural Science Foundation of China (No. 11801120) and the Fundamental Research Funds for the Central Universities (Grant No.HIT.NSRIF.2019058).\\
		\indent ${}^{2}$Division of Mathematical Sciences, School of Physical
		and Mathematical Sciences, Nanyang Technological University,
		637371, Singapore. The research of this author is partially supported by Singapore MOE AcRF Tier 2 Grants: MOE2017-T2-2-014 and MOE2018-T2-1-059. Email: lilian@ntu.edu.sg.
		 \\
		\indent ${}^{3}$State Key Laboratory of Computer Science/Laboratory of Parallel Computing,  Institute of Software, Chinese Academy of Sciences, Beijing 100190, China. The research of this author is partially   supported by the National Natural Science Foundation of China (91130014, 11471312 and 91430216). Email:
	huiyuan@iscas.ac.cn}
	
\begin{abstract}  In this paper, we introduce  a new theoretical framework built upon  fractional Sobolev-type spaces involving Riemann-Liouville (RL) fractional integrals/derivatives, which is naturally arisen from exact  representations of Chebyshev expansion coefficients,  for optimal error estimates
of Chebyshev  approximations to functions with limited regularity.  % to a large class of singular functions with interior or endpoint singularities.
The essential pieces of the puzzle for the error  analysis  include (i) fractional integration by parts (under the weakest possible   conditions), and (ii)
 generalised Gegenbauer functions of fractional degree (GGF-Fs): a new family of special functions with notable   fractional calculus properties.   Under this framework,  we are able to  estimate  the optimal decay rate of Chebyshev expansion coefficients  for a large class of  functions with interior and endpoint singularities, which are deemed suboptimal or  complicated to characterize in  existing literature.
  %With these  new and sharp bounds of Chebyshev expansion coefficients at our disposal,
   We can then  derive optimal error estimates for spectral expansions and the related Chebyshev interpolation and quadrature measured in various norms, and  also  improve  the available results in usual Sobolev spaces of integer regularity exponentials  in several senses.
As a by-product,   this study results in some analytically perspicuous formulas particularly on  GGF-Fs,  which are potentially
useful in spectral algorithms. The idea and analysis techniques  can be extended to general Jacobi spectral approximations.
\end{abstract}
\maketitle

\vspace*{-15pt}

\thispagestyle{empty}

\section{Introduction}
%This paper is the first work  of a series on orthogonal  polynomial approximation of functions with limited regularity and interior/endpoint  singularities in newly defined fractional Sobolev-type spaces.

It is known that polynomial approximation theory is of fundamental importance in numerical analysis and  algorithm development of many computational methods, e.g.,     $p$/$hp$ finite elements  or spectral/spectral-element methods (see, e.g., \cite{Funa92,Schwab1998Book,MR1470226,CHQZ06,HGG07,ShenTangWang2011} and references therein).
% Many documented approximation results are for functions in (weighted) Sobolev spaces with integer order derivatives.
Typically, the  documented approximation results   take  the form
\begin{equation}\label{existestm}
\| Q_N u-u \|_{\mathcal S_{\rm l}}\le cN^{-\sigma} |u|_{\mathcal B_{\rm r}},\quad \sigma\ge 0,
\end{equation}
where $Q_N$ is an orthogonal projection (or interpolation operator) upon the set of all polynomials of degree at most $N,$ and $c$ is a positive constant independent of $N$ and $u$. In \eqref{existestm},    ${\mathcal S}_{\rm l}$ is  a certain Sobolev space,
${\mathcal B}_{\rm r}$ is a related  Sobolev  or Besov space, and $\sigma$ depends on the regularity exponentials of both ${\mathcal B}_{\rm r}$ and ${\mathcal S}_{\rm l}$. In practice,  one would expect  {\em {\rm \bf (a)} the space ${\mathcal B}_{\rm r}$ should contain the classes of functions as broad as possible; {\rm and  {\bf  (b)}}  the space ${\mathcal B}_{\rm r}$ can best characterise their regularity   leading to  optimal order
of  convergence.}  In general,  the  space ${\mathcal B}_{\rm r}$ is of the following types.    % in literature  is as follows.
\begin{itemize}
\item[(i)]  ${\mathcal B}_{\rm r}$ is the standard weighted Sobolev space $H^m_\omega(\Omega)$ with integer $m\ge 0$ and certain
weight function $\omega(x)$ on $\Omega=(-1,1)$ (see, e.g.,  \cite{MR1470226,CHQZ06,HGG07}).  However,  it  could not  lead to optimal order  for functions with endpoint singularities  (see, e.g., \cite{MR1470226,Guo.W04}) or with interior singularities, e.g., $|x|$ (see \cite{Trefethen2008SIREV}).

\vskip 4pt
\item[(ii)]  ${\mathcal B}_{\rm r}$ is the  non-uniformly Jacobi-weighted Sobolev space (see, e.g., \cite{Funa92,Bab.G00,Bab.G01,Guo.W04,Guo.SW06,ShenTangWang2011}).  For example,  ${\mathcal B}_{\rm r} =H^{m,\beta}(\Omega)$ with integer $m\ge 0$ and $\beta>-1,$  is defined as  a closure of $C^\infty$-functions endowed with the weighted norm
\begin{equation}\label{weightednorm}
\|u\|_{H^{m,\beta}(\Omega)}=\bigg\{\sum_{k=0}^m \int_{-1}^1 |u^{(k)}(x)|^2(1-x^2)^{\beta+k} dx  \bigg\}^{1/2}.
\end{equation}
Compared with the standard Sobolev space in (i), such spaces can better  describe the  endpoint singularities, but
 still produce suboptimal estimates for $(1+x)^\alpha$-type  singular functions with non-integer $\alpha>0$   (cf. \cite[P. 474]{Castillo2002MC}).
Indeed, for the Chebyshev approximation, we  %consider the Chebyshev polynomial expansion  of $u(x)=(1+x)^\gamma$ with non-integer $\gamma>0,$ and
find that  $u=(1+x)^\alpha \in  {H^{m,-1/2}(\Omega)}$ with integer  $m<2\alpha+1/2,$   and   % (cf. \cite{SchwabEst}):
\begin{equation}\label{Jacoweighted00}
   \|\pi_N^C u- u\|_{L^2_\omega(\Omega)}\le c N^{-m}|u|_{H^{m,-1/2}(\Omega)},  %\quad m=[2\gamma+1/2].
   \end{equation}
    where $\pi_N^C u$ is the $L^2_\omega$-orthogonal projection of $u$ (with $\omega=(1-x^2)^{-1/2}$). However, the expected optimal order is  $O(N^{-2\alpha-1/2}),$ so the loss of an order of the fractional part of $2\alpha+1/2$ or one order (when $2\alpha=k+1/2$ with $k\in {\mathbb N_0}$),  is inevitable under this framework. This is due to the space $H^{m,\beta}(\Omega)$ is only defined for integer $m\ge 0.$

\vskip 4pt
  % as it should be a certain fractional space.
\item[(iii)] In a series of works \cite{Bab.G00,Bab.G01,Babuska2002MMMAS},   Babu\v{s}ka and Guo  introduced the Jacobi-weighted Besov space
defined by space interpolation based on  the so-called K-method. % to optimally treat e.g., $(1+x)^\gamma$-type singularities.
One commonly used  Besov space for $(1+x)^\alpha$-type corner singularities  is  ${\mathcal B}^{s,\beta}_{2,2}(\Omega)=(H^{l,\beta}(\Omega),H^{m,\beta}(\Omega))_{\theta, 2}$ with integers $l<m$ and  $s=(1-\theta) l+\theta m, \theta\in (0,1),$  equipped with the norm
\begin{equation}\label{weightednorm22}
{~~}\qquad\quad \|u\|_{{\mathcal B}^{s,\beta}_{2,2}(\Omega)}=\bigg(\int_{0}^\infty t^{-2\theta} |K(t,u)|^2 \frac{dt} t  \bigg)^{ 1/2}\!,\quad
K(t,u)=\inf_{u=v+w} \big(\|v\|_{H^{l,\beta}(\Omega)}+t\|w\|_{H^{m,\beta}(\Omega)} \big).
\end{equation}
%This framework provides an optimal characterisation of  $(1+x)^\gamma$-type singularities.
However, to deal with $(1+x)^\alpha \log^\nu (1+x)$-type corner singularities,  Babu\v{s}ka and Guo had to further modify  the K-method by incorporating  a log-factor into  the norm.

%It is noteworthy that the Besov framework played a fundamental role in the numerical analysis of problems with corner singularities, which was  a topic of long-time interest (see, e.g.,  \cite{Babuska1981SINUM,Gui1986NM,Babuska1987MMAN, Schwab1998Book,Georgoulis2010JSC,Cos.DN12}).

% \item[(iv)]  As with the Fourier approximation,  ${\mathcal B}_r$ can be a  fractional Sobolev space related to the decay  rate of  the expansion coefficients
%  (cf.  \cite{Xiang2012SINUM,Canuto2014CMA, Glazyrinaa2015JAT}).  In practice,   one needs to know about the expansion coefficients to find the space which the underlying function belongs to.
\end{itemize}

 %\item[(iv)]
The aforementioned framework    might lead to  suboptimal estimates for  functions with interior singularities.
For example, we consider  $u(x)=|x|$  and note that   $u'(x)=2H(x)-1$  and $ u''(x)=2\delta(x)$ (where $H,\delta$ are respectively the Heaviside function and the Dirac delta function).   Since $u''\not \in L^2(\Omega)$,
the Chebyshev approximation of $|x|$ has a convergence:
   \begin{equation}\label{Jacoweighted}
   \|\pi_N^C u- u\|_{L^2_\omega(\Omega)}\le c N^{-1}|u|_{H^{1,-1/2}(\Omega)},
   \end{equation}
but   the expected  optimal order is  $O(N^{-3/2})$ (cf. \cite{Trefethen2008SIREV,Trefethen2013Book}). In fact,
as shown in  \cite[Thms 4.2-4.3]{Trefethen2008SIREV} and \cite[Thms 7.1-7.2]{Trefethen2013Book} (also see Lemma \ref{trenstha} below),  one should choose ${\mathcal B}_{\rm r}\subset {\rm BV}(\Omega)$ (the space of functions of bounded variation)
to achieve optimality (see Section \ref{sect:existingest}, and  refer to  \cite{Trefethen2008SIREV,Xiang2010NM,Trefethen2013Book,Majidian2017ANM}  for more details).
Unfortunately, the Sobolev spaces therein  were defined through integer-order derivatives, so they could not best
characterise  the regularity of e.g., $u(x)=|x|^{\alpha}$ with non-integer $\alpha>0.$ In other words,  the order of convergence can only be  suboptimal.

In this paper, we intend  to introduce a new framework of fractional Sobolev-type spaces that can  meet the two  requirements {\bf (a)-(b)} and
overcome the deficiencies mentioned above.
 %so as to achieve the optimal convergence order for orthogonal polynomial  approximations of functions with limited regularity.
We focus on the Chebyshev approximation but the  analysis techniques
are extendable to general Jacobi approximations.    Here, we  put the emphasises  on
estimating  the decay rate of  expansion coefficients for the reason that  the errors of spectral expansions in various norms, and the related interpolation and quadratures
  can be estimated directly from the sums of the coefficients (cf. \cite{Trefethen2008SIREV,Majidian2017ANM}).
%,   the error analysis of Chebyshev orthogonal projections, and the related   essentially depends on estimating the decay rate  of $|\hat u_n^C|.$
%
% in \cite{Majidian2017ANM},   the error analysis of Chebyshev orthogonal projections, and the related interpolation and quadrature errors  essentially depends on estimating the decay rate  of $|\hat u_n^C|.$
The essential ideas and main contributions of this study  are summarised as follows.
 \begin{itemize}
\item[(i)]   We derive the exact representation of the Chebyshev expansion coefficients  (see Theorem \ref{IdentityForUn}) by using  the fractional calculus properties of  GGF-Fs and  fractional integration by parts (under the weakest possible conditions).
  This allows us to naturally define the fractional Sobolev spaces to characterise the regularity
 of a large class of singular functions, leading to optimal order of convergence.
%
%     We  show that this framework .  It can handle both interior singularities (as the fractional spaces in-between the spaces of integer-order derivatives of $u^{(m)}$ in (iv)),
%  and endpoint singularities (as an alternative to the Besov-framework in (iii) for dealing with corner singularities through more explicit norms).
% Here,  we focus on  Chebyshev approximation. % and stress the analysis of the expansion coefficients.
%  \item[(1)] Under this new framework, we derive  optimal bounds for Chebyshev approximation.
%  for a large class of  singular functions with interior and endpoint singularities. This can be extended to the approximation by Jacobi polynomials.

\vskip 4pt
   \item [(ii)] When the fractional regularity exponential  $s\to 1,$  our results  improve  the existing bounds in usual Sobolev spaces (see, e.g., \cite{Trefethen2008SIREV,Xiang2010NM,Trefethen2013Book,Majidian2017ANM}).
   % in several senses (cf. Section \ref{sect:existingest}).
In this sense,   the fractional Sobolev-type space with regularity exponential $m+s$ ($s\in (0,1)$ and integer $m\ge 0$)  can be regarded as  an intermediate space inbetween the spaces
with regularity exponentials $m+1$ and $m$ in  \cite{Trefethen2013Book}.

    \item[(iii)]  We provide some useful analytical formulas on fractional calculus of GGF-Fs, and  the Chebyshev expansions of some specific singular functions.    Some of them  are new or difficult to be derived by other means  (cf. \cite{Gui1986NM,Boyd1989AMC,Wang2014arXiv}).  They can  also be useful in the design of spectral algorithms.
 \end{itemize}
\vskip 4pt

 The paper is organised as follows.  In Sections \ref{sect:gegenbaure}-\ref{sect:fintderG}, we introduce the GGF-Fs, and present their important  properties, including the uniform bounds and RL fractional integral/derivative formulas.  We derive the main results in Section \ref{mainsect:ms}, and    improve the existing estimates in Sobolev spaces with integer-order derivatives in Section \ref{sect:existingest}.   We discuss in Section \ref{Sect6Analysis} the extension of the main results to the analysis of interpolation, quadrature and endpoint singularities.

\section{Generalised Gegenbauer functions of fractional degree}\label{sect:gegenbaure}
In this section,  we collect some relevant properties of the hypergeometric functions and Gegenbauer polynomials, upon which
we define  the GGF-Fs and derive their relevant properties. These pave the way for the  forthcoming  error  analysis.
 %As we shall see, the GGF-Fs play an important part in the forthcoming  error  analysis.

\subsection{Hypergeometric functions and Gegenbauer polynomials}
Let $\mathbb Z$ and $\mathbb R$ be the sets of all integers and real numbers, respectively,  and denote
\begin{equation}\label{mathNR}
\mathbb N=\big\{k\in {\mathbb Z}: k\ge 1\big\},\;\;   {\mathbb N}_0:=\{0\}\cup {\mathbb N}, \;\; {\mathbb R}^+:= \big\{a\in {\mathbb R}: a> 0\big\}, \;\; {\mathbb R}^+_0:=\{0\}\cup {\mathbb R}^+.
\end{equation}
For $a\in {\mathbb R},$  the rising factorial in the Pochhammer symbol is defined by
\begin{equation}\label{anotation}
(a)_0=1; \;\;\;\;  % (a)_j=a(a+1)\cdots (a+j-1),\;\;  {\rm for}\;\; j>0;\quad
(a)_j=a(a+1)\cdots (a+j-1), \;\;\;\forall\, j\in {\mathbb N}. %=\frac{\Gamma(a+j)}{\Gamma(a)},\;\; {\rm for}\;\; j\ge 1.
\end{equation}

 The  hypergeometric function is a power series, defined by  (cf. \cite{Andrews1999Book})
\begin{equation}\label{hyperboscs}
{}_2F_1(a,b;c; z)=\sum_{j=0}^\infty \frac{(a)_j(b)_j}{(c)_j}\frac{z^j}{j!} =
1+
\sum_{j=1}^\infty \frac{a(a+1)\cdots (a+j-1)}{1\cdot 2\cdots j} \frac{b(b+1)\cdots (b+j-1)}
{c (c+1)\cdots (c+j-1)}
 {z^j},
\end{equation}
where  $a,b,c\in {\mathbb R}$ and $-c\not \in {\mathbb N}_0.$
The series  converges absolutely for all $|z|<1$, and apparently, we have
 \begin{equation}\label{obvsfact}
 {}_2F_1(a,b;c; 0)=1,\quad {}_2F_1(a,b;c; z)={}_2F_1(b,a;c; z).
 \end{equation}
If $a=-n$ with $n\in {\mathbb N}_0,$  then $(a)_j=0, j\ge n+1,$ so ${}_2F_1(-n,b;c; x)$ reduces to a polynomial of degree not more than  $n.$

 The following properties
  can be found in \cite[Ch.\!\! 2]{Andrews1999Book}, if not stated otherwise.
  %The series might diverge at $z=\pm 1$ with the following behaviors.
  \begin{itemize}
  \item If $c-a-b>0,$ the series \eqref{hyperboscs} converges absolutely at $z=\pm 1,$ and
    \begin{equation}\label{GaussAnd}
{}_2F_1(a,b;c;1)=\frac{\Gamma(c)\Gamma(c-a-b)}{\Gamma(c-a)\Gamma(c-b)}. %,\quad c-a-b>0,
\end{equation}
Here,  the Gamma function with negative non-integer arguments should be  understood by the Euler's reflection formula: % (cf.  \cite{Abr.I64}):
\begin{equation}\label{nonitA}
\Gamma(1-a)\Gamma(a)=\frac{\pi} {\sin (\pi a)},\quad a\not \in \mathbb Z.
\end{equation}
Note that   $\Gamma(-a)=\infty,$  if $a\in {\mathbb N}.$
\item If $-1<c-a-b\le 0,$ the series \eqref{hyperboscs} converges conditionally at $z=-1,$ but diverges at $z=1;$ while for
$ c-(a+b)\le -1$, it diverges at $z=\pm 1.$ In fact, it has the following singular behaviours at $z=1:$
\begin{equation}\label{Nist15421}
\lim_{z\to 1^-}\frac{ {}_2F_1(a,b;c;z)}{-\ln (1-z)}=\frac{\Gamma(c)}{\Gamma(a)\Gamma(b)},\quad {\rm if}\;\; c=a+b,
\end{equation}
and
\begin{equation}\label{Nist15421cc}
\lim_{z\to 1^-}\frac{ {}_2F_1(a,b;c;z)}{(1-z)^{c-a-b}}=\frac{\Gamma(c)\Gamma(a+b-c)}{\Gamma(a)\Gamma(b)},\quad {\rm if}\;\; c<a+b.
\end{equation}
\end{itemize}
Recall  the transform identity: for $a,b,c \in {\mathbb R}$ and $-c\not \in {\mathbb N}_0,$
\begin{align}\label{Euler}
{}_2F_1(a,b;c;z) =(1-z)^{c-a-b}{}_2F_1(c-a,c-b;c;z),\quad |z|<1.
\end{align}
The hypergeometric function satisfies the differential equation (cf. \cite[P. 98]{Andrews1999Book}):
\begin{equation}\label{SLPHF}
\begin{split}
\big\{z^c(1-z)^{a+b-c+1}y'(z) \big\}'=abz^{c-1}
(1-z)^{a+b-c}y(z),\quad y(z)={}_2F_1(a,b;c;z).
\end{split}
\end{equation}
We shall use the value  at $z=1/2$ (cf. \cite[(15.4.28)]{Olver2010Book}): %(cf.  \cite[(P. 148)]{Andrews1999Book})
\begin{align}\label{Fatzero}
{}_2F_1\Big(a,b;\frac{a+b+1}2;\frac 1 2\Big) =
\frac{\sqrt{\pi}\,\Gamma((a+b+1)/2)}{\Gamma((a+1)/2)\Gamma((b+1)/2)}.
\end{align}

%\subsection{Generalized Gegenbauer functions of fractional degree}

Many functions are  associated with the hypergeometric function.
For example,  the Jacobi polynomial of degree $n\in {\mathbb N}_0$  with   $\alpha,\beta>-1$  (cf.   Szeg\"o \cite{szeg75}) is  defined by
\begin{equation}\label{Jacobidefn00}
\begin{split}
P_n^{(\alpha,\beta)}(x)&=\frac{(\alpha+1)_n}{n!}{}_2F_1\Big(\!\!-n, n+\alpha+\beta+1;\alpha+1;\frac{1-x} 2\Big)\\
&=(-1)^n \frac{(\beta+1)_n}{n!}{}_2F_1\Big(\!\!-n, n+\alpha+\beta+1;\beta+1;\frac{1+x} 2\Big),\;\;\;  x\in (-1,1),
\end{split}
\end{equation}
 which satisfies
\begin{equation}\label{Gjbinp}
P_n^{(\alpha,\beta)}(-x)=(-1)^n P_n^{(\beta,\alpha)}(x),\quad  P_n^{(\alpha,\beta)}(1)= \frac{(\alpha+1)_n}{n!}.
\end{equation}
For  $\alpha,\beta>-1,$   the   Jacobi polynomials  are orthogonal with respect to the Jacobi weight function:  $\omega^{(\alpha,\beta)}(x) = (1-x)^{\alpha}(1+x)^{\beta},$ namely,
\begin{equation}\label{jcbiorth}
\int_{-1}^1 {P}_n^{(\alpha,\beta)}(x) {P}_{n'}^{(\alpha,\beta)}(x) \omega^{(\alpha,\beta)}(x) \, dx= \gamma _n^{(\alpha,\beta)} \delta_{nn'},
\end{equation}
where $\delta_{nn'}$ is the Kronecker Delta symbol, and
\begin{equation}\label{co-gamma}
\gamma _n^{(\alpha,\beta)} =\frac{2^{\alpha+\beta+1}\Gamma(n+\alpha+1)\Gamma(n+\beta+1)}{(2n+\alpha+\beta+1) n!\,\Gamma(n+\alpha+\beta+1)}.
\end{equation}
\begin{rem}\label{identityCase} {\em The definition \eqref{Jacobidefn00} is valid for all  $\alpha,\beta\in {\mathbb R}$.
In fact,  many properties  of the classical Jacobi polynomials  {\rm(}e.g., \eqref{Gjbinp}{\rm)} still hold, but the orthogonality is lacking  in general {\rm(}cf.   Szeg\"o \cite[P. 63-67]{szeg75}{\rm).}}
\end{rem}

Throughout this paper,  the  Gegenbauer polynomial with  $\lambda>-1/2$ is defined by
\begin{equation}\label{Jacobidefn0}
\begin{split}
G_n^{(\lambda)}(x)&=\frac{P_n^{(\lambda-1/2,\lambda-1/2)}(x)} {P_n^{(\lambda-1/2,\lambda-1/2)}(1)}
={}_2F_1\Big(\!-n, n+2\lambda;\lambda+\frac 1 2;\frac{1-x} 2\Big)\\
&=(-1)^n\,{}_2F_1\Big(\!-n, n+2\lambda;\lambda+\frac 1 2;\frac{1+x} 2\Big), \;\; x\in (-1,1),
\end{split}
\end{equation}
which has  a normalization different from that in Szeg\"o \cite{szeg75}.  If $\lambda=0,$ it reduces to the Chebyshev polynomial   % is (cf. \cite[8.942]{Gradshteyn2015Book}):
\begin{equation}\label{Chebydefn0}
T_n(x)=G_n^{(0)}(x)={}_2F_1\Big(\!\!-n, n; \frac 1 2; \frac{1-x} 2\Big)=\cos (n\, {\rm arccos}(x)).
\end{equation}
Note that under the above  normalization,   we derive  from \eqref{jcbiorth}-\eqref{co-gamma}  the orthogonality:
\begin{equation}\label{genorth}
    \int_{-1}^1 {G}_n^{(\lambda)}(x) {G}_m^{(\lambda)}(x)\,  \omega_{\lambda}(x) \, dx= \gamma _n^{(\lambda)} \delta_{nm};\;\;\; \gamma _n^{(\lambda)} =\frac{2^{2\lambda-1}\Gamma^2(\lambda+1/2)\, n!} {(n+\lambda) \Gamma(n+2\lambda)},
\end{equation}
where $\omega_\lambda(x) = (1-x^2)^{\lambda-1/2}.$
%\begin{equation}\label{co-gamma}
%\omega_\lambda(x) = (1-x^2)^{\lambda-1/2},\quad
%\end{equation}
In the analysis, we shall use the derivative relation derived from the generalised Rodrigues' formula (see \cite[(4.10.1)]{szeg75} with
 $\alpha=\beta=\lambda-1/2>-1$ and $m=1$):
	\begin{equation}\label{RodriguesF}
	\begin{split}
	\omega_{\lambda}&(x) G_{n}^{(\lambda)}(x)=\, -\frac{1} {2\lambda+1}\, \frac{d}{dx} \big\{ \omega_{\lambda+1}(x) G_{n-1}^{(\lambda+1)}(x) \big\} ,  \;\;\; n\ge 1.
	\end{split}
	\end{equation}
%and also the Sturm-Liouville equation (cf. \cite[(4.2.2)]{szeg75}):
%\begin{equation}\label{SLP}
%\begin{split}
%\frac{d}{dx}\Big\{(1-x^2)^{\lambda+1/2}\frac{d}{dx}G_n^{(\lambda)}(x)\Big\}=-n(n+2\lambda)(1-x^2)^{\lambda-1/2}G_n^{(\lambda)}(x).
%\end{split}
%\end{equation}

\subsection{Generalised Gegenbauer functions of fractional degree} As an indispensable tool for the error analysis,
we introduce   the GGF-Fs  by allowing the degree  $n$  of the Gegenbauer polynomials in \eqref{Jacobidefn0}  to be real.
\begin{defn}\label{gjfdefinition}  {\em For real $\lambda>-1/2$ and real $\nu\ge 0,$ the right GGF-F of degree $\nu$ is defined by
\begin{equation}\label{rgjfdef}
{}^{r\!}G_\nu^{(\lambda)}(x)=\, {}_2F_1\Big(\!\!-\nu, \nu+2\lambda;\lambda+\frac 1 2;\frac{1-x} 2\Big)=1+\sum_{j=1}^\infty \frac{(-\nu)_j(\nu+2\lambda)_j}{
j!\; (\lambda+1/2)_j }\Big(\frac{1-x}{2}\Big)^j,
\end{equation}
for $x\in (1,1);$ while the left GGF-F of degree $\nu$ is defined by
\begin{equation}\label{lgjfdef}
\begin{split}
{}^{l}G_\nu^{(\lambda)}(x)= &(-1)^{[\nu]}  \, {}_2F_1\Big(\!\!-\nu,\nu+2\lambda;\lambda+\frac 1 2;\frac{1+x}{2}\Big)= (-1)^{[\nu]}  \,\bigg\{1+
\sum_{j=1}^\infty \frac{(-\nu)_j(\nu+2\lambda)_j}{
 j!\; (\lambda+1/2)_j}\Big(\frac{1+x}{2}\Big)^j\bigg\}, %\quad  x\in (-1,1),
\end{split}
\end{equation}
where $[\nu]$ is the largest integer $\le \nu.$ } \qed
\end{defn}
% round to nearest integer towards zero
\begin{rem}\label{olddefinition} {\em
 For  $\lambda=1/2,$  the right GGF-F  turns to be  the Legendre function {\rm(}cf. \!\cite{Andrews1999Book}{\rm)}: $P_\nu(x)={}^{r\!}G_\nu^{(1/2)}(x).$
 In  Handbook  \cite[(15.9.15)]{Olver2010Book},  ${}^{r\!}G_\nu^{(\lambda)}(x)$ {\rm(}with a different normalisation{\rm)} is defined  as
the  Gegenbauer function. However, there is nearly no discussion on its properties.}
\end{rem}

Observe from  \eqref{Jacobidefn0} and Definition \ref{gjfdefinition} that the GGF-Fs reduce to the classical  Gegenbauer polynomials when $\nu\in {\mathbb N}_0,$  but they are non-polynomials when $\nu$ is not an integer.
\begin{prop}\label{obvprop00} The GGF-Fs defined in Definition \ref{gjfdefinition} satisfy
\begin{subequations}
\begin{equation}\label{obsvers0}
  {}^{r\!}G_n^{(\lambda)}(x)=  {}^lG_n^{(\lambda)}(x)=G_n^{(\lambda)}(x), \quad n\in {\mathbb N}_0;
\end{equation}
\begin{equation}\label{obsvers}
{}^{r\!}G_\nu^{(\lambda)}(-x)=(-1)^{[\nu]} \,
{}^{l}G_\nu^{(\lambda)}(x), \quad {}^{r\!}G_\nu^{(\lambda)}(1)=1,\quad
{}^{l}G_\nu^{(\lambda)}(-1)=(-1)^{[\nu]}.
\end{equation}
\end{subequations}
\end{prop}

 The special  GGF-Fs $\big\{{}^{r\!}G_{n-\alpha}^{(\alpha+1/2)}(x)\big\}$ and $\big\{{}^{l}G_{n-\alpha}^{(\alpha+1/2)}(x)\big\}$
 are closely related to   the Jacobi polynomials with the  parameters maybe $\le -1$ (cf. Remark \ref{identityCase}).
\begin{prop}\label{nonvemberA}   For $\alpha>-1 $ and $n\ge   \alpha $ with $n\in {\mathbb N}_0,$ we have
%\begin{subequations}\label{Pnlambdastar}
\begin{equation}\label{lmbda0}
  \frac{P_n^{(\alpha,-\alpha)}(x)}{P_n^{(\alpha,-\alpha)}(1)} =\Big(\frac {1+x}2\Big)^{\alpha} \,
  {}^{r\!}G_{n-\alpha}^{(\alpha+1/2)}(x);\quad  \frac{P_n^{(-\alpha,\alpha)}(x)}{P_n^{(\alpha,-\alpha)}(1)} =(-1)^{[\alpha]}\Big(\frac {1-x}2\Big)^{\alpha} \, {}^{l}G_{n-\alpha}^{(\alpha+1/2)}(x)\,.
%  =\frac{ 2^\lambda n!} {(\lambda+1)_n} (1+x)^{-\lambda} ,
\end{equation}
%\begin{equation}\label{lmbda01}
%  \frac{P_n^{(-\alpha,\alpha)}(x)}{P_n^{(\alpha,-\alpha)}(1)} =(-1)^{[\alpha]}\Big(\frac {1-x}2\Big)^{\alpha} \, {}^{l}G_{n-\alpha}^{(\lambda)}(x)\,.
%\end{equation}
%\end{subequations}
\end{prop}
\begin{proof}
Taking $a=-n+\alpha, b=n+\alpha+1, c=\alpha+1$ and $z=(1-x)/2$ in \eqref{Euler},  we obtain from \eqref{obvsfact} that
 \begin{equation}\label{defeqnA01}
 \begin{split}
{}_2F_1\Big(\!\!-n+\alpha, & \,n+\alpha+1;\alpha+1; \frac{1-x}{2}\Big)=\Big(\frac {1+x} 2\Big)^{-\alpha}\!
{}_2F_1\Big(\!\!-n,n+1;\alpha+1;\frac{1-x}{2}\Big).
%&=\Big(\frac {1+x} 2\Big)^{-\lambda}\,  \frac{P_n^{(\lambda,-\lambda)}(x)} {P_n^{(\lambda,-\lambda)}(1)}=\Big(\frac {1+x} 2\Big)^{-\lambda}\,
%\frac{n!
%} {(\lambda+1)_n} P_n^{(\lambda,-\lambda)}(x)\,.
   \end{split}
\end{equation}
By \eqref{Jacobidefn00}-\eqref{Gjbinp},
$$
{}_2F_1\Big(\!\!-n,n+1;\alpha+1;\frac{1-x}{2}\Big)=   \frac{P_n^{(\alpha,-\alpha)}(x)}{P_n^{(\alpha,-\alpha)}(1)}\,,
$$
and by \eqref{rgjfdef} (taking $\nu=n-\alpha$),  the hypergeometric function in the left-hand side of \eqref{defeqnA01} equals to  $\, {}^{r\!}G_{n-\alpha}^{(\alpha+1/2)}(x).$ Thus, we derive the first identity in \eqref{lmbda0}.

Thanks to \eqref{Gjbinp} and   \eqref{obsvers}, the second identity in  \eqref{lmbda0} follows from the first one immediately.
\end{proof}
\begin{rem}\label{spJacibGlam} {\em   If  $-1<\alpha <1, $ we  rewrite \eqref{lmbda0} as
\begin{equation}\label{basisfunc}
{}^{r\!}G_{n-\alpha}^{(\alpha+1/2)}(x)=d_{n,\alpha} (1+x)^{-\alpha} P_n^{(\alpha,-\alpha)}(x);\;\;\;
{}^{l}G_{n-\alpha}^{(\alpha+1/2)}(x)=(-1)^{[\alpha]}d_{n,\alpha} (1-x)^{-\alpha} P_n^{(-\alpha,\alpha)}(x),
\end{equation}
where $d_{n,\alpha}=2^{\alpha}/P_n^{(\alpha,-\alpha)}(1).$ From  \eqref{jcbiorth}-\eqref{co-gamma}, we immediately obtain  the orthogonality:
\begin{equation}\label{orhtognalityA}
\begin{split}
\int_{-1}^{1} &  {}^{r\!}G_{n-\alpha}^{(\alpha+1/2)}(x)\,  {}^{r\!}G_{m-\alpha}^{(\alpha+1/2)}(x)\,  (1-x^2)^{\alpha}\, dx \\
%\qquad \text{(note: $\omega_\lambda(x)=(1-x^2)^{\alpha}$) } \\
 &=d_{n,\alpha} d_{m,\alpha} \int_{-1}^1 P_n^{(\alpha,-\alpha)}(x) P_m^{(\alpha,-\alpha)}(x)
(1-x)^\alpha (1+x)^{-\alpha}\, dx =d_{n,\alpha}^2 \gamma_n^{(\alpha,-\alpha)} \delta_{mm},
\end{split}
\end{equation}
and likewise for  $\{{}^{l}G_{n-\alpha}^{(\alpha+1/2)}(x)\}.$  It is noteworthy  that  $\{(1+x)^{-\alpha} P_n^{(\alpha,-\alpha)}\}$ are defined as the Jacobi polyfractonomials in  \cite{zayernouri2013fractional} and  special  generalised Jacobi functions in  \cite{Guo.SW09,Chen.SW2016}, which   serve as effective {\rm(}singular{\rm)} basis functions  in  accurate solutions of  fractional differential  equations {\rm(}cf. \cite{zayernouri2013fractional,Chen.SW2016}{\rm).}  It is seen from  \eqref{basisfunc} that they turn out to be special GGF-Fs. }
\end{rem}

%Observe  from \eqref{lmbda0}  that for $-1/2<\lambda<1/2,$ ${}^{r\!}G_{n-\alpha}^{(\lambda)}(x)\to 0$ as $x\to -1^{+},$ while it diverges at $x=-1$ for $\lambda>1/2.$
%As stated below,
It is important to point out that the GGF-Fs may be singular at $x=\pm 1,$ and they  behave differently in different ranges of $\lambda$.
\begin{prop}\label{nonvember}  Let $\nu\in {\mathbb R}^+_0.$   %and $\alpha=\lambda-1/2.$
\begin{itemize}
\item[(i)] If $-1/2<\lambda<1/2$, then
\begin{equation}\label{defBS00}
 {}^{r\!}G_\nu^{(\lambda)}(-1) =
\frac{ \cos((\nu+\lambda)\pi)}{\cos(\lambda\pi)}= (-1)^{[\nu]}\,  {}^{l}G_\nu^{(\lambda)}(1)\,.
 \end{equation}
\item[(ii)] If $\lambda=1/2$ and $\nu\not \in {\mathbb N}_0,$ then
\begin{equation}\label{defBS0A}
\lim_{x\to -1^+}\frac{{}^{r\!}G_\nu^{(\lambda)}(x)}{\ln(1+x)}= \frac{\sin (\nu\pi)} {\pi} =  \lim_{x\to 1^-}\frac{(-1)^{[ \nu]}\, {}^{l}G_\nu^{(\lambda)}(x)}{\ln(1-x)}\,.
\end{equation}
\item[(iii)] If $\lambda>1/2$ and $\nu\not \in {\mathbb N}_0,$ then
\begin{equation}\label{defBS0B}
 \begin{split}
\lim_{x\to -1^+} \Big(\frac{1+x} 2\Big)^{\lambda-1/2}\,  {}^{r\!}G_\nu^{(\lambda)}(x)&=
 -\frac{\sin(\nu\pi)}\pi
\frac{\Gamma(\lambda-1/2)\Gamma(\lambda+1/2)\Gamma(\nu+1)}{\Gamma(\nu+2\lambda)}\\
&= (-1)^{[\nu]} \!\lim_{x\to 1^-} \Big(\frac{1-x} 2\Big)^{\lambda-1/2}\,  {}^{l}G_\nu^{(\lambda)}(x).
   \end{split}
\end{equation}
\end{itemize}
\end{prop}
\begin{proof}  Thanks to  \eqref{obsvers}, it suffices to prove the results for  $ {}^{r\!}G_\nu^{(\lambda)}(x).$

 (i)  By  \eqref{GaussAnd}, \eqref{nonitA} and  \eqref{rgjfdef},
 \begin{equation}\label{defBS}
 \begin{split}
 {}^{r\!}G_\nu^{(\lambda)}(-1) &=
\, {}_2F_1(-\nu,\nu+2\lambda;\lambda+1/2;1)=
\frac{\Gamma(\lambda+1/2)\Gamma(1/2-\lambda)}{\Gamma(\nu+\lambda+1/2)\Gamma(-\nu-\lambda+1/2)} \\
&= \frac{\pi} {\sin((\lambda+1/2) \pi )} \frac{\sin((\nu+\lambda+1/2)\pi )}{\pi}
=\frac{ \cos((\nu+\lambda)\pi)}{\cos(\lambda\pi)}\,,
   \end{split}
\end{equation}
which yields \eqref{defBS00}.

(ii) Using \eqref{nonitA}, \eqref{Nist15421}  and \eqref{rgjfdef}, and noting that $\ln ((1+x)/2)/\ln(1+x)\to 1$ as $x\to -1^+,$  we obtain \eqref{defBS0A}.

%To derive \eqref{defBS0B},  we resort to the transform identity  (cf. \cite[P. 68]{Andrews1999Book}): for $a,b,c \in {\mathbb R}$ and $-c\not \in {\mathbb N}_0,$
%\begin{align}\label{Euler}
%{}_2F_1(a,b;c;z) =(1-z)^{c-a-b}{}_2F_1(c-a,c-b;c;z),\quad |z|<1.
%\end{align}

(iii) Next, taking $a=-\nu, b=\nu+2\lambda, c=\lambda+1/2$ and $z=(1-x)/2$ in  \eqref{Euler},  and using \eqref{obvsfact},  we obtain
 \begin{equation*}\label{defeqnA}
 \begin{split}
{}_2F_1\Big(\!\!-\nu, & \,\nu+2\lambda;\lambda+\frac 1 2; \frac{1-x}{2}\Big)=\Big(\frac  2{1+x} \Big)^{\lambda-1/2}\,  {}_2F_1\Big(\nu+\lambda+\frac1 2,-\nu-\lambda+\frac 1 2;\lambda+\frac 1 2;\frac{1-x}{2}\Big).
%&=\Big(\frac {1+x} 2\Big)^{-\beta}\,
%{}_2F_1\Big(\!-\nu-\beta,\nu+\alpha+1;\alpha+1;\frac{1-x}{2}\Big).
   \end{split}
\end{equation*}
For $\lambda>1/2,$ we find from \eqref{GaussAnd} and \eqref{nonitA} that
 \begin{equation*}\label{defeqnA1}
{}_2F_1\Big(\nu+\lambda+\frac 1 2,-\nu-\lambda+\frac 1 2;\lambda+\frac 1 2;1\Big)=-\frac{\sin(\nu\pi)}\pi
\frac{\Gamma(\lambda-1/2)\Gamma(\lambda+1/2)\Gamma(\nu+1)}{\Gamma(\nu+2\lambda)},
\end{equation*}
so we derive  \eqref{defBS0B} from  \eqref{rgjfdef} and the above.
\end{proof}

%From the above definition and \eqref{hyperboscs}, we observe that  for any $\alpha,\beta>-1, n\in
%{\mathbb N}_0$ and $\nu\in {\mathbb R}_0^+,$

%In view of \eqref{Nist15421}-\eqref{Euler} and Definition \ref{gjfdefinition}, the series ${}^{r\!}G_\nu^{(\lambda)}(x)$ (resp.
%${}^{l}G_\nu^{(\lambda)}(x)$) may diverge at $x=-1$ (resp. $x=1$), when $\nu$ is not an integer.

As some illustrations, we depict  in Figure \ref{FigForRGGF}  the right generalised Chebyshev/Legendre functions, i.e.,
${}^{r\!}G_\nu^{(\lambda)}(x)$ with $\lambda=0, 1/2$ and for various $\nu$. Note that the left counterparts   ${}^{l}G_\nu^{(\lambda)}(x)=(-1)^{[\nu]}\, {}^{r\!}G_\nu^{(\lambda)}(-x)$ (cf. \eqref{obsvers}). Observe that in the Legendre case (the figure on the right),  ${}^{r\!}G_\nu^{(1/2)}(x)$ with non-integer degree  has a logarithmic singularity at $x=-1$  (cf. \eqref{defBS0A}), while the generalised Chebyshev functions (left) are well defined  at $x=-1$.

\vskip 5pt
\begin{figure}[!ht]
	\begin{center}
	%{~}\hspace*{-20pt}	
	\includegraphics[width=0.45\textwidth]{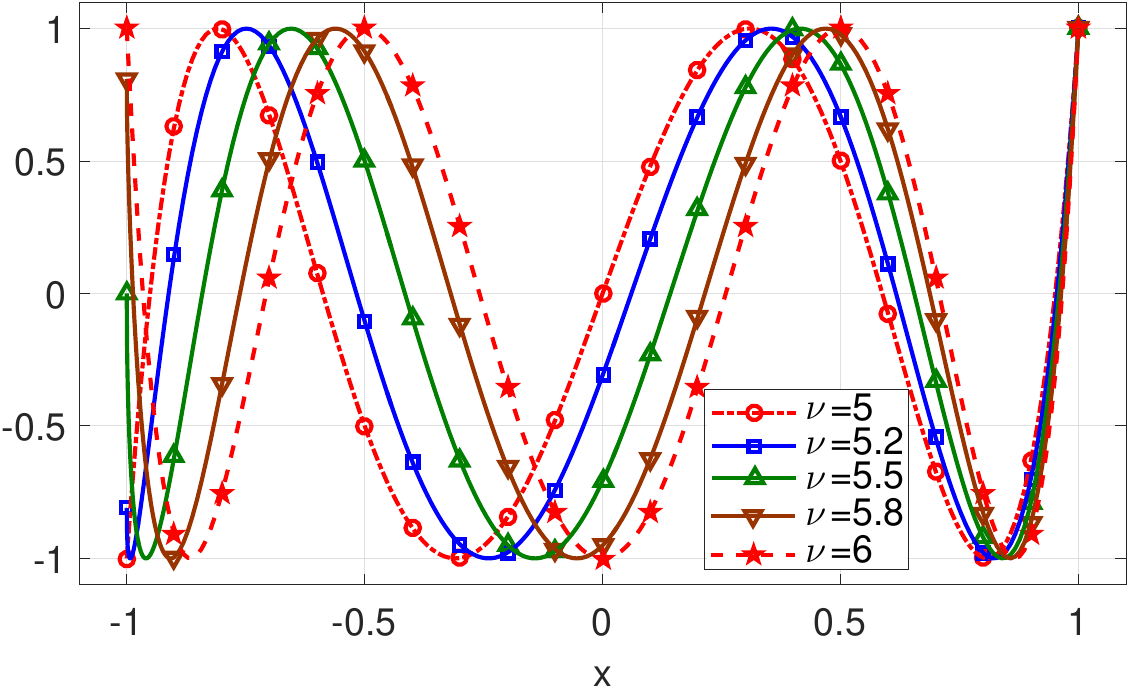} \qquad
		\includegraphics[width=0.45\textwidth]{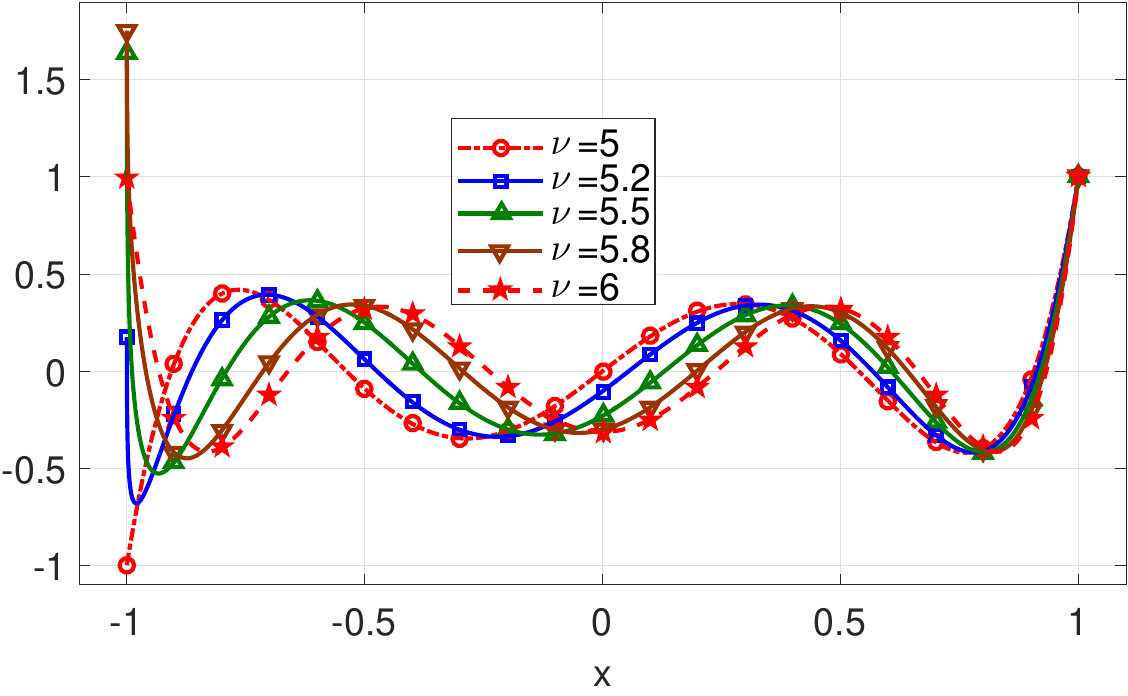}
	\caption{Graphs of ${}^{r\!}G_\nu^{(\lambda)}(x)$ with  $\lambda=0$ (left) and  $\lambda=1/2$ (right) for various $\nu$.}
	\label{FigForRGGF}
		\end{center}
\end{figure}

\subsection{Uniform upper bounds}
The uniform bounds of the GGF-Fs stated in the following two theorems  are of paramount importance  in the forthcoming  error analysis.
\begin{theorem}\label{LBoundForGegPoly} For   $\lambda \ge 1$ and  real $\nu\ge 0,$ we have
	\begin{equation}\label{LBoundGeg}
	\max_{|x|\le 1}\big\{	\omega_{\lambda}(x) |{}^{r\!}G_\nu^{(\lambda)}(x)|,	\; \omega_{\lambda}(x)| {}^{l}G_\nu^{(\lambda)}(x)| \big\}\le
	 \kappa_\nu^{(\lambda)},
	\end{equation}
	where  $\omega_\lambda(x)=(1-x^2)^{\lambda-1/2}$ and
	\begin{equation}\label{kappaNl}
	\begin{split}
	 \kappa_\nu^{(\lambda)}=\frac{\Gamma(\lambda+ 1/2)}{\sqrt \pi}\bigg(
\frac{ \cos^2(\pi\nu/ 2) \Gamma^2( ({\nu}+1)/ 2)}{\Gamma^2( ({\nu}+1)/2+\lambda)}+\frac{4\sin ^2\big(\pi  {\nu} /2\big)}{2\lambda-1+\nu(\nu+2\lambda)}\frac{\Gamma^2({\nu}/2+1)}{\Gamma^2({\nu}/2+\lambda)}\bigg)^{1/2}.
\end{split}	\end{equation}
\end{theorem}
\begin{proof}   Thanks to  \eqref{obsvers},  it suffices to  prove the result for  $ {}^{r\!}G_\nu^{(\lambda)}(x).$
For notational simplicity, we denote
\begin{equation}\label{notationGMH}
\begin{split}
& G(x):={}^{r\!}G_\nu^{(\lambda)}(x); \;\;  M(x):= \omega_{\lambda}(x) G(x); \;\;   H(x):=M^2(x)+ {\varrho^{-1}}(1-x^2)\big(M'(x)\big)^2,
\end{split}
\end{equation}
where the constant $ \varrho:=2\lambda-1+\nu(\nu+2\lambda).$

\vskip 3pt
We take three steps to complete the proof.
\vskip 3pt
\underline{\sc Step 1}: ~ Show that  $H(x)$ is continuous on $[-1,1],$ that is, $H(\pm 1)$ are well defined.
 It is evident that by \eqref{obsvers},  $M(1)=0;$ and  from  \eqref{defBS0B}, we find that  $M(-1)$ is a finite value,   when $\lambda\ge 1.$
Next, from \eqref{dFCI++} with $s=1,$  we derive
\begin{equation}\label{M1x2}
(1-x^2)^{1/2} M'(x)
  =(1-2\lambda)\, (1-x^2)^{\lambda-1} \, {}^{r\!}G_{\nu+1}^{(\lambda-1)}(x).
\end{equation}
 Similarly, by  \eqref{obsvers}, $(1-x^2)^{1/2} M'(x)|_{x=1}=0$ for $\lambda>1,$ and it's finite for $\lambda=1.$  We now justify
 $(1-x^2)^{1/2} M'(x)|_{x=-1}$ is also well defined.  We infer from  Proposition \ref{nonvember}   that (a)  if $1\le \lambda <3/2,$
 ${}^{r\!}G_{\nu+1}^{(\lambda-1)}(x)$ is finite at  $x=-1;$  (b) if $\lambda=3/2,$    ${}^{r\!}G_{\nu+1}^{(\lambda-1)}(-1)=0$;  and (c) if $\lambda>3/2,$     ${}^{r\!}G_{\nu+1}^{(\lambda-1)}(x)$ tends to a finite value as $x\to -1.$ Hence, by \eqref{notationGMH}, $H(\pm 1)$ are well defined.

\vskip 3pt
\underline{\sc  Step 2}: ~ Derive the identity:
\begin{equation}\label{LBoundGeg007}
	H'(x)= -\frac{4(\lambda-1)x}{\varrho}\big(M'(x)\big)^2,\quad x\in (-1,1).
	 \end{equation}
Indeed,  taking $a=-\nu, b=\nu+2\lambda,c=\lambda +1/2$ and $z=(1\pm x)/2$ in  \eqref{SLPHF},  we find that
$G(x)$ satisfies the Sturm-Liouville problem
\begin{equation}\label{SLPGGF-F}
\begin{split}
\big\{\omega_{\lambda+1}(x) G'(x)\big\}'+\nu(\nu+2\lambda)\omega_{\lambda}(x) G(x)=0.
\end{split}
\end{equation}
Substituting $G(x)=\omega^{-1}_\lambda(x) M(x)$ into \eqref{SLPGGF-F},
we obtain from  a direct calculation that
	\begin{equation}\label{LBoundGeg-1}
	(1-x^2)M''(x)+(2\lambda -3)x M'(x)+\varrho M(x)=0.
	\end{equation}
Differentiating $H(x)$ and using \eqref{LBoundGeg-1}, leads to
	\begin{equation}\label{LBoundGeg-3}
	\begin{split}
	H'(x) &=  \frac 2 \varrho M' (x)\big\{(1-x^2)M''(x) +\varrho M(x)\big\}-\frac {2x}\varrho (M'(x))^2= -\frac{4(\lambda-1)x}{\varrho}\big(M'(x)\big)^2.
	 \end{split}
	\end{equation}
	
	\vskip 3pt
\underline{\sc Step 3}: ~ Prove the following bounds  and calculate the values at $x=0:$
\begin{equation}\label{LBoundGeg-40}
M^2(x)\le	H(x)\le 	H(0)=M^2(0)+ {\varrho^{-1}}\big(M'(0)\big)^2,\quad \forall\, x\in [-1,1].
\end{equation}
	
  By \eqref{LBoundGeg007},    we have  $H'(x)\equiv 0,$  if $\lambda=1, $ so $H(x)$ is a constant and $H(x)=H(0).$ In other words,   \eqref{LBoundGeg-40}  is true  for $\lambda=1.$

If $\lambda> 1$,  we deduce from  \eqref{LBoundGeg007} that the stationary points of $H(x)$ are $x=0$ or zeros of $M'(x)$ (if any). Let $0\not =\tilde x\in (-1,1)$ be any zero of $M'(x)$ (note:  $M(\tilde x)\not =0$).  Apparently,  by  \eqref{LBoundGeg007},    $H'(x)$ does not change sign in the neighbourhood of  $\tilde x,$ which means $\tilde x$ cannot be an extreme point of $H(x).$  In fact,  $x=0$ is the only extreme point  in $(-1,1).$
We also see from \eqref{LBoundGeg007} that $H'(x)\ge 0$ (resp. $H'(x)\le 0$) as $x\to 1^-$ (resp. $x\to -1^+$).
Note that
 $H(x)$  attains its maximum at $x=0$, as  $H(x)$ is ascending  when $x<0,$ and is descending when $x>0.$
 Therefore,  we obtain \eqref{LBoundGeg-40} from \eqref{notationGMH} and the above reasoning.

	Now, we calculate $H(0).$  From  \eqref{nonitA} and \eqref{Fatzero},  we obtain  that for $\lambda\ge 0,$
	\begin{equation}\label{LBoundGeg-5}
	\begin{split}
	M(0)=&{}^{r\!}G_\nu^{(\lambda)}(0)=
	{}_2F_1\Big(\!\!-\nu, \nu+2\lambda;\lambda+\frac 1 2;\frac{1} 2\Big)=\frac{\sqrt{\pi}\, \Gamma(\lambda+1/2)}{\Gamma(- {\nu}/2+
	1/2)\Gamma({\nu}/2+\lambda+ 1/2)}\\
		=&\sin \big(\pi ( {\nu} +1)/2\big)\frac{\Gamma(\lambda+1/2)\Gamma( {\nu}/2+1/2)}{\sqrt{\pi}\, \Gamma( {\nu}/2+\lambda+
		1/2)},
	\end{split}\end{equation}
which, together with  \eqref{dFCI--}, implies
	\begin{equation}\label{LBoundGeg-6}
	\begin{split}
&\big\{(1-x^2)^{1/2}	M'(x)\big\}|_{x=0}=
(1-2\lambda){}^{r\!}G_{\nu+1}^{(\lambda-1)}(0)\\
&=(1-2\lambda)\sin \big(\pi ( {\nu} +2)/2\big)\frac{\Gamma(\lambda- 1/2)\Gamma({\nu}/2+1)}{\sqrt{\pi}\, \Gamma({\nu}/2+\lambda)}=2\sin \big(\pi  {\nu} /2\big)\frac{\Gamma(\lambda+1/2)\Gamma({\nu}/2+1)}{\sqrt{\pi}\, \Gamma( {\nu}/2+\lambda)}.
\end{split}	\end{equation}
In the last step, we used the identity: $\Gamma(z+1)=z\Gamma(z).$

Substituting  \eqref{LBoundGeg-5}-\eqref{LBoundGeg-6} into  \eqref{LBoundGeg-40},  we obtain the bound in \eqref{LBoundGeg}.
\end{proof}

As a direct consequence of  Theorem \ref{LBoundForGegPoly}, we have the following bound for the Gegenbauer polynomials.
\begin{cor}\label{intcaseI} For real $\lambda\ge 1$ and integer   $l\ge 0$,  we have
\begin{subequations}
\begin{equation}\label{polynomialBnd}
\max_{|x|\le 1} \big\{\omega_\lambda(x) |G_{2l}^{(\lambda)}(x)|\big\}\le \frac{\Gamma(\lambda+1/2)\Gamma(l+1/2)}{\sqrt \pi\, \Gamma(l+\lambda+1/2)};
\end{equation}
\begin{equation}\label{polynomialBnd2}
\max_{|x|\le 1} \big\{\omega_\lambda(x) |G_{2l+1}^{(\lambda)}(x)|\big\}\le \frac{2l+1} {\sqrt{2\lambda-1+(2l+1)(2l+2\lambda+1)}}\,  \frac{\Gamma(\lambda+1/2)\Gamma(l+1/2)}{\sqrt \pi\, \Gamma(l+\lambda+1/2)}.
\end{equation}
\end{subequations}
\end{cor}
\begin{rem}\label{nmfdmew} {\em The bounds for Gegenbauer polynomials multiplying by a different weight function: $(1-x^2)^{\lambda/2-1/4}$ can be found in  \cite{Nevai1994SIMA}.  To the best of our knowledge, the bounds herein are new. }
\end{rem}

 The bound  in  Theorem \ref{LBoundForGegPoly} is valid for $\lambda\ge 1.$ In the analysis,
 we also need to  use  the bound  with   $0< \lambda < 1$.  Note that in this case, we have to multiply  the GGF-Fs by  a different weight function, and conduct the analysis in a slightly  different manner.
\begin{theorem}\label{BoundGegB}
	For real  $0< \lambda < 1$ and real $\nu\ge 0$, we have
	\begin{equation}\label{BoundGegB-0}
	\max_{|x|\le 1}\big\{	(1-x^2)^{\lambda/2} |{}^{r\!}G_\nu^{(\lambda)}(x)|,	\; (1-x^2)^{\lambda/2} | {}^{l}G_\nu^{(\lambda)}(x)| \big\}\le
	\widehat	\kappa_\nu^{(\lambda)},
	\end{equation}
	where
	\begin{equation}\label{kappaB}
	\begin{split}
	\widehat\kappa_\nu^{(\lambda)}=\frac{\Gamma(\lambda+1/2)}{\sqrt \pi}\bigg(
	\frac{ \cos^2(\pi\nu/ 2)\Gamma^2(\nu/2+1/2) }{\Gamma^2( ({\nu}+1)/2+\lambda)}+\frac{4\sin ^2\big(\pi  {\nu} /2\big)}{\nu^2+2\lambda \nu+\lambda}\frac{\Gamma^2(\nu/2+1)}{\Gamma^2({\nu}/2+\lambda)}\bigg)^{1/2}.
	\end{split}	\end{equation}
\end{theorem}
\begin{proof}
	Once again, thanks to  \eqref{obsvers},  it suffices to  prove the result for  $ {}^{r\!}G_\nu^{(\lambda)}(x).$ Here,
	we denote
	\begin{equation}\label{notationGMHB}
	\begin{split}
	&  \widehat M(x):= (1-x^2)^{\lambda/2} {}^{r\!}G_\nu^{(\lambda)}(x); \quad \widehat H(x):=\widehat M^2(x)+ \frac 1 {\rho(x)}\big(\widehat M'(x)\big)^2,\\
	& \rho(x):=\big((\nu+\lambda)^2(1-x^2)-\lambda (\lambda-1)\big)(1-x^2)^{-2}.
	\end{split}
	\end{equation}
Using Proposition \ref{nonvember}, we can justify as with  Step 1 in the proof of  Theorem \ref{LBoundForGegPoly} that  $\widehat H(x)$ is continuous on $[-1,1].$  A direct calculation from   \eqref{SLPGGF-F} leads to
	\begin{equation}\label{BoundGegB-1}
	(1-x^2)\widehat M''(x)-x \widehat M'(x)+(1-x^2)\rho(x) \widehat M(x)=0,  \quad x\in (-1,1).
	\end{equation}
Like  \eqref{LBoundGeg007}, we can  show
	\begin{equation}\label{BoundGegB-2}
	\begin{split}
	\widehat H'(x)
	% &=  2M'(x)M(x)-\frac{\varrho'(x)}{\varrho^2(x)}(M'(x))^2+\frac{2}{\varrho(x)}M'(x)M''(x)\\
	%&= \frac{2x(1-x^2)^{-1}\varrho(x)-\varrho'(x)}{\varrho^2(x)}\big(M'(x)\big)^2\\
	 &=\frac{2\lambda(\lambda-1)x}{(\lambda+\nu)^2(1-x^2)^2-\lambda(\lambda-1)(1-x^2)}\big(\widehat M'(x)\big)^2,\quad x\in (-1,1).
	\end{split}
	\end{equation}
For $0<\lambda<1,$     $	\widehat H(x)$ is increasing for $x<0,$ and decreasing  for $x>0,$  so  $H(x)$ attains its maximum at $x=0.$ Thus,
	\begin{equation}\label{BoundGegB-3}
	\widehat M ^2(x)\le\widehat	H(x)\le \widehat	H(0)=\widehat M^2(0)+ {\rho^{-1}(0)}\big(\widehat M'(0)\big)^2,\quad \forall\, x\in [-1,1].
	\end{equation}
	By \eqref{LBoundGeg-5},
	\begin{equation}\label{BoundGegB-4}
	\begin{split}
	\widehat M(0)
	=&\cos \big(\pi {\nu}/2\big)\frac{\Gamma(\lambda+1/2)\Gamma( {\nu}/2+1/2)}{\sqrt{\pi}\, \Gamma( {\nu}/2+\lambda+
		1/2)}.
	\end{split}\end{equation}
	Recall the identity (cf. \cite[(15.5.1)]{Olver2010Book}):
	\begin{equation}\label{BoundGegB-5}\frac{d}{dx} {}_2F_1(a, b;c;z)=\frac{ab}{c} {}_2F_1(a+1, b+1;c+1;z).
	\end{equation}
From  \eqref{rgjfdef} and  \eqref{BoundGegB-5}, we obtain
	\begin{equation}\label{BoundGegB-6}
	\begin{split}
	\frac{d}{dx}{}^{r\!}G_\nu^{(\lambda)}(x)
	=&\frac{d}{dx} {}_2F_1\Big(\!\!-\nu, \nu+2\lambda;\lambda+\frac 1 2;\frac{1-x} 2\Big)\\
	=&	\frac{\nu(\nu+2\lambda)}{2\lambda+1} {}_2F_1\Big(\!\!-\nu+1, \nu+2\lambda+1;\lambda+\frac 3 2;\frac{1-x} 2\Big).
	\end{split}\end{equation}
%If $\nu> 1/2$, using \eqref{BoundGegB-6}, we have
%	\begin{equation}\label{BoundGegB-6+}
%	\begin{split}
%	\frac{d}{dx}{}^{r\!}G_\nu^{(\lambda)}(x)
%	=	\frac{\nu(\nu+2\lambda)}{2\lambda+1}{}^{r\!}G_{\nu-1}^{(\lambda+1)}(x).
%	\end{split}\end{equation}
	Thanks to
	$$\widehat M'(x)= -\lambda x(1-x^2)^{\lambda/2-1}\, {}^{r\!}G_\nu^{(\lambda)}(x)+(1-x^2)^{\lambda/2}\frac{d}{dx}{}^{r\!}G_\nu^{(\lambda)}(x),$$
	we deduce from \eqref{nonitA}, \eqref{Fatzero} and  \eqref{BoundGegB-6} that
	\begin{equation}\label{BoundGegB-7}
	\begin{split}
	\{\rho^{-1/2}(x)\widehat M'(x)\}|_{x=0}
	=&\frac{2\sin(\pi\nu/2 )}{\sqrt{\pi}}\frac{\Gamma(\nu/2+1)\,\Gamma(\lambda+1/2)}{\sqrt{\nu^2+2\lambda \nu+\lambda}\,\, \Gamma(\lambda+\nu/2)}.
	\end{split}\end{equation}
	From  \eqref{BoundGegB-3}-\eqref{BoundGegB-4} and \eqref{BoundGegB-7}, we derive
	\eqref{BoundGegB-0}-\eqref{kappaB}.	
\end{proof}

\section{Fractional integral/derivative formulas of GGF-Fs}\label{sect:fintderG}
\setcounter{equation}{0}
\setcounter{lmm}{0}
\setcounter{thm}{0}

In this section, we show that  GGF-Fs enjoy some remarkable  fractional  calculus properties,
which are  important pieces of the puzzle for the analysis. % and are useful for spectral algorithms.
%Before pre we recall the definitions of RL fractional integrals/derivatives.
 %, and introduce the related spaces of functions.

%\subsection{Spaces of functions}
\subsection{Fractional  integrals/derivatives and related  spaces  of functions}
 Let $\Omega=(a,b)\subset {\mathbb R}$ be a finite open interval. For  real $p\in [1, \infty],$ let $L^p(\Omega)$ (resp. $W^{m,p}(\Omega)$ with $m\in {\mathbb N}$)
 be the usual
$p$-Lebesgue space (resp. Sobolev space), equipped  with the norm $\|\cdot\|_{L^p(\Omega)}$ (resp. $\|\cdot\|_{W^{m,p}(\Omega)}$).
 Let $C(\bar \Omega)$ be the classical space of continuous functions on $[a,b].$
%and $C^{0,\mu}(\bar\Omega)$ with $\mu>0,$ be the space of H\"older continuous functions of order $\mu,$ that is,  for any $f\in C^{0,\mu}(\bar\Omega),$ there exists a positive constant $c,$ such that
%$$
%|f(x)-f(y)|\le c|x-y|^{\mu},\quad \forall\, x,y\in [a,b].
%$$
%If $\mu=1,$ the function is also Lipschitz continuous. If $\mu>0,$ $f(x)$ is continuous.  However,  if $\mu=0,$ it may not be continuous, but bounded. On the other hand,  if $\mu>1,$ $f(x)$ must be a constant.
Denote by ${\rm AC}(\Omega)$
the space of {\em absolutely continuous} functions on $[a,b].$
Recall that (cf. \cite{Samko1993Book,Leoni2009Book}):  {\em $f(x)\in {\rm AC}(\Omega)$ if and only if $f(x)\in L^1(\Omega)$, $f(x)$ has a derivative $f'(x)$
almost everywhere on $[a,b]$ such that   $f'(x)\in L^1(\Omega), $      and $f$ has the integral representation:}
\begin{equation}\label{AVint}
f(x)=f(a)+\int_a^x f'(t)\,dt,\quad \forall x\in [a,b].
\end{equation}
Note that    ${\rm AC}(\Omega)=W^{1,1}(\Omega)$ (cf. \!\cite[Sec.\! 7.2]{Leoni2009Book}).
%In general, let us denote by ${\rm AC}^k(\Omega)$ with $k\in {\mathbb N},$ the space of functions $f(x)$ which have continuous derivatives up to order $k-1$ on $\Omega$ with $f^{(k-1)}\in {\rm AC}(\Omega)$ (cf.  \cite[P. 3]{Samko1993Book}).   Note that ${\rm AC}^1(\Omega)={\rm AC}(\Omega).$
Let ${\rm BV}(\Omega)$ be the space of functions of bounded  variation on $\Omega.$ It is known that {\em every function in  ${\rm BV}(\Omega)$ has at most a countable number of discontinuities, which are  either jump or removable discontinuities,  % (see,  e.g., \cite[Thm 2.8]{Wheeden2015Book}).
so it is differentiable almost everywhere}. As such, there hold
\begin{equation}\label{concluA}
%C(\bar \Omega)\subsetneq
W^{1,1}(\Omega)={\rm AC}(\Omega)\subsetneq {\rm BV}(\Omega) \subsetneq L^1(\Omega).
\end{equation}

%\begin{rem}\label{ACW121}
%It is noteworthy that in the above one-dimensional setting,   there hold the inclusions:
%\begin{equation*}\label{inclusionA}
%\begin{split}
% &\text{Continuously Differentiable $\subseteq$ Lipschitz Continuous $\subseteq$ Absolutely Continuous}\\
% &\text{ $\subseteq$  Bounded Variation $\subseteq$  Differentiable almost everywhere}.  \qquad  \qquad \qquad   \qed
% \end{split}
%\end{equation*}
% \end{rem}

% \begin{rem}\label{ACW11}
% We have ${\rm AC}(\Omega)=W^{1,1}(\Omega)$ (cf. \cite[P. 233]{Leoni2009Book}).
%It is known that $W^{1,1}(\Omega) \subseteq C(\bar \Omega),$ and for any $f\in
%W^{1,1}(\Omega),$
%\begin{equation}\label{socontinus}
%\|f\|_{L^\infty(\Omega)}\le |f(y)|+\|f'\|_{L^1(\Omega)},\quad \forall\, y\in [a,b].
%\end{equation}
%Hereafter, we shall use only the {\rm AC}-notation.  \qed
%\end{rem}

In what follows,
we denote the ordinary derivatives by  $D=d/dx$ and  $D^k=d^k/dx^k $ with integer  $k\ge 2$.
Recall the definitions of  the  RL fractional integrals and derivatives (cf. \cite[P. 33, P. 44]{Samko1993Book}).
\begin{defn}\label{RLFintder} {\em
 For any   $u\in L^1(\Omega),$  the left-sided and right-sided  RL  fractional integrals of order $s \in {\mathbb R}^+$ are defined  by
	\begin{equation}\label{leftintRL}
	\begin{split}
	&{}_{a} I_{x}^s\, u (x)=\frac 1 {\Gamma(s)}\int_{a}^x \frac{u(y)}{(x-y)^{1-s}} dy; \quad
	%\quad  a<x<b,
	{}_{x}I_{b}^s\, u (x)=\frac 1 {\Gamma(s)}\int_{x}^b \frac{u(y)}{(y-x)^{1-s}} dy, \;\;\; x\in \Omega.
	\end{split}
	\end{equation}
	A function $u\in L^1(\Omega)$ is said to possess  a   left-sided {\rm(}resp. right-sided\,{\rm)} RL fractional derivative
 ${}_{a}^{R} D_{x}^s u$ {\rm (}resp. ${}_{x}^{R} D_{b}^s u${\rm)}  of
order $s \in (0,1),$ if  ${}_{a} I_{x}^{1-s}\, u\in {\rm AC}(\Omega) $ {\rm(}resp. ${}_{x} I_{b}^{1-s}\, u\in {\rm AC}(\Omega)${\rm)}.  Moreover,   we have
\begin{equation}\label{reprezentacja20s01}
{}_{a}^{R} D_{x}^{s} u= D\big\{{}_{a} I_{x}^{1-s}\, u\big\},\quad {}_{x}^{R} D_{b}^{s} u= -D\big\{{}_{x} I_{b}^{1-s}\, u\big\},\quad x\in \Omega.
\end{equation}
   Similarly, for $s\in [k-1,k)$ with $k\in {\mathbb N},$ the  higher order  left-sided and right-sided  RL fractional derivatives for $u\in L^1(\Omega)$
   satisfying  ${}_{a} I_{x}^{1-s}\, u, {}_{x} I_{b}^{1-s}\, u \in {\rm AC}^k(\Omega)$ {\rm(}i.e., the space of all $f(x)$ having continuous derivatives up to order $k-1$ on $\Omega$ and $f^{(k-1)}\in {\rm AC}(\Omega)${\rm)} are defined  by
	\begin{equation}\label{left-fra-der-rl}
{}_{a}^{R} D_{x}^s\, u=D^k\big\{\, {}_{a}I_{x}^{k-s}\, u\big\};\quad
{}_x^R D_{b}^s\, u(x)=(-1)^kD^k\big\{\,{}_xI_{b}^{k-s}\, u\big\}.
	\end{equation} }
	\end{defn}
%	Note that  ${}_{a} I_{x}^s, 	{}_{x}I_{b}^s$ are linear continuous operators from $L^p(\Omega)$ to $L^p(\Omega)$ wit $1\le p\le \infty$ (cf.
%	\cite[P. 48]{Samko1993Book}). The fractional integration improves the regularity of the function, and    the range of the RL fractional integral operators  resides in a subspace of $L^p(\Omega).$
%\begin{prop}\label{indfrata} Let  the triple-$(s,p,q)$ be real numbers satisfying  $s\in (0,1)$ and $p,q\in [1,\infty].$ %Then for any $u\in L^p(\Omega),$ we have
%\begin{itemize}
%\item[(i)] ${}_{a} I_{x}^s, 	{}_{x}I_{b}^s$ are linear continuous operators from $L^p(\Omega)$ to $L^q(\Omega)$ for every
%    \begin{equation}\label{pqcondA}
%    1\le p<\frac 1 s,\quad 1\le q< \frac p{1-sp};
%    \end{equation}
%   or   for $p=1/s,$  and every  $1\le q<\infty.$
% \item[(ii)] For every $p> 1/s, $  ${}_{a} I_{x}^s, {}_{x}I_{b}^s$ are linear continuous operators from $L^p(\Omega)$ to $C^{0,s-1/p}(\bar \Omega).$
% \item[(iii)] For $p=\infty,$   ${}_{a} I_{x}^s, {}_{x}I_{b}^s$ are linear continuous operators from $L^\infty(\Omega)$ to $C^{0,s}(\bar \Omega).$
%\end{itemize}
%\end{prop}
% We refer to  \cite[P. 56, P. 66-67, P. 91]{Samko1993Book} for the proofs of the following properties, and also see
%\cite[Prop. 3]{Bourdin2015ADE} and \cite[Prop. 2.2]{Berg2016} for a summary of these claims.

\vskip 4pt

As a generalisation of  \eqref{AVint},  we have the following fractional integral representation,   which can also be regarded as the   definition of RL fractional derivatives alternative to Definition \ref{RLFintder} (see \cite[Prop.\!  3]{Bourdin2015ADE} and    \cite[P. 45]{Samko1993Book}).
\begin{prop}\label{fractrep}
 A function $u\in L^1(\Omega)$  possesses  a left-sided RL fractional derivative
 ${}_{a}^{R} D_{x}^s u$  of
order $s \in (0,1),$ if and only if there exist  $C_a\in {\mathbb R}$ and  $\phi \in L^1(\Omega)$ such that
\begin{equation}\label{reprezentacja}
u(x)= \frac{C_a}{\Gamma(s)}(x-a)^{s-1} + {}_{a} I_{x}^s\,\phi(x) \;\;  \text{a.e. on}\;\;  [a,b],
\end{equation}
where $C_a=({}_{a} I_{x}^{1-s}u)(a)$ and $\phi(x) = {}_{a}^{R} D_{x}^s\, u(x)$ a.e. on $ [a,b].$

Similarly, a function
$u\in L^1(\Omega)$  has a right-sided RL fractional derivative
 ${}_{x}^{R} D_{b}^s u$  of
order $s \in (0,1),$ if and only if there exist  $C_b\in {\mathbb R}$ and  $\psi\in L^1(\Omega)$ such that
\begin{equation}\label{reprezentacja2}
u(x)= \frac{C_b}{\Gamma(s)}(b-x)^{s-1} + {}_{x} I_{b}^s\,\psi(x) \;\;   \text{a.e. on}\;\;  [a,b],
\end{equation}
where $C_b=({}_{x} I_{b}^{1-s}u)(b)$ and $\psi(x) = {}_{x}^{R} D_{b}^s u(x) $ a.e. on $ [a,b].$
\end{prop}

\begin{rem}\label{equivlenceresult} {\em We infer from Proposition \ref{fractrep} the equivalence of these two fractional spaces:
\begin{equation}\label{Ws1Ber007}
 W_{\!{\rm RL},a+}^{s,1}(\Omega):=\big\{u\in L^1(\Omega) :  {}_{a} I_{x}^{1-s} u\in {\rm AC}(\Omega)\big\}\equiv
\big\{u\in L^1(\Omega) :  {}_{a}^{R} D_{x}^s u \in L^1(\Omega)\big\},
\end{equation}
for $s\in (0,1).$  The inclusion  ``\,$\subseteq$\," follows immediately from  $u\in  L^1(\Omega), {}_{a} I_{x}^{1-s} u\in  {\rm AC}(\Omega)$ and Definition \ref{RLFintder}.
To show the opposite inclusion ``\,$\supseteq$\,", we find
\begin{equation*}
\begin{split}
 \int^b_a |{}_{a} I_{x}^{1-s} u|dx=&\frac{1}{\Gamma(1-s)} \int^b_a \Big|\int_ a^x (x-y)^{-s}u(y)dy\Big|dx \le \frac{1}{\Gamma(1-s)} \int^b_a \int_ a^x (x-y)^{-s}|u(y)|dydx \\
 =&\frac{1}{\Gamma(1-s)} \int^b_a \Big(\int_ y^b (x-y)^{-s}dx\Big) |u(y)|dy=\frac{1}{\Gamma(2-s)} \int^b_a  (b-y)^{1-s} |u(y)|dy\\
 \le &\frac{(b-a)^{1-s}}{\Gamma(2-s)} \int^b_a  |u(y)|dy\,.
 \end{split}
 \end{equation*}
   Since $u\in L^1(\Omega)$, we conclude $ {}_{a} I_{x}^{1-s} u \in  L^1(\Omega)$. As ${}_{a}^{R} D_{x}^s u =D \{{}_{a} I_{x}^{1-s} u\}  \in L^1(\Omega)$, we infer that  ${}_{a} I_{x}^{1-s} u\in W^{1,1}(\Omega)(\equiv{\rm AC}(\Omega))$.    Therefore,  the equivalence in \eqref{Ws1Ber007} follows. The same property for $ W_{\!{\rm RL},b-}^{s,1}(\Omega)$ with    ${}_{x} I_{b}^{1-s} u$, ${}_{x}^{R} D_{b}^s u$ in place of ${}_{a} I_{x}^{1-s} u,  {}_{a}^{R} D_{x}^s u,$ respectively, holds.   We  refer to   \cite{Berg2016} for insightful discussions of  the relation between $ W_{\!{\rm RL},a+}^{s,1}(\Omega)$ and the fractional  Sobolev space
in the sense of  Gagliardo \cite{Nezza2012BSM}. }
\end{rem}

%\begin{rem}\label{expformA}
Recall the explicit formulas  $($cf.  \cite{Samko1993Book}$)$:   for  real $\eta>-1$ and $s>0,$
\begin{equation}\label{intformu}
{}_a I_{x}^s \, (x-a)^\eta=  \dfrac{\Gamma(\eta+1)}{\Gamma(\eta+s+1)} (x-a)^{\eta+s};\quad
 {}_{a}^{R} D_{x}^s  \, (x-a)^\eta=\frac{\Gamma(\eta+1)}{\Gamma(\eta-s+1)} (x-a)^{\eta-s}.
\end{equation}
We have similar formulas for right-sided RL fractional integral/derivative of $(b-x)^\eta.$ In particular,
\begin{equation}\label{intformu007}
{}_a I_{x}^{1-s} \, (x-a)^{s-1}=  \Gamma(s);\quad
{}_x I_{b}^{1-s} \, (b-x)^{s-1}=  \Gamma(s),\;\;\; s\in (0,1),
\end{equation}
which implies the boundary values $C_a$ and $C_b$ in Proposition \ref{fractrep} are not always zero as $x\to a^+$ and $x\to b^-,$ respectively.
On the other hand,   if $\eta-s+1=-n$ with $n\in {\mathbb N}_0$ in the second formula of \eqref{intformu} (note: $\Gamma(-n)=\infty$),  then
\begin{equation}\label{0intformu}
 {}_{a}^{R} D_{x}^s  \, (x-a)^{s-n-1}= {}_{x}^{R} D_{b}^s  \, (b-x)^{s-n-1}=0,\quad {\rm for}\;\; s>n\in {\mathbb N}_0.
\end{equation}
We see that the first term in the integral representations in  \eqref{reprezentacja}-\eqref{reprezentacja2} actually  plays the same role as a ``constant"  in \eqref{AVint}.
 %\qed\end{rem}

%{\begin{color}{blue}
% We now  introduce the RL fractional Sobolev-type spaces.  For $s\in (0,1)$ and $1\le p\le \infty,$ we define the fractional AC spaces
%	(cf. \cite[Definitions 3-4]{Bourdin2015ADE}):
%\begin{equation}\label{Ws1}
%\begin{split}
%& {\rm AC}_{a,+}^{s,p}(\Omega):=\left\{u\in L^1(\Omega) :  {}_{a}^{R} D_{x}^s u\in L^p(\Omega)\right\};\\
%& {\rm AC}_{b,-}^{s,p}(\Omega):=\left\{u\in L^1(\Omega) :  {}_{x}^{R} D_{b}^s u\in L^p(\Omega)\right\}.
%\end{split}
%\end{equation}
%\begin{rem}\label{relaWs1}
%It is evident that if $s=0,1$ and $p=1,$ they  reduce to $L^1(\Omega)$ and ${\rm AC}(\Omega),$  respectively.
%In fact,  ${\rm AC}_{a,+}^{s,1}(\Omega)$ and $ {\rm AC}_{b,-}^{s,1}(\Omega)$ are fractional spaces in-between $L^1(\Omega)$ and ${\rm AC}(\Omega),$  which can be verified readily from Definition \ref{RLFintder} and Proposition \ref{indfrata}.
%\qed
%\end{rem}
%
%\begin{rem}\label{relaWs1jcbi}  It was also  shown in   \cite[Prop. 3.1]{Berg2016} that $W_{a,+}^{s,1}(\Omega)\subset  {\rm AC}_{a,+}^{s,1}(\Omega)$ and  $W_{b,-}^{s,1}(\Omega)\subset {\rm AC}_{b,-}^{s,1}(\Omega).$ In fact, thanks to  Proposition \ref{fractrep} (i.e., \cite[Prop. 3]{Bourdin2015ADE}), we have
%\begin{equation}\label{iinclusional}
%W_{a,+}^{s,1}(\Omega)={\rm AC}_{a,+}^{s,1}(\Omega),\quad  W_{b,-}^{s,1}(\Omega)={\rm AC}_{b,-}^{s,1}(\Omega), \;\;\; s\in (0,1).
%\end{equation}
%{\bf Provide a simple proof of the other direction of inclsuion!} \qed
%\end{rem}

\subsection{Important formulas}
%The following formulas play an important role in the error analysis.
\begin{theorem}\label{FCI}
 For  real $   \nu\ge s>0$ and real  $\lambda>-1/2$, the GGF-Fs on $(-1,1)$ satisfy the RL fractional integral formulas:  %for $x\in(-1,1),$
 \vspace*{-4pt}
 \begin{subequations}
\begin{equation}\label{FCI++}
\begin{split}
 {}_{x} I_{1}^{s}\big\{\omega_\lambda(x)
  \,{}^{r\!}G_{\nu}^{(\lambda)}(x)\big\}
 =&h^{(-s)}_{\lambda}\, \omega_{\lambda+s}(x)\, {}^{r\!}G_{\nu-s}^{(\lambda+s)}(x),\\
 \end{split}
\end{equation}
 \begin{equation}\label{FCI--}
\begin{split}
 {}_{-1} I_{x}^{s}
 \big\{ \omega_\lambda(x)\, {}^{l}G_{\nu}^{(\lambda)}(x)\big\}
 =&\,(-1)^{ [ \nu]+[ \nu-s]}\, h^{(-s)}_{\lambda}\,
 \omega_{\lambda+s}(x)\,{}^{l}G_{\nu-s}^{(\lambda+s)}(x).
 \end{split}
\end{equation}
\end{subequations}
For  real $\lambda>s-1/2$ and real $\nu\ge 0,$ the GGF-Fs on $(-1,1)$ satisfy the RL fractional derivative  formulas:  % for $x\in(-1,1),$
\vspace*{-8pt}
\begin{subequations}
\begin{equation}\label{dFCI++}
\begin{split}
 {}_{x}^R D_{1}^{s}\big\{\omega_{\lambda}(x)
  \,{}^{r\!}G_{\nu}^{(\lambda)}(x)\big\}
 =&h^{(s)}_{\lambda}\, \omega_{\lambda-s}(x)\, {}^{r\!}G_{\nu+s}^{(\lambda-s)}(x),\\
 \end{split}
\end{equation}
 \begin{equation}\label{dFCI--}
\begin{split}
 {}_{-1}^{~~R} D_{x}^{s}
 \big\{ \omega_{\lambda}(x) \, {}^{l}G_{\nu}^{(\lambda)}(x)\big\}
 =&\,(-1)^{[ \nu]+[ \nu+s]}\, h^{(s)}_{\lambda}\,
 \omega_{\lambda-s}(x)\,{}^{l}G_{\nu+s}^{(\lambda-s)}(x).
 \end{split}
\end{equation}
\end{subequations}
In the above, we denote
\begin{equation}\label{alsoAhomega}
\omega_\alpha(x)=(1-x^2)^{\alpha-\frac 12},\quad h^{(\beta)}_{\lambda}=\frac{2^\beta \,\Gamma(\lambda+1/2)}{\Gamma(\lambda-\beta+1/2)}.
\end{equation}
\end{theorem}
\begin{proof}
%By (2.2.7) in  \cite[P. 68]{Andrews1999Book}
%\begin{align}
%\label{Euler}
%{}_2F_1(a,b;c;z) =(1-z)^{c-a-b}{}_2F_1(c-a,c-b;c;z),\hspace{1cm}\textup{(Euler)}
%\end{align}
Recall the Bateman's fractional integral formula (cf. \cite[P. 313]{Andrews1999Book}):  for $c,s>0$ and $|z|<1,$
\begin{align}
\label{Bateman}
{}_2F_1(a,b;c+s;z) = z^{1-(c+s)} \frac{\Gamma(c+s)}{\Gamma(c)\Gamma(s)}
\int_0^{z} t^{c-1} (z-t)^{s-1} {}_2F_1(a,b;c;t) \,dt,
\end{align}
which, together with \eqref{Euler},  yields
 \begin{equation}\label{BatemanB}
\begin{split}
z^{c+s-1}&(1-z)^{c+s-a-b}{}_2F_1(c-a+s,c-b+s;c+s;z)
\\
&= \frac{\Gamma(c+s)}{\Gamma(c)\Gamma(s)}
\int_0^{z} t^{c-1} (z-t)^{s-1} (1-t)^{c-a-b}{}_2F_1(c-a,c-b;c;t)\,dt.
\end{split}
\end{equation}
Applying the variable substitutions:  $z={(1-x)}/{2}$ and $t= {(1-y)}/{2} $
 to  \eqref{BatemanB},  leads to
% \begin{equation}\label{BatemanBV2}
%\begin{split}
%\Big(\frac{1+x}{2}\Big)^{c-1+s}&\Big(\frac{1-x}{2}\Big)^{c-a-b+s}{}_2F_1\Big(c-a+s,c-b+s;c+s;\frac{1+x}{2}\Big)
%\\
%&= \frac{\Gamma(c+s)}{2^s\Gamma(c)\Gamma(s)}
%\int_{-1}^{x} \big(\frac{1+y}{2})^{c-1}  \big(\frac{1-y}{2})^{c-a-b}(x-y)^{s-1}{}_2F_1\big(c-a,c-b;c;\frac{1+y}{2}\big)dy.
%\end{split}
%\end{equation}
 \begin{equation}\label{BatemanBV2}
\begin{split}
(1-x)&^{c+s-1}(1+x)^{c+s-a-b}{}_2F_1\Big(c-a+s,c-b+s;c+s;\frac{1-x}{2}\Big)
\\
&= \frac{2^s\, \Gamma(c+s)}{\Gamma(c)\Gamma(s)}
\int_{x}^{1} (1-y)^{c-1}  (1+y)^{c-a-b}(y-x)^{s-1}{}_2F_1\Big(c-a,c-b;c;\frac{1-y}{2}\Big)dy.
\end{split}
\end{equation}
Taking $a= \nu+\lambda+1/2$, $b= -\nu-\lambda+1/2$ and $c= \lambda+1/2$ in \eqref{BatemanBV2}, we obtain
 \begin{equation*}\label{BatemanBV2+}
\begin{split}
&(1-x^2)^{\lambda+s+1/2} \, {}_2F_1\Big(s-\nu,\nu+s+2\lambda;\lambda+s+\frac 1 2;\frac{1-x}{2}\Big)
\\
&= \frac{2^s\, \Gamma(\lambda+s+1/2)}{\Gamma(\lambda+1/2)\Gamma(s)}
\int_{x}^{1} (1-y^2)^{\lambda-1/2}(y-x)^{s-1}{}_2F_1\Big(\!\!-\nu,\nu+2\lambda;\lambda+\frac1 2;\frac{1-y}{2}\Big)dy.
\end{split}
\end{equation*}
From \eqref{leftintRL} and \eqref{rgjfdef}, we  derive \eqref{FCI++} immediately.

Similarly,  performing  the variable substitutions: $z={(1+x)}/{2}$ and $t= {(1+y)}/{2} $
 to  \eqref{BatemanB},   we can obtain \eqref{FCI--} in the same manner.

%By \cite[Thm 2.14]{Diet10},  we have that for any absolutely integrable function $v,$ and real $s \ge 0,$
%\begin{equation}\label{rulesa}
%   {}_{x}^R D_{1}^{s}\,  {}_{x}I_{1}^{s} v(x) = v(x), \quad {}_{-1}^{~~R} D_{x}^{s}\,  {}_{-1}I_{x}^{s} v(x) = v(x),\quad \text{a.e. in}\;\; (-1,1).
% \end{equation}
Applying ${}_{x}^R D_{1}^{s}$ to both sides of \eqref{FCI++} and noting that    ${}_{x}^R D_{1}^{s}\,  {}_{x}I_{1}^{s}$ is an identity operator   (cf. \cite{Samko1993Book}),  we  obtain  %from \eqref{rulesa} that
for real $   \nu\ge s>0$ and real  $\lambda>-1/2$,
\begin{equation}\label{dFCI+}
\begin{split}
\omega_{\lambda}(x)
  \,{}^{r\!}G_{\nu}^{(\lambda)}(x)
 =&h^{(-s)}_{\lambda}\, {}_{x}^R D_{1}^{s} \big\{\omega_{\lambda+s}(x)\, {}^{r\!}G_{\nu-s}^{(\lambda+s)}(x)\big\}.
 \end{split}
\end{equation}
Replacing  $\lambda, \nu$ in the above equation by $\lambda-s, \nu+s$, and noting
that
\begin{equation}\label{hmuest}
\big( h^{(-s)}_{\lambda-s}\big)^{-1}=\frac{2^s \,\Gamma(\lambda+1/2) }{\Gamma(\lambda-s+1/2)}=h^{(s)}_{\lambda},
\end{equation}
we obtain \eqref{dFCI++}.  Similarly, applying ${}_{-1}^{~~R} D_{x}^{s}$ to both sides of \eqref{FCI--}, we can derive \eqref{dFCI--}.
\end{proof}

%\begin{rem}\label{specialRLformulas}  The formulas in Theorem \ref{FCI} link up  GGF-Fs of  different orders and  parameters, which  can be very useful for developing efficient spectral algorithms for fractional differential equations. Indeed, corresponding to the special cases  in Proposition \ref{nonvemberA} and Remark \ref{spJacibGlam},  we can derive from \eqref{lmbda0} and \eqref{FCI++}   that the fractional operator ${}_{x}I_1^s$ takes $\{(1-x)^{\alpha} P_n^{(\alpha,-\alpha)}(x)\}$ to $\{(1-x)^{\alpha+s} P_n^{(\alpha+s,s-\alpha)}(x)\}.$ Indeed, such types of explicit formulas played an essential role in  spectral methods (see, e.g., \cite{zayernouri2013fractional,Chen.SW2016}). \qed
%\end{rem}

\vskip 10pt
	
\section{Chebyshev approximations of functions  in fractional Sobolev-type spaces}\label{mainsect:ms}
\setcounter{equation}{0}
\setcounter{lmm}{0}
\setcounter{thm}{0}
\setcounter{cor}{0}

In this section, we introduce a new  theoretical framework and present the main results on Chebyshev approximations.
 Here, we focus on the approximation of functions with interior singularities, and shall  extend the estimates to deal with functions with endpoint singularities in Subsection \ref{endptsingula}.

\subsection{Fractional Sobolev-type spaces}
For  a fixed   $\theta\in \Omega:=(-1,1), $ we denote $\Omega_\theta^-:=(-1,\theta)$ and  $\Omega_\theta^+:=(\theta,1).$  For   $m\in {\mathbb N}_0$ and $s\in (0,1),$
we define the fractional Sobolev-type space:
	\begin{equation}\label{fracA}
\begin{split}
{\mathbb  W}^{m+s}_{\theta}(\Omega):= \big\{ &u\in L^1(\Omega)\, :\,   u, u',\cdots, u^{(m-1)}\in {\rm AC}(\Omega)  \;\; {\rm and}\\
& {}_{x}I_{\theta}^{1-s} u^{(m)}\in  {\rm BV}(\Omega_\theta^- ),\quad {}_{\theta}I_{x}^{1-s} u^{(m)}\in  {\rm BV}(\Omega_\theta^+)\big\},
	\end{split}
\end{equation}
equipped with the norm (note:  ${\rm AC}(\Omega)=W^{1,1}(\Omega) $):
\begin{equation}\label{Wmsnorm}
\|u\|_{{\mathbb  W}^{m+s}_{\theta}(\Omega)}=\sum_{k=0}^m \|u^{(k)}\|_{L^1(\Omega)} +   U^{m,s}_\theta,
\end{equation}
where the semi-norm is defined by
\begin{itemize}
\item for $m=1,2,\cdots,$
\begin{equation}\label{seminormF}
  \begin{split}
  U^{m,s}_\theta:=&\int_{-1}^\theta  \big | {}_{x}^R D_{\theta}^{s} \, u^{(m)}(x) \big|\,  dx
  +\int^{1}_\theta    |{}_{\theta}^R D_{x}^{s}  u^{(m)}(x)| dx\\
  &+\big|\big\{ {}_{x}I_{\theta}^{1-s} u^{(m)}\big\}(\theta+)\big|+\big|\big\{ {}_{\theta}I_{x}^{1-s} u^{(m)}\big\}(\theta-)\big|;
  \end{split}
  \end{equation}
  \item for  $m=0$ and $s\in (1/2,1),$
      \begin{equation}\label{Pnorm00}
  \begin{split}
 U^{0,s}_\theta:= & \int_{-1}^\theta  \big | {}_{x}^R D_{\theta}^{s} \, u(x) \big|\,\omega_{s/2}(x)  dx
  +\int^{1}_\theta    |{}_{\theta}^R D_{x}^{s}  u(x)| \, \omega_{s/2}(x) dx\\
  &\quad\quad\quad+\big|\big\{\omega_{s/2}\, {}_{x}I_{\theta}^{1-s} u\big\}(\theta+)\big|+\big|\big\{\omega_{s/2}\, {}_{\theta}I_{x}^{1-s} u\big\}(\theta-)\big|.
  \end{split}
  \end{equation}
 \end{itemize}

 \begin{rem}\label{propremk} {\em  Some remarks are in order.
 \begin{itemize}
 \item[(i)] If $u\in {\mathbb  W}^{m+s}_{\theta}(\Omega),$  we infer from Proposition \ref{fractrep} and Remark \ref{equivlenceresult} that
${}_{\theta}^R D_{x}^{s}  u^{(m)}, {}_{x}^R D_{\theta}^{s} \, u^{(m)}$ are well-defined and belong to $L^1(\Omega).$
\item[(ii)]  The parameter $\theta$ is related to the location of the singular point of $u(x).$ For example, if $u=|x|,$ then $\theta=0.$
For a function of multiple interior  singular points, we partition $(-1,1)$ into multiple subintervals and introduce the same number of parameters accordingly.       %{\rm(}see Section \ref{Iusexample}{\rm)}.
\item[(iii)]  The so-defined space ${\mathbb  W}^{m+s}_{\theta}(\Omega)$  is an intermediate  fractional  space in the sense that
\begin{equation}\label{WWsps}
W^{1,1}(\Omega) \subsetneq  {\mathbb  W}^{s}_{\theta}(\Omega)  \subsetneq L^1(\Omega); \quad {\mathbb  W}^{m+1}(\Omega) \subsetneq  {\mathbb  W}^{m+s}_{\theta}(\Omega)  \subsetneq  {\mathbb  W}^{m}(\Omega),\;\; m\ge1.
\end{equation}
In particular, when $s\to 1^-,$  the space ${\mathbb  W}^{m+s}_{\theta}(\Omega)$ reduces to
\begin{equation}\label{intergerASobB}
\begin{split}
{\mathbb  W}^{m+1}(\Omega):= \big\{ &u\in L^1(\Omega)\, :\,    u',\cdots, u^{(m-1)}\in {\rm AC}(\Omega), \; u^{(m)}\in  {\rm BV}(\Omega)\big\},
	\end{split}
\end{equation}
which has been used in \cite{Trefethen2013Book,Majidian2017ANM} for Chebyshev approximation of functions with limited regularity.
 \end{itemize}
}
\end{rem}

To deal with endpoint singularities,  letting $\theta\to \pm 1,$ we denote the corresponding  fractional spaces by
\begin{equation}\label{resultpm1}
\begin{split}
& {\mathbb  W}^{m+s}_{+}(\Omega):= \big\{ u\in L^1(\Omega)\, :\,   u, u',\cdots, u^{(m-1)}\in {\rm AC}(\Omega),\;  {}_{x}I_{1}^{1-s} u^{(m)}\in  {\rm BV}(\Omega)\big\},\\
& {\mathbb  W}^{m+s}_{-}(\Omega):= \big\{ u\in L^1(\Omega)\, :\,   u, u',\cdots, u^{(m-1)}\in {\rm AC}(\Omega),\;  {}_{-1}I_{x}^{1-s} u^{(m)}\in  {\rm BV}(\Omega)\big\}.
\end{split}
\end{equation}
%and  ${\mathbb  W}^{m+s}_{-}(\Omega)$  is defined similarly  with  ${}_{-1}I_{x}^{1-s} u^{(m)}$ in place of ${}_{x}I_{1}^{1-s} u^{(m)}.$
Accordingly, the semi-norm
  $U^{m,s}_+$ (resp.    $U^{m,s}_-$)  only involves the right  (resp.  left) RL fractional integrals/derivatives.
We remark that for $s\in (0,1),$ ${\mathbb  W}^{s}_{-}(\Omega) \supsetneq  W_{\!{\rm RL},-1+}^{s,1}(\Omega)$  defined in \eqref{Ws1Ber007}.

%\begin{rem}\label{impremk}
%Bergounioux et al  \cite{Berg2016} introduced  the  fractional Sobolev spaces:
%\begin{equation*}\label{Ws1Ber}
%\begin{split}
%& W_{+}^{s,1}(\Omega):=\big\{u\in L^1(\Omega) :  {}_{x} I_{1}^{1-s} u\in W^{1,1}(\Omega)\big\};\;\; W_{-}^{s,1}(\Omega):=\big\{u\in L^1(\Omega) :  {}_{-1} I_{x}^{1-s} u\in W^{1,1}(\Omega)\big\},
%\end{split}
%\end{equation*}
%for $s\in (0,1),$ as an intermediate space in-between $L^1(\Omega)$ and $W^{1,1}(\Omega),$  where  their relation with the usual  fractional Sobolev spaces
%in the sense of  Gagliardo \cite{Nezza2012BSM} was discussed.   Note that as $W^{1,1}(\Omega)\subsetneq {\rm BV}(\Omega),$  we have
%$W_{\pm}^{s,1}(\Omega)\subsetneq {\mathbb  W}^{s}_{\pm}(\Omega)$ for $s\in (0,1).$ \qed
% \end{rem}

\subsection{Exact formulas  and decay rate of Chebyshev expansion coefficients}  Let $\omega(x)=(1-x^2)^{-1/2}=\omega_0(x)$ be the Chebyshev weight function.  For any  $ u\in L^2_{\omega}(\Omega)$, we expand it in Chebyshev series and denote the  partial sum  by
\begin{equation}\label{Cbexp}
u(x)=\sum_{n=0}^{\infty}{\!'}\,\hat u_n^C\, T_n(x),\quad \pi_N^C  u(x)=\sum_{n=0}^{N}{\!'}\,\hat u_n^C\, T_n(x),
\end{equation}
where the prime denotes a sum whose first  term is
halved, and
\begin{equation}\label{ancoef}
\hat u_n^C=\frac{2}{\pi}\int_{-1}^{1}{u(x)\frac{T_n(x)}{\sqrt{1-x^2}}}dx=\frac{2}{\pi}\int_{0}^{\pi}u(\cos \theta)\cos(n\theta)d\theta.
\end{equation}

 Recall the formula of integration by parts  involving
the Stieltjes integrals   (cf. \cite[(1.20)]{Klebaner2005Book}).
% As a special case, we have  the classical integration by parts formula  \eqref{IPPW11} below (cf. \cite[P. 89]{Leoni2009Book}).
\begin{lmm}\label{IntByPartsInW11}
For any  $u,v\in {\rm BV}(\Omega)$, we have
	\begin{align} \label{IPPW1122}
	\int_{a}^b u(x-)  \, dv(x) =\{u(x)v(x)\}\big|_{a+}^{b-} - \int_{a}^b  v(x-)\,  d u(x),
	\end{align}
where the notation $u(x\pm)$ stands for the right- and left-limit of $u$ at $x,$ respectively. Here, $u(x-), v(x-)$ can also be replaced by $u(x+), v(x+).$

In particular, if   $u,v\in {\rm AC}(\Omega)$, we have
	\begin{align} \label{IPPW11}
	\int_{a}^b u (x) v'(x) \, dx + \int_{a}^b  u'(x) v(x) \, dx  =\{u(x)v(x)\}\big|_{a}^b.
	\end{align}
\end{lmm}

  As highlighted  in \cite{Majidian2017ANM},   the error analysis of Chebyshev expansions in various norms, and the related interpolation and quadrature errors  essentially depends on estimating the decay rate  of $|\hat u_n^C|.$  We present the main results below.
\begin{theorem}\label{IdentityForUn} Given $\theta\in (-1,1),$ if  $u\in {\mathbb  W}^{m+s}_{\theta}(\Omega) $ with
$s\in (0,1)$ and integer  $m\ge 0$, then for $n\ge m+s>1/2,$
	\begin{equation}\label{HatUnCaseC}
	\begin{split}
	 \hat u_n^{C}&= \frac{1}{\sqrt{\pi}\,2^{m+s-1}\Gamma(m+s+1/2)}
	\bigg\{(-1)^{ n+[ n-s]}\int_{-1}^\theta   {}_{x}^R D_{\theta}^{s} \, u^{(m)}(x) \,  {}^{l}G_{n-m-s}^{(m+s)}(x) \, \omega_{m+s}(x)\, dx
\\&+ (-1)^{ n+[ n-s]}\big\{{}_{x}I_{\theta}^{1-s}u^{(m)}  (x) \, {}^{l}G_{n-m-s}^{(m+s)}(x)\, \omega_{m+s}(x)\big\}\big|_{x=\theta-}\\	&
+\int^{1}_\theta    {}_{\theta}^R D_{x}^{s}  u^{(m)}(x) \,  {}^{r\!}G_{n-m-s}^{(m+s)}(x) \, \omega_{m+s}(x)\, dx+\big\{ {}_{\theta}I_{x}^{1-s}u^{(m)}  (x)\, {}^{r\!}G_{n-m-s}^{(m+s)}(x)\, \omega_{m+s}(x) \big\}\big|_{x=\theta+}\bigg\},
	\end{split}
	\end{equation}
where  $\omega_{\lambda}(x)=(1-x^2)^{\lambda-1/2}.$

For  $n\ge m+s,$ we have  the following upper bounds:
	\begin{itemize}
	\item[(i)] If $m=0$ and  $s\in (1/2,1)$,   then we have
%\begin{equation}\label{BoundForUnRealA}
%\begin{split}
%|\hat u_n^{C}|\le & \frac{U_{0,s}}{2^{s-1} \pi}\max\bigg\{
%\frac{ \Gamma( ({n-s}+1)/2) }{\Gamma( ({n+s}+1)/2)},\frac{2}{\sqrt{n^2-s^2+s}}\frac{\Gamma((n-s)/2+1)}{\Gamma((n+s)/2)}\bigg\}.
%\end{split}
%\end{equation}
\begin{equation}\label{BoundForUnRealA}
\begin{split}
|\hat u_n^{C}|\le & \frac{U_\theta^{0,s}}{2^{s-1} \pi}\max\bigg\{
\frac{ \Gamma( ({n-s}+1)/2) }{\Gamma( ({n+s}+1)/2)},\frac{2}{\sqrt{n^2-s^2+s}}\frac{\Gamma((n-s)/2+1)}{\Gamma((n+s)/2)}\bigg\}.
\end{split}
\end{equation}
	\item[(ii)]If $m\ge 1$,  then we have
		\begin{equation}\label{BoundForUnRealB}
	\begin{split}
	|\hat u_n^{C}|\le & \frac{U_\theta^{m,s}}{2^{m+s-1} \pi}
	\frac{ \Gamma( ({n-m-s}+1)/ 2)}{\Gamma( ({n+m+s}+1)/2)} .
	\end{split}
	\end{equation}
\end{itemize}
\end{theorem}
\begin{proof}   Substituting
	$n\to n-k, \lambda\to k$ in   \eqref{RodriguesF}, leads to
		\begin{equation}\label{dFCI00}
		\begin{split}
		\omega_{k}&(x) G_{n-k}^{(k)}(x)=\, -\frac{1} {2k+1}\, \big\{ \omega_{k+1}(x) G_{n-k-1}^{(k+1)}(x) \big\} ',  \;\;\; n\ge k+1.
		\end{split}
		\end{equation}	
For $u,u',\cdots, u^{(m-1)}\in {\rm AC}(\Omega), $  using   \eqref{dFCI00} with $k=0, 1,\cdots, m-1,$ and the integration by parts  in Lemma \ref{IntByPartsInW11}, we obtain that  for  $n\ge m$,
	\begin{equation}\label{byparts0}
	\begin{split}
	\hat u_n^{C}&=\frac{2}{\pi}\int_{-1}^{1}u(x) G_n^{(0)}(x) \omega_{0}(x)\,  dx
	=-\frac{2}{\pi}\int_{-1}^{1}u(x) \big\{G_{n-1}^{(1)}(x) \omega_{1}(x) \big\}'\, dx\\
	&=\frac{2}{\pi}\int_{-1}^{1}u'(x)G_{n-1}^{(1)}(x) \omega_{1}(x) dx
	=-\frac{2}{3\pi} \int_{-1}^{1}u'(x)\big\{G_{n-2}^{(2)}(x) \omega_{2}(x)\big\}'\,dx
	\\
	&= \frac 1 3 \frac{2}{\pi} \int_{-1}^{1}u''(x)G_{n-2}^{(2)}(x) \omega_{2}(x) dx
	=-\frac 1 {3\cdot 5}  \frac{2}{\pi} \int_{-1}^{1}u''(x)\big\{G_{n-3}^{(3)}(x) \omega_{3}(x)\big\}'\,dx
	\\
	&= \cdots= \frac 1 {(2m-1)!!}  \frac 2\pi \int_{-1}^{1} u^{(m)}(x)\, G_{n-m}^{(m)}(x) \omega_{m}(x)\,dx.
	\end{split}
	\end{equation}
Using the  identity  (cf. \cite{Olver2010Book}):
	\begin{equation}\label{gammn12}
	\Gamma(k+1/2)=\frac{\sqrt \pi\, (2k-1)!!}{2^k},\quad k\in {\mathbb N}_0,
	\end{equation}
we can rewrite the expansion coefficient as
	\begin{equation}\label{IdentityForUn-1}
\begin{split}
\hat u_n^{C}&=\frac{1}{\sqrt{\pi}\,2^{m-1}\Gamma(m+1/2)}\int_{-1}^{1} u^{(m)}(x)\, G_{n-m}^{(m)}(x) \omega_{m}(x) \,dx.
\end{split}
\end{equation}

We proceed with the proof by fractional integration by parts. Then it is necessary to use the following identities:   for $m+s>1/2,$ and $n\ge m+s,$
\begin{equation}\label{nmbds0100}
\begin{split}
\omega_{m}(x) G_{n-m}^{(m)}(x)
=    &
- \frac{\Gamma(m+1/2)}{2^{s }\, \Gamma(m+s  +1/2)} \,  {}_{x}I_{1}^{1-s}  \big\{ \omega_{m+s }(x)\, {}^{r\!}G_{n-m-s  }^{(m+s )}(x) \big\}'\\
=&(-1)^{ n+[ n-s]} \frac{\Gamma(m+1/2)}{2^{s }\, \Gamma(m+s +1/2)} \,{}_{-1}I_{x}^{1-s}   \big\{ \omega_{m+s  }(x)\, {}^{l}G_{n-m-s  }^{(m+s)}(x)\big\}'.
\end{split}
\end{equation}
To derive \eqref{nmbds0100}, we substitute $s, \lambda, \nu$ in \eqref{FCI++}-\eqref{FCI--} by $1-s, m-s, n-m+s,$ respectively, leading to
\begin{equation}\label{nmbds001}
\begin{split}
\omega_{m}(x) G_{n-m}^{(m)}(x)
=    &
\frac{2^{1-s}\Gamma(m+1/2)}{\Gamma(m+s -1/2)} \, {}_{x}I_{1}^{1-s} \big\{  \omega_{m+s  -1}(x)\, {}^{r\!}G_{n-m-s+1}^{(m+s -1 )}(x) \big\}\\
=&(-1)^{n+[ n-s+1]}\frac{2^{1-s}\Gamma(m+1/2)}{\Gamma(m+s -1/2)} \, {}_{-1}I_{x}^{1-s} \big\{ \omega_{m+s  -1}(x)\, {}^{l}G_{n-m-s+1}^{(m+s -1 )}(x) \big\}.
\end{split}
\end{equation}
Taking  $s =1, \lambda=m+s $ and $\nu=n-m-s $ in  \eqref{dFCI++}-\eqref{dFCI--}, we obtain  that for $m+s>1/2$,
\begin{equation}\label{IdentityForUn-2}
\begin{split}
\omega_{m+s  -1}(x)\, {}^{r\!}G_{n-m-s +1}^{(m+s -1 )}(x)
=    & -
\frac{\Gamma(m+s-1/2)}{2\, \Gamma(m+s +1/2)} \, \big\{  \omega_{m+s }(x)\, {}^{r\!}G_{n-m-s  }^{(m+s)}(x) \big\}',\\
 \omega_{m+s  -1}(x)\, {}^{l}G_{n-m-s+1 }^{(m+s -1)}(x)
=    &
-\frac{\Gamma(m+s-1/2)}{2\,\Gamma(m+s +1/2)} \, \big\{  \omega_{m+s }(x)\, {}^{l}G_{n-m-s  }^{(m+s  )}(x) \big\}'.
\end{split}
\end{equation}
Substituting \eqref{IdentityForUn-2} into \eqref{nmbds001} leads to  \eqref{nmbds0100}.

For notational convenience, we denote
\begin{equation}\label{fghdefn}
f(x)=u^{(m)}(x),\quad  g(x)= (-1)^{ n+[ n-s]}\omega_{m+s }(x)\, {}^{l}G_{n-m-s  }^{(m+s)}(x),\quad h(x)=- \omega_{m+s }(x)\, {}^{r\!}G_{n-m-s}^{(m+s)}(x).
\end{equation}
By \eqref{nmbds0100},  we can rewrite  \eqref{IdentityForUn-1}  as
\begin{equation}\label{NewhatUn}
\begin{split}
\hat u_n^{C}&=\frac{1}{\sqrt{\pi}\, 2^{m-1}\Gamma(m+1/2)} \bigg\{\int_{-1}^{\theta}  u^{(m)}\, G_{n-m}^{(m)}\, \omega_{m}\,dx+ \int_{\theta}^{1}  u^{(m)}\, G_{n-m}^{(m)}\, \omega_{m}\, dx\bigg\}
\\
&= \frac{1}{\sqrt{\pi}\, 2^{m+s  -1}\Gamma(m+s +1/2)}  \bigg\{
\int^{\theta}_{-1} f(x)\, {}_{-1}I_{x}^{1-s} g'(x) \,dx + \int_{\theta}^{1} f(x)\, {}_{x}I_{1}^{1-s} h'(x)\,  dx \bigg\}.
\end{split}
\end{equation}
We find from  \eqref{obsvers}  and \eqref{IdentityForUn-2},  $g'(x)$ (resp. $h'(x)$)  is  continuous on $(-1,\theta] $ (resp. $[\theta, 1)$), and  they are also integrable when $m+s>1/2.$
Thus, for $f\in L^1(\Omega),$  changing the order of integration by the Fubini's Theorem, we derive from  \eqref{leftintRL}  that
\begin{equation}\label{FracIntPart}
\begin{split}
& \int^{\theta}_{-1}  f(x) {}_{-1}I_{x}^{1-s} g'(x)\,dx=\frac{1}{\Gamma(1-s)}\int_{-1}^\theta  \bigg\{\int _{-1}^x \frac{g'(y)}{(x-y)^s} dy\bigg\}  f(x) \, dx\\
&=\frac{1}{\Gamma(1-s)}\int_{-1}^\theta  \bigg\{\int _{y}^\theta \frac{f(x)}{(x-y)^s} dx\bigg\}  g'(y) \, dy=\frac{1}{\Gamma(1-s)}\int_{-1}^\theta  \bigg\{\int _{x}^\theta \frac{f(y)}{(y-x)^s} dy\bigg\}  g'(x) \, dx
\\
& =\int^{\theta}_{-1} g'(x)\, {}_{x}I_{\theta}^{1-s} f(x)\,dx.
%=&\big\{g(x) \, {}_{x}I_{\theta}^{1-\mu} f(x)\big\}\big|^\theta_{-1}
%+ \int^{\theta}_{-1}   g(x) \, {}_{x}^R D_{\theta}^{\mu}  \, f(x) \, dx.
\end{split}
\end{equation}
Similarly, we can show
\begin{equation}\label{fgdxds}
\int^{1}_{\theta} f(x)\, {}_{x}I_{1}^{1-s} h'(x)\,dx=  \int_{\theta}^{1} h'(x) \, {}_{\theta}I_{x}^{1-s} f(x)\,dx.
\end{equation}
Thus,  if ${}_{x}I_{\theta}^{1-s} f(x)\in {\rm BV}(\Omega_\theta^-)$ and $ {}_{\theta}I_{x}^{1-s} f(x) \in {\rm BV}(\Omega_\theta^+),$ we use
Lemma \ref{IntByPartsInW11}, and   derive
\begin{equation}\label{FracIntPart22}
\begin{split}
\int^{\theta}_{-1}   f(x) & \,{}_{-1}I_{x}^{1-s} g'(x)\,dx  =\int^{\theta}_{-1} g'(x)\, {}_{x}I_{\theta}^{1-s} f(x)\,dx\\
& = \big\{g(x)\, {}_{x}I_{\theta}^{1-s} f(x)\big\}\big|_{-1+}^{\theta-} - \int^{\theta}_{-1} g(x)\, ({}_{x}I_{\theta}^{1-s} f(x))'\,dx\\
&=  \big\{g(x)\, {}_{x}I_{\theta}^{1-s} f(x)\big\}\big|_{x=\theta-} + \int^{\theta}_{-1} g(x)\, {}_{x}^R D_{\theta}^{s} f(x)\,dx,
\end{split}
\end{equation}
where we used  the fact $g(-1)=0$ for $m+s>1/2$ due to \eqref{obsvers},   and also used  \eqref{left-fra-der-rl}.

Similarly, we can show that for $m+s>1/2,$
\begin{equation}\label{fgdxds22}
\int^{1}_{\theta} f(x)\, {}_{x}I_{1}^{1-s} h'(x)\,dx=- \big\{h(x)\, {}_{\theta}I_{x}^{1-s} f(x)\big\}\big|_{x=\theta+} - \int_{\theta}^{1} h(x)\, {}_{\theta}^R D_{x}^{s}\, f(x)\,dx.
\end{equation}
Substituting \eqref{fghdefn} and \eqref{FracIntPart22}-\eqref{fgdxds22} into \eqref{NewhatUn},  we obtain \eqref{HatUnCaseC}.

\vskip 4pt
We next derive  the bounds in \eqref{BoundForUnRealA}-\eqref{BoundForUnRealB}.

(i) For $m=0$ and $s\in (1/2,1)$, we   take  $\lambda=s$ and $\nu=n-s$ in Theorem \ref{BoundGegB},
and then obtain from  \eqref{HatUnCaseC}  and the bound \eqref{BoundForUnRealA} directly.

\vskip 3pt 	
		(ii) We now turn to the proof of \eqref{BoundForUnRealB}.
We first  show  the inequality:
				\begin{equation}\label{BetterBound}
				\begin{split}
				 \frac{2}{\sqrt{2\lambda-1+\nu(\nu+2\lambda)}}\frac{\Gamma({\nu}/2+1)}{\Gamma({\nu}/2+\lambda)}
				\le \frac{ \Gamma( ({\nu}+1)/ 2)}{\Gamma( ({\nu}+1)/2+\lambda)},\;\;\; \nu\ge 0,\;\; \lambda\ge 1.
				\end{split}	\end{equation}	
				To prove \eqref{BetterBound}, we use the property in  \cite[Corollary 2]{Bustoz1986MC}, that is, the ratio
				$$f(z):=\frac{1}{\sqrt z}\frac{\Gamma(z+1)}{\Gamma(z+1/2)},\;\;\; z>0,$$
				is decreasing. Then using the facts:
				$$ (\nu-1)/2+\lambda>0,\;\;\; (\nu-1)/2+\lambda> (\nu+1)/2,$$
we can derive
				\begin{equation}\label{BoundForUnReal-1}
				\begin{split}
				\frac 1 {\sqrt{(\nu-1)/2+\lambda}}\frac{\Gamma( ({\nu}+1)/2+\lambda)}{\Gamma( {\nu}/2+\lambda)}\le&\frac 1 {\sqrt{\nu/2+1/2}}\frac{\Gamma( ({\nu}+3)/2)}{\Gamma( {\nu}/2+1)}
				=\sqrt{\frac {\nu+1} 2}\frac{\Gamma( ({\nu}+1)/2)}{\Gamma( {\nu}/2+1)},
				\end{split}\end{equation}
				where in the last step, we used the identity: $\Gamma(z+1)=z\Gamma(z).$	
Next, we rewrite  \eqref{BoundForUnReal-1} as
				\begin{equation}\label{BoundForUnReal-2}
				\begin{split}	
				%(\nu/2+\lambda-1/2)^{-1/2}(\nu/2+1/2)^{-1/2}
\frac{2}{\sqrt{(\nu+2\lambda-1)   (\nu+1)}}
\frac{\Gamma( {\nu}/2+1)}{\Gamma( {\nu}/2+\lambda)}\le\frac{\Gamma( ({\nu}+1)/2)}{\Gamma( ({\nu}+1)/2+\lambda)}.
				\end{split}	\end{equation}
				Noting that
				\begin{equation*}%\label{BoundForUnReal-5}
				\frac{2}{\sqrt{2\lambda-1+\nu(\nu+2\lambda)}}=\frac{2}{\sqrt{(\nu+2\lambda-1)   (\nu+1)}},
				\end{equation*}	
				  we  obtain \eqref{BetterBound} from \eqref{BoundForUnReal-2} immediately.
Thanks to \eqref{BetterBound},  we   derive from Theorem \ref{LBoundForGegPoly}  that
				\begin{equation}\label{NewKappaB}
				\begin{split}
					\max_{|x|\le 1}\big\{	\omega_{\lambda}(x) \big|{}^{r\!}G_\nu^{(\lambda)}(x)\big|,	\; \omega_{\lambda}(x)\big| {}^{l}G_\nu^{(\lambda)}(x)\big| \big\}
				\le &\frac{\Gamma(\lambda+ 1/2)}{\sqrt \pi}
				\frac{ \Gamma( ({\nu}+1)/ 2)}{\Gamma( ({\nu}+1)/2+\lambda)} ,
				\end{split}	\end{equation}
so the bound in \eqref{BoundForUnRealB} follows from \eqref{HatUnCaseC} with  $\lambda=m+s $ and $ \nu=n-m-s$ in \eqref{NewKappaB}.
\end{proof}

\subsection{$L^\infty$- and $L^2$-estimates of Chebyshev expansions}   With Theorem \ref{IdentityForUn} at our disposal, we can analyze all related orthogonal projections, interpolations and quadratures (cf. \!\cite{Majidian2017ANM}).
Here, we first estimate the Chebyshev expansion errors  in the  $L^\infty$-norm and $L^2_\omega$-norm for functions with interior singularities.  We remark that if the function is sufficiently  smooth, we understand the results with $s=1$, i.e.,  in the space ${\mathbb  W}^{m+1}(\Omega)$ defined in   \eqref{intergerASobB}. We refer to Theorem \ref{TruncatedChebyshev} for the integer case.
\begin{thm}\label{TruncatedChebyshevLinf} 	Given $\theta\in (-1,1),$ if  $u\in {\mathbb  W}^{m+s}_{\theta}(\Omega) $ with
	$s\in (0,1)$ and integer  $m\ge 0$, we have the following estimates.
	\begin{itemize}
		\item[(i)] For $1< m+s\le N+1,$
			\begin{equation}\label{FracLinfA}
		\begin{split}
		\|u-\pi_N^C u\|_{L^\infty (\Omega)}\le &\frac {U_\theta^{m,s} }{2^{m+s-2} (m+s-1)\pi}\frac{ \Gamma( ({N-m-s})/2+1)}{\Gamma( ({N+m+s})/2)}.
		\end{split}
		\end{equation}
		\item[(ii)] For $ 1/2<m+s< N+1,$
			\begin{equation}\label{FracL2B}
		\begin{split}
		\|u-\pi_N^C u\|_{L^2_{\omega}(\Omega)}\le\bigg\{\frac {2^3}{(2m+2s-1)\pi}
		\frac{ \Gamma( {N-m-s+1})}{\Gamma( {N+m+s})}\bigg\}^{1/2}U_\theta^{m,s}.
		\end{split}
		\end{equation}
	\end{itemize}
	%where  $C$ is a  positive constant independent of $N$ and $u$.
\end{thm}
\begin{proof}  We first prove \eqref{FracLinfA}.  For simplicity,  we  denote
\begin{equation}\label{Snsigma}
{\mathcal S}_n^\sigma:=\frac{ \Gamma( ({n-\sigma}+1)/ 2)}{\Gamma( ({n+\sigma}+1)/2)},\quad
{\mathcal T}_n^\sigma:=\frac{ \Gamma( ({n-\sigma}+1)/ 2)}{\Gamma( ({n+\sigma}-1)/2)} ,\quad \sigma:=m+s.
\end{equation}
A direct calculation leads to the identity:
\begin{equation}\begin{split}\label{FracLinf-1}
{\mathcal T}_n^\sigma-{\mathcal T}_{n+2}^\sigma=&
\frac{{n+\sigma}-1} 2\frac{ \Gamma( ({n-\sigma}+1)/ 2)}{\Gamma( ({n+\sigma}+1)/2)}-\frac{{n-\sigma}+1} 2\frac{ \Gamma( ({n-\sigma}+1)/ 2)}{\Gamma( ({n+\sigma}+1)/2)}\\
=&(\sigma-1)\frac{ \Gamma( ({n-\sigma}+1)/ 2)}{\Gamma( ({n+\sigma}+1)/2)}= (\sigma-1){\mathcal S}_n^\sigma,
\end{split}\end{equation}
where we used the identity $z\Gamma(z)=\Gamma(z+1)$.
By \eqref{BoundForUnRealB},
\begin{equation}\begin{split}\label{FracLinf-3}
\big|u(x) -\pi_N^C u(x)\big|&\le \sum_{n=N+1}^{\infty}|\hat u_n^{C}| \le
\frac{U_\theta^{m,s}}{2^{\sigma-1} \pi}\sum_{n=N+1}^{\infty}
{\mathcal S}_n^\sigma
= \frac{U_\theta^{m,s}}{2^{\sigma-1}  (\sigma-1)\pi}  \sum_{n=N+1}^{\infty}\big\{{\mathcal T}_n^\sigma-{\mathcal T}_{n+2}^\sigma\big\}
\\&=\frac{U_\theta^{m,s}}{2^{\sigma-1}  (\sigma-1)\pi} \bigg\{{\mathcal T}_{N+1}^\sigma+{\mathcal T}_{N+2}^\sigma
+
\sum_{n=N+3}^{\infty}{\mathcal T}_n^\sigma-
\sum_{n=N+1}^{\infty}{\mathcal T}_{n+2}^\sigma\bigg\}\\&=\frac{U_\theta^{m,s}}{2^{\sigma-1}  (\sigma-1)\pi} \big\{{\mathcal T}_{N+1}^\sigma+{\mathcal T}_{N+2}^\sigma\big\}.
	\end{split}\end{equation}
 We find  from
	\cite[(1.1) and Theorem 10]{Alzer1997MC} that for $0\le a\le b$, the ratio
	\begin{equation}\label{gammratio}
	R_b^a(z):=\frac{\Gamma(z + a)}{\Gamma(z + b)},\quad z\ge 0,
	\end{equation}
	is decreasing with respect to $z.$ As $\sigma-1> 0,$ we have
	\begin{equation}\label{Tsigmabnd}
{\mathcal T}_{N+2}^\sigma=R_{\sigma-1}^0(1+(N-\sigma+1)/2) \le R_{\sigma-1}^0(1+(N-\sigma)/2)={\mathcal T}_{N+1}^\sigma.
	\end{equation}
Therefore, the estimate \eqref{FracLinfA}  follows from \eqref{FracLinf-3} and \eqref{Tsigmabnd}.

We now turn to the estimate \eqref{FracL2B} with $1<  m+s\le N+1$. Similar to   \eqref{Tsigmabnd}, we can use \eqref{gammratio} to show that ${\mathcal S}_n^\sigma\le
{\mathcal S}_{n-1}^\sigma.$ Thus, using the identity
\begin{equation}\begin{split}\label{DuplicationFormula}
	\Gamma(2z)=\pi^{-1/2}2^{2z-1}\Gamma(z)\Gamma(z+1/2),
	\end{split}\end{equation}
	we derive
	 \begin{equation}\begin{split}\label{FracL2-4}
({\mathcal S}_n^\sigma)^2 &\le {\mathcal S}_n^\sigma {\mathcal S}_{n-1}^\sigma = \frac{ \Gamma( ({n-\sigma}+1)/ 2)}{\Gamma( ({n+\sigma}+1)/2)}\frac{ \Gamma( ({n-\sigma})/ 2)}{\Gamma( ({n+\sigma})/2)} =2^{2\sigma}\frac{ \Gamma( {n-\sigma})}{\Gamma( {n+\sigma})}\\
&=\frac{2^{2\sigma}}{2\sigma-1}
\bigg(\frac{ \Gamma( {n-\sigma})}{\Gamma( {n-1+\sigma})}-\frac{ \Gamma( {n+1-\sigma})}{\Gamma( {n+\sigma})}\bigg).
\end{split}\end{equation}
Then, for $\sigma> 1$,
\begin{equation}\label{FracL2-5}
\begin{split}
\big\|u-\pi_N^{C} u\big\|_{L^2_{\omega}(\Omega)}^2=&\frac{\pi}{2}\sum_{n=N+1}^{\infty}\big|\hat u_n^{C}\big|^2
\le   \frac { (U_\theta^{m,s})^2}{2^{2\sigma-3}\pi}\sum_{n=N+1}^{\infty}	
({\mathcal S}_n^\sigma)^2 \le\frac {2^3 (U_\theta^{m,s})^2}{(2\sigma-1)\pi}
\frac{ \Gamma( {N-\sigma+1})}{\Gamma( {N+\sigma})}.
\end{split}
\end{equation}

Finally, we prove \eqref{FracL2B} with $m=0$ and $s\in (1/2,1)$ by using  \eqref{BoundForUnRealA}.  Note that
 \eqref{FracL2-5} is valid for $m=0,$ so we have
 \begin{equation}\label{FracL2-500A}
\begin{split}	
({\mathcal S}_n^s)^2=\frac{\Gamma^2((n-s)/2+1)}{\Gamma^2((n+s)/2)}
& \le\frac{2^{2s}}{2s-1}
\bigg(\frac{ \Gamma( {n-s})}{\Gamma( {n-1+s})}-\frac{ \Gamma( {n+1-s})}{\Gamma( {n+s})}\bigg).
\end{split}
\end{equation}
 For the second factor in the upper bound \eqref{BoundForUnRealA},  we also use  \eqref{gammratio} and \eqref{DuplicationFormula}-\eqref{FracL2-4} to show
\begin{equation}\label{FracL2-6}
\begin{split}
\frac{4}{n^2-s^2+s}&\frac{\Gamma^2((n-s)/2+1)}{\Gamma^2((n+s)/2)}
\le \frac{4}{n^2-s^2+s}\frac{\Gamma((n-s)/2+1)}{\Gamma((n+s)/2)}
\frac{\Gamma((n-s)/2+1/2)}{\Gamma((n+s)/2-1/2)}\\
&= \frac{2^{2s}}{n^2-s^2+s}\frac{\Gamma(n-s+1)}{\Gamma(n+s-1)}
=2^{2s}\frac{n^2-s^2+s-n}{n^2-s^2+s}\frac{\Gamma(n-s)}{\Gamma(n+s)}\\
& \le 2^{2s}\frac{\Gamma(n-s)}{\Gamma(n+s)}
=\frac{2^{2s}}{2s-1}
\bigg(\frac{ \Gamma( {n-s})}{\Gamma( {n-1+s})}-\frac{ \Gamma( {n+1-s})}{\Gamma( {n+s})}\bigg).
%=  \frac{2^{2s}}{2s-1}\frac{ \Gamma( {N-s+1})}{\Gamma( {N+s})}.
	\end{split}
	\end{equation}
	Thus, from \eqref{BoundForUnRealA}, we obtain
\begin{equation}\label{BoundForUnRealA00}
\begin{split}
|\hat u_n^{C}|^2\le & \frac{4 (U_\theta^{0,s})^2}{(2s-1) \pi^2} \bigg(\frac{ \Gamma( {n-s})}{\Gamma( {n-1+s})}-\frac{ \Gamma( {n+1-s})}{\Gamma( {n+s})}\bigg).
\end{split}
\end{equation}
With this, we derive
\begin{equation}\label{FracL2-500}
\begin{split}
\big\|u-\pi_N^{C} u\big\|_{L^2_{\omega}(\Omega)}^2&=\frac{\pi}{2}\sum_{n=N+1}^{\infty}\big|\hat u_n^{C}\big|^2
%\le   \frac { U_{0,s}^2}{2^{2s-3}\pi}\sum_{n=N+1}^{\infty}	
%\frac{\Gamma^2((n-s)/2+1)}{\Gamma^2((n+s)/2)}\\&
\le\frac {2^3 (U_\theta^{0,s})^2}{(2s-1)\pi}
\frac{ \Gamma( {N-s+1})}{\Gamma( {N+s})}.
\end{split}
\end{equation}
This completes the proof.
\end{proof}
\begin{rem}\label{expbounds} {\em Recall that {\rm(cf.}  \cite[(5.11.13)]{Olver2010Book}{\rm):} for $a<b$,
\begin{equation}\label{gamratioA}
\frac{\Gamma(z+a)}{\Gamma(z+b)}=z^{a-b}+\frac{1}{2}(a-b)(a+b-1)z^{a-b-1}+O(z^{a-b-2}),\quad z\gg 1.
\end{equation}	
Thus,  under the conditions of Theorem \ref{TruncatedChebyshevLinf} and for fixed $m$ and large $n$ or $N,$ we have
\begin{equation}\label{newbndadd}
\begin{split}
&|\hat u_n^C|\le C n^{-(m+s)} U_\theta^{m,s}\, ,\qquad  \|u-\pi_N^C u\|_{L^\infty (\Omega)}\le CN^{1-(m+s)} U_\theta^{m,s}\, ,\\[4pt]
& \|u-\pi_N^C u\|_{L^2_\omega (\Omega)}\le CN^{\frac 1 2-(m+s)} U_\theta^{m,s}\,,
\end{split}
\end{equation}
where $C$ is a positive constant independent of $n, N$ and $u.$ }
\end{rem}

%\section{Illustrative examples}\label{Iusexample}
%\setcounter{equation}{0}
%\setcounter{lmm}{0}
%\setcounter{thm}{0}

\subsection{Applications to functions with interior singularities}\label{Iusexample} In what follows, we apply the main results to  two typical types of singular functions, and provide numerical illustrations of the optimal convergence order.
\begin{itemize}
\item{\sc Type-I}:\,  Consider
\begin{equation}\label{uthetafun}
u(x)=|x-\theta|^\alpha,\quad  \alpha>-1/2,\;\;\; x, \theta\in(-1,1),
\end{equation}
where   $\alpha$ is not an even integer.
\vskip 4pt
\item{\sc Type-II}:\,  Consider
\begin{equation}\label{uthetafun2}
u(x)=|x-\theta|^\alpha \ln |x-\theta|,\quad  \alpha>-1/2,\;\;\; x, \theta\in(-1,1).
\end{equation}
\end{itemize}
\subsubsection{{\sc Type-I:}\, $u(x)=|x-\theta|^\alpha$ in  \eqref{uthetafun}} %  with $\alpha>-1/2$ and $x,\theta\in(-1,1)$}
% The following proposition shows
%that our proposed fractional framework can best characterise its regularity.  Moreover, the identity  \eqref{HatUnCaseC} in Theorem \ref{IdentityForUn} can lead to precise formulas for the expansion coefficients.

\begin{thm}\label{Wspacproof} Given  the function in \eqref{uthetafun},  we have that {\rm (i)} if $\alpha$ is an odd integer, then $ u\in {\mathbb W}^{\alpha+1}(\Omega)$ {\rm (defined in \eqref{intergerASobB})};  and {\rm (ii)} if $\alpha$ is not an integer, then $ u\in {\mathbb W}^{\alpha+1}_\theta(\Omega)$ {\rm (}defined in \eqref{fracA}{\rm)}.

Its Chebyshev expansion coefficients can be expressed as
\begin{equation}\label{SpecialCase}
	\begin{split}
	\hat u_n^{C} =&\frac{\Gamma(\alpha+1)}{2^{\alpha}\Gamma(\alpha+3/2)\sqrt{\pi}}\big\{{}^{r\!}G_{n-\alpha-1}^{(\alpha+1)}(\theta)
-(-1)^{n+[n-\alpha]}  \,{}^{l}G_{n-\alpha-1}^{(\alpha+1)}(\theta) \big\} \omega_{\alpha+1}(\theta),
	\end{split}
	\end{equation}
		for all  $n\ge \alpha+1.$ Moreover, we have the  bounds uniform for $n:$
\begin{itemize}
	\item[(a)] for $-1/2<\alpha<0$,
	\begin{equation}\label{CaseOne-0}
	\begin{split}
|	\hat u_n^{C}| \le &\frac{\Gamma(\alpha+1)}{2^{\alpha-1}\pi}(1-\theta^2)^{\alpha/2}
	\max\bigg\{
	\frac{ \Gamma((n-\alpha)/2) }{\Gamma((n+\alpha)/2+1)},\frac{2}{\sqrt{n^2-\alpha(\alpha+1)}}\frac{\Gamma((n-\alpha+1)/2)}
{\Gamma((n+\alpha+1)/2)}\bigg\};
	\end{split}
	\end{equation}
	\item[(b)] for $\alpha \ge 0$,
	\begin{equation}\label{CaseTwo-0}
	\begin{split}
	|\hat u_n^{C}| \le &\frac{\Gamma(\alpha+1)}{2^{\alpha-1}\pi}\frac{ \Gamma((n-\alpha)/2) }{\Gamma((n+\alpha)/2+1)}.
	\end{split}
	\end{equation}
\end{itemize}
\end{thm}
\begin{proof} (i) If $\alpha$ is an odd integer, we find
\begin{equation}\label{mderivaR}
\begin{split}
 & u^{(k)} =  d_\alpha^k\,   |x-\theta|^{\alpha-k}\, ({\rm sgn }(x-\theta))^k  \in  {\rm AC}(\Omega),\quad 0\le k\le \alpha-1; \quad  \\
 &  u^{(\alpha)}=  d_\alpha^\alpha\,  (2H(x-\theta)-1)\in {\rm BV}(\Omega).
 %,\quad u^{(\alpha+1)}= 2 d_\alpha^\alpha\,  \delta(x-\theta)\in L^1(\Omega),
\end{split}
\end{equation}
% sign function
where ${\rm sign} (z), H(z), \delta(z) $ are the sign,  Heaviside  and Dirac Delta functions, respectively, and
\begin{equation}\label{dkalpha}
d_\alpha^k:=\alpha (\alpha-1) \cdots (\alpha-k+1)=\frac{\Gamma(\alpha+1)}{\Gamma(\alpha-k+1)}.
\end{equation}
Thus, from  \eqref{intergerASobB}, we claim   $ u\in {\mathbb W}^{\alpha+1}(\Omega). $ Moreover, by \eqref{HatUnCaseA} (with $m=\alpha$),   \eqref{gammn12} and
   \eqref{mderivaR},
\begin{equation*}\label{HatUnCaseA007}
	\begin{split}
	\hat u_n^{C}&= \frac 1 {(2\alpha+1)!!}  \frac 2\pi \int_{-1}^{1} \, G_{n-\alpha-1}^{(\alpha+1)}(x) \omega_{\alpha+1}(x) \,d u^{(\alpha)}(x)\\& =\frac{\Gamma(\alpha+1)}{2^{\alpha-1}\Gamma(\alpha+3/2)\sqrt{\pi}} G_{n-\alpha-1}^{(\alpha+1)}(\theta) \omega_{\alpha+1}(\theta),
	\end{split}
	\end{equation*}
which is identical to \eqref{SpecialCase} with $\alpha$ being an odd integer, thanks to \eqref{obsvers0}.

(ii) If $\alpha$ is not an integer,  let
$m=[\alpha]+1 $ and  $s=\{\alpha+1\}\in (0,1). $ Like \eqref{mderivaR}, we have $u, \cdots, u^{(m-1)}\in {\rm AC}(\Omega).$  By a direct calculation, we infer from \eqref{intformu} that
for $ x\in (-1,\theta),$
\begin{equation}\label{mderivaR00}
\begin{split}
 &{}_{x}I_{\theta}^{1-s} u^{(m)}=(-1)^m\,d_\alpha^m \, {}_{x}I_{\theta}^{m-\alpha} (\theta-x)^{\alpha-m}=(-1)^m\,d_\alpha^m\, \Gamma(s)=( -1)^{[\alpha]+1}\, \Gamma(\alpha+1),\;\;\;\\
\end{split}
\end{equation}
while for $x\in (\theta,1),$
\begin{equation}\label{mderivaR00B}
\begin{split}
 & {}_{\theta}I_{x}^{1-s} u^{(m)}=d_\alpha^m \, {}_{\theta}I_{x}^{m-\alpha} (x-\theta)^{\alpha-m}=d_\alpha^m\, \Gamma(s)=\Gamma(\alpha+1).
\end{split}
\end{equation}
Therefore, by the definition \eqref{fracA},  we have $u\in {\mathbb W}^{\alpha+1}_\theta(\Omega).$

 It is clear that by \eqref{mderivaR00}-\eqref{mderivaR00B},
${}_{x}^R D_{\theta}^{s} \, u^{(m)}(x)={}_{\theta}^R D_{x}^{s} \, u^{(m)}(x)=0,$
so we can derive the exact formula \eqref{SpecialCase} by using \eqref{HatUnCaseC} straightforwardly.

	(a) For $-1/2<\alpha<0$, taking $s=\alpha+1$ in \eqref{BoundForUnRealA}, leads to 		
	\begin{equation*}%\label{CaseOne}
	\begin{split}
	|\hat u_n^{C}| \le &\frac{\Gamma(\alpha+1)}{2^{\alpha-1}\pi}(1-\theta^2)^{\alpha/2}
	\max\bigg\{
	\frac{ \Gamma((n-\alpha)/2) }{\Gamma((n+\alpha)/2+1)},\frac{2}{\sqrt{n^2-\alpha(\alpha+1)}}\frac{\Gamma((n-\alpha+1)/2)}{\Gamma((n+\alpha+1)/2)}\bigg\},
	\end{split}
	\end{equation*}	
where we used the fact $U_{0,\alpha+1}=2(1-\theta^2)^{\alpha/2}\Gamma(\alpha+1).$

	(b) Similarly, we can obtain \eqref{CaseTwo-0} directly from \eqref{BoundForUnRealB}.
%	  that for $\alpha \ge 0$,
%	\begin{equation*}%\label{CaseTwo}
%	\begin{split}
%|	\hat u_n^{C}| \le &\frac{\Gamma(\alpha+1)}{2^{\alpha-1}\pi}\frac{ \Gamma((n-\alpha)/2) }{\Gamma((n+\alpha)/2+1)}.
%	\end{split}
%	\end{equation*}
%This completes the proof.	
\end{proof}

%Thus, by \eqref{HatUnCaseC} and \eqref{mderivaR00},
%\begin{equation}\label{SpecialCase00}
%	\begin{split}
%	\hat u_n^{C}
%	=&\frac{\Gamma(\alpha+1)}{2^{\alpha}\Gamma(\alpha+3/2)\sqrt{\pi}}\Big\{(-1)^{[\alpha]} \, \omega_{\alpha+1}(\theta)\,{}^{l}G_{n-\alpha-1}^{(\alpha+1)}(\theta)-\omega_{\alpha+1}(\theta)\, {}^{r\!}G_{n-\alpha-1}^{(\alpha+1)}(\theta)\Big\},
%	\end{split}
%	\end{equation}

\begin{rem}\label{symformA} {\em As a special case of
 \eqref{SpecialCase} with $\theta=0$,  we obtain from \eqref{obsvers} and \eqref{LBoundGeg-5} that
 the Chebyshev expansion coefficients of $|x|^\alpha$ have the exact representation for each integer $n\ge 0$,
	\begin{equation}\label{SpecialCasetheta=0}
\begin{split}
\hat u_n^{C}
%=&\frac{\big\{(-1)^{n-1}-1\big\}\Gamma(\alpha+1)}{2^{\alpha}\Gamma(\alpha+3/2)\sqrt{\pi}} {}^{r\!}G_{n-\alpha-1}^{(\alpha+1)}(0)\\
%=&\frac{\big\{(-1)^{n-1}-1\big\}\Gamma(\alpha+1)}{2^{\alpha}\Gamma(\alpha+3/2)\sqrt{\pi}}\sin \big(\pi (n-\alpha)/2\big)\frac{\Gamma(\alpha+3/2)\Gamma( {(n-\alpha)}/2)}{\sqrt{\pi}\, \Gamma( (n+\alpha)/2+1)}\\
=&\big((-1)^{n}+1\big)\frac{\Gamma(\alpha+1)\Gamma( {(n-\alpha)}/2)}{2^{\alpha}\pi \Gamma( (n+\alpha)/2+1)} \sin \Big(\frac{(n-\alpha) \pi} 2\Big),
\end{split}
\end{equation}
which implies that for integer $k\ge 0,$
\begin{equation}\label{SpecialCasetheta=01}
\begin{split}
\hat u_{2k+1}^C=0,\quad \hat u_{2k}^C=(-1)^{k}\sin \frac{\alpha \pi} 2
\frac{\Gamma(\alpha+1)} {2^{\alpha-1}\pi}\frac{\Gamma(k-\alpha/2)} {\Gamma(k+\alpha/2+1)}.
\end{split}
\end{equation}
It is noteworthy that the following asymptotic estimate for large $k$ was obtained in   \cite[Sec. 3.11]{Olver1974Book}:
\begin{equation}\label{unCest}
	\hat u_{2k}^{C}\simeq (-1)^{k}\sin \frac{\alpha \pi} 2
\frac{\Gamma(\alpha+1)} {2^{\alpha-1}\pi} k^{-\alpha-1},%+ O(k^{-\alpha-3})=O(k^{-\alpha-1}),
\end{equation}
but by different means.
Indeed, our approach leads to exact representations for all $n$. }
\end{rem}

%\begin{rem}\label{newaddrmk} The analysis is  applicable to more general functions, e.g., $u(x)=|x-\theta|^\alpha g(x)$ with smooth $g(x)$. \qed
%\end{rem}

Note that we can directly apply Theorem \ref{TruncatedChebyshevLinf} (also see Remark \ref{expbounds}) to bound the errors of the Chebyshev expansion of the above type of singular functions. For example, if $\alpha$ is not an integer, we know $ u\in {\mathbb W}^{\alpha+1}_\theta(\Omega),$ so we have
\begin{equation}\label{newbndaddbus}
\begin{split}
 \|u-\pi_N^C u\|_{L^\infty (\Omega)}\le CN^{-\alpha}\,  ,\quad  \|u-\pi_N^C u\|_{L^2_\omega (\Omega)}\le CN^{-\alpha-1/2}\,.
\end{split}
\end{equation}
%where  $m=[\alpha+1]$ and $s=\{\alpha+1\}.$
We tabulate in Table \ref{LinfSeta} the errors and convergence order of Chebyshev approximations to  $u(x)=|x-\theta|^\alpha$ with various $\alpha$ and $\theta=0,1/2.$
\begin{table}[!htbp]
	\centering
	\caption{Convergence order of  $u=|x-\theta|^\alpha$ with $\theta=0,1/2$.} %of  $\|u-\pi_N^C u\|_{L^\infty (\Omega)}$ and $\|u-\pi_N^C u\|_{L^2_\omega(\Omega)}$.}
	\small
	\begin{tabular}{|c|c|c|c|c|c|c|c|c|c|c|c|c|}
		\hline
		\multicolumn{1}{ |c  }{\multirow{2}{*}{$N$} } &
		\multicolumn{6}{ |c| }{$u=|x|^\alpha$\;  (error in $L^\infty$-norm)} & \multicolumn{6}{ |c| }{$u=|x-1/2|^\alpha$ \;  (error in $L^\infty$-norm)}      \\ \cline{2-13}
		\multicolumn{1}{ |c  }{}                        &
		\multicolumn{1}{ |c| }{$\alpha=0.1$}  & order
		&$ \alpha=1.2$  & order  & $\alpha=2.6$  & order  & $\alpha=0.1$  & order
		&$ \alpha=1.2$  & order & $\alpha=2.6$  & order \\ \hline
%		$2^3$ &7.60e-1 &-- & 3.98e-2 & -- & 2.51e-3 &-- &7.58e-1 & -- & 3.86e-2 & -- & 2.29e-3 & --   \\
%		$2^4$ &7.14e-1 & 0.09 & 1.86e-2 & 1.10 & 4.70e-4 & 2.42   &7.06e-1 & 0.10 & 1.62e-2 & 1.25 & 3.55e-4 & 2.69   \\
		$2^5$ &6.68e-1 & -- & 8.37e-3 & -- & 8.32e-5 & -- &6.60e-1 & -- & 7.31e-3 & -- & 6.18e-5 & --  \\
		$2^6$ &6.24e-1 & 0.10 & 3.71e-3 & 1.17 & 1.42e-5 & 2.55 &6.16e-1 & 0.10 & 3.15e-3 & 1.21 & 1.00e-5 & 2.63  \\
		$2^7$ &5.83e-1 & 0.10 & 1.63e-3 & 1.19 & 2.40e-6 & 2.57 &5.75e-1 & 0.10 & 1.38e-3 & 1.19 & 1.68e-6 & 2.57 \\
		$2^8$ &5.44e-1 & 0.10 & 7.13e-4 & 1.19 & 3.99e-7 & 2.59  &5.36e-1 & 0.10 & 6.01e-4 & 1.20 & 2.76e-7 & 2.61 \\
		$2^9$ &5.08e-1 & 0.10 & 3.11e-4 & 1.20 & 6.62e-8 & 2.59 &5.00e-1 & 0.10 & 2.62e-4 & 1.20 & 4.58e-8 & 2.59 \\
		$2^{10}$ &4.74e-1 & 0.10 & 1.36e-4 & 1.20 & 1.09e-8 & 2.60  &4.67e-1 & 0.10 & 1.14e-4 & 1.20 & 7.54e-9 & 2.60\\
		\hline
		\hline
		\multicolumn{1}{ |c  }{\multirow{2}{*}{$N$} } &
		\multicolumn{6}{ |c| }{$u=|x|^\alpha$ \;  (error in $L^2_\omega$-norm) } & \multicolumn{6}{ |c| }{$u=|x-1/2|^\alpha$ \;  (error in $L^2_\omega$-norm)}      \\ \cline{2-13}
		\multicolumn{1}{ |c  }{}                        &
		\multicolumn{1}{ |c| }{$\alpha=0.1$}  & order
		&$ \alpha=1.2$  & order & $\alpha=2.6$  & order  & $\alpha=0.1$  & order
		&$ \alpha=1.2$  & order & $\alpha=2.6$  & order \\ \hline
%		$2^3$  & 4.09e-2 & -- & 1.52e-2 & -- & 1.53e-3 & -- & 4.33e-2 & 0.49 & 1.56e-2 & 1.49 & 1.48e-3 & 3.30  \\
%		$2^4$ & 2.80e-2 & 0.55 & 5.18e-3 & 1.56 & 2.10e-4 & 2.87& 2.81e-2 & 0.62 & 4.61e-3 & 1.76 & 1.64e-4 & 3.17 \\
		$2^5$  & 1.88e-2 & -- & 1.68e-3 & -- & 2.68e-5 & --  & 1.89e-2 & -- & 1.49e-3 & -- & 2.02e-5 & --\\
		$2^6$ & 1.25e-2 & 0.59 & 5.31e-4 & 1.66 & 3.27e-6 & 3.03 & 1.24e-2 & 0.61 & 4.53e-4 & 1.72 & 2.31e-6 & 3.13\\
		$2^7$ & 8.29e-3 & 0.59 & 1.66e-4 & 1.68 & 3.91e-7 & 3.07 & 8.22e-3 & 0.59 & 1.41e-4 & 1.68 & 2.75e-7 & 3.07\\
		$2^8$  & 5.48e-3 & 0.60 & 5.13e-5 & 1.69 & 4.61e-8 & 3.08 & 5.42e-3 & 0.60 & 4.33e-5 & 1.70 & 3.19e-8 & 3.11 \\
		$2^9$ & 3.61e-3 & 0.60 & 1.58e-5 & 1.70 & 5.41e-9 & 3.09 & 3.58e-3 & 0.60 & 1.34e-5 & 1.70 & 3.74e-9 & 3.09\\
		$2^{10}$  & 2.37e-3 & 0.61 & 4.88e-6 & 1.70 & 6.33e-10 & 3.10& 2.36e-3 & 0.60 & 4.11e-6 & 1.70 & 4.36e-10 & 3.10 \\
		\hline	
	\end{tabular} \label{LinfSeta}
\end{table}

%
%\begin{table}[!htbp]
%	\centering
%	\caption{\small Order of $\|u-\pi_N^C u\|_{L^2_{\omega}(\Omega)}$.}
%	\small
%	\begin{tabular}{|c|c|c|c|c|c|c|c|c|c|c|c|c|}
%		\hline
%		\multicolumn{1}{ |c  }{\multirow{2}{*}{$N$} } &
%		\multicolumn{6}{ |c| }{$u=|x|^\alpha$} & \multicolumn{6}{ |c| }{$u=|x-1/2|^\alpha$}     \\ \cline{2-13}
%		\multicolumn{1}{ |c  }{}                        &
%		\multicolumn{1}{ |c| }{$\alpha=0.1$}  & order
%		&$ \alpha=1.2$  & order & $\alpha=2.6$  & order  & $\alpha=0.1$  & order
%		&$ \alpha=1.2$  & order & $\alpha=2.6$  & order \\ \hline
%%		$2^3$  & 4.09e-2 & -- & 1.52e-2 & -- & 1.53e-3 & -- & 4.33e-2 & 0.49 & 1.56e-2 & 1.49 & 1.48e-3 & 3.30  \\
%%		$2^4$ & 2.80e-2 & 0.55 & 5.18e-3 & 1.56 & 2.10e-4 & 2.87& 2.81e-2 & 0.62 & 4.61e-3 & 1.76 & 1.64e-4 & 3.17 \\
%		$2^5$  & 1.88e-2 & 0.57 & 1.68e-3 & 1.63 & 2.68e-5 & 2.97  & 1.89e-2 & 0.57 & 1.49e-3 & 1.63 & 2.02e-5 & 3.02\\
%		$2^6$ & 1.25e-2 & 0.59 & 5.31e-4 & 1.66 & 3.27e-6 & 3.03 & 1.24e-2 & 0.61 & 4.53e-4 & 1.72 & 2.31e-6 & 3.13\\
%		$2^7$ & 8.29e-3 & 0.59 & 1.66e-4 & 1.68 & 3.91e-7 & 3.07 & 8.22e-3 & 0.59 & 1.41e-4 & 1.68 & 2.75e-7 & 3.07\\
%		$2^8$  & 5.48e-3 & 0.60 & 5.13e-5 & 1.69 & 4.61e-8 & 3.08 & 5.42e-3 & 0.60 & 4.33e-5 & 1.70 & 3.19e-8 & 3.11 \\
%		$2^9$ & 3.61e-3 & 0.60 & 1.58e-5 & 1.70 & 5.41e-9 & 3.09 & 3.58e-3 & 0.60 & 1.34e-5 & 1.70 & 3.74e-9 & 3.09\\
%		$2^{10}$  & 2.37e-3 & 0.61 & 4.88e-6 & 1.70 & 6.33e-10 & 3.10& 2.36e-3 & 0.60 & 4.11e-6 & 1.70 & 4.36e-10 & 3.10 \\ 	
%		\hline
%	\end{tabular} \label{L2Seta}
%\end{table}

\vskip 6pt

\subsubsection{{\sc Type-II:}\, $u(x)=|x-\theta|^\alpha\ln|x-\theta|$ in \eqref{uthetafun2}}
We first present the following  useful  formulas.
\begin{lemma}\label{LmmForLogSingularity}
For real $\eta>-1, s\ge 0$ and $x>a,$
	\begin{equation}\label{IntmuGamma}
	\begin{split}
	{}_a I_{x}^s \{(x-a)^\eta\ln (x-a)\}  =	\frac{\Gamma(\eta+1)}{\Gamma(\eta+s+1)}
	 \big\{\ln (x-a)+ \psi(\eta+1) -\psi(\eta+s+1)\big\} (x-a)^{\eta+s},
	\end{split}
	\end{equation}
and the same formula holds for ${}_x I_{b}^s \{(b-x)^\eta\ln (b-x)\}$ {\rm(}for $x<b${\rm)} with
$b-x$ in place of $x-a$.
Here,  %the $\psi$ function is defined as %(see, e.g., \cite{Alzer1997MC} or \cite[P. 911]{Gradshteyn2015Book})
\begin{equation}\label{psizfun}
\ln z-\frac 1 {2z}<\psi(z)=\frac{ \Gamma'(z)} {\Gamma(z)}<\ln z-\frac 1 z,\quad z>0.
\end{equation}
	\end{lemma}
\begin{proof} The formula \eqref{IntmuGamma} is a direct consequence of  \cite[(2.50)]{Samko1993Book}.
%
%It is known that
%	\begin{equation}\label{LogFun}
%	\ln z= \lim_{\varepsilon\to 0}\frac{z^\varepsilon-1}{\varepsilon},\quad z>0.
%	\end{equation}
%	Thus from \eqref{intformu} and \eqref{LogFun}, we obtain for $\eta>-1$ and $s\ge 0$,
%	\begin{equation*}
%	\begin{split}
%	{}_a I_{x}^s \{(x-a)^\eta  & \ln (x-a)\}=\lim_{\varepsilon\to 0}{}_aI_{x}^s\Big\{\frac{(x-a)^{\eta+\varepsilon}-(x-a)^\eta}{\varepsilon}\Big\}\\
%	& =\lim_{\varepsilon\to 0}\frac{1}{\varepsilon}\Big\{\frac{\Gamma(\eta+\varepsilon+1)}{\Gamma(\eta+\varepsilon+s+1)}(x-a)^{\eta+\varepsilon+s}
%	-\frac{\Gamma(\eta+1)}{\Gamma(\eta+s+1)}(x-a)^{\eta+s}\Big\}\\
%	&=(x-a)^{\eta+s}\lim_{\varepsilon\to 0}\Big\{\frac{\Gamma(\eta+\varepsilon+1)}{\Gamma(\eta+\varepsilon+s+1)}\frac{(x-a)^{\varepsilon}-1}{\varepsilon}\Big\}\\
%	&
%	\quad +(x-a)^{\eta+s}\lim_{\varepsilon\to 0}\frac{1}{\varepsilon}\Big\{\frac{\Gamma(\eta+\varepsilon+1)}{\Gamma(\eta+\varepsilon+s+1)}
%	-\frac{\Gamma(\eta+1)}{\Gamma(\eta+s+1)}\Big\}\\
%	&=
%	\frac{\Gamma(\eta+1)}{\Gamma(\eta+s+1)} \big\{\ln (x-a)+\psi(\eta+1) -\psi(\eta+s+1)\big\}(x-a)^{\eta+s},
%	\end{split}
%	\end{equation*}	
%	where in the last step, we used the L'Hospital's rule.   Thus, the formula \eqref{IntmuGamma} follows.
The property of the $\psi$-function in \eqref{psizfun} can be found in  \cite[(2.2)]{Alzer1997MC}.
Note that we can derive the formula for ${}_x I_{b}^s \{(b-x)^\eta\ln (b-x)\}$ in the same manner.
	\end{proof}

\begin{thm}\label{SpaceForLog}  For any $\alpha\ge 0$ and $\theta\in (-1,1),$ we have
\begin{equation}\label{uexpsolu}
u(x)=|x-\theta|^\alpha\ln|x-\theta|\in {\mathbb W}^{\alpha+1-\epsilon}_\theta(\Omega),\quad \forall\, \epsilon\in (0,1).
\end{equation}
Moreover, we have the following uniform bound of the Chebyshev expansion coefficients:
		\begin{equation}\label{BoundForUnLogSing}
	\begin{split}
	|\hat u_n^{C}|\le & \frac{U_\theta^{[\sigma],\{\sigma\}}}{2^{\sigma-1} \pi}
	\frac{ \Gamma( ({n-\sigma-1})/ 2)}{\Gamma( ({n+\sigma+1})/2)} ,
	\end{split}
	\end{equation}
where $\sigma:=\alpha+1-\epsilon,$ and $|\hat u_n^C|\le Cn^{-\sigma}$ for large $n.$

If  $\theta=0,$ then we have    $\hat u_{2k+1}^C=0,$ and the exact formula:
		 \begin{equation}\label{UncLog-0}
		 \begin{split}
		 \hat u_{2k}^C=&\frac{\Gamma(\alpha+1)} {2^{\alpha-2}\pi}\frac{\Gamma(k-\alpha/2)} {\Gamma(k+\alpha/2+1)}
		 \Big\{	\pi\cos \frac{\alpha \pi} 2	+\sin \frac{\alpha \pi} 2\big(2\psi(\alpha+1) \\
		 &
		 -2\ln 2- \psi(k-\alpha/2)-\psi(k+\alpha/2+1)\big)\Big\},\quad \forall\,  k\in {\mathbb N}_0,
		 \end{split}
		 \end{equation}
which  enjoys  the asymptotic behaviour
		\begin{equation}\label{UncLogasymptotic-0}
\begin{split}
\hat u_{2k}^C=&\frac{\Gamma(\alpha+1)} {2^{\alpha-3}\pi}k^{-\alpha-1}
\Big\{\frac{\pi}2 \cos \frac{\alpha \pi} 2+\sin \frac{\alpha \pi} 2\big(\psi(\alpha+1)-\ln 2-\ln k\big)\Big\}\\
&+O(k^{-\alpha-3}\ln k)\,\sin \frac{\alpha \pi} 2+O(k^{-\alpha-3}),\quad k\gg 1.
\end{split}
\end{equation}		
%	where
%	\begin{equation*}
%	\begin{split}
%	U_{1,s}\le &\frac{2}{(1-s)\Gamma(2-s)}  +
%	\frac{2}{\Gamma(2-s)}\big|\psi(2) -\psi(2-s)+1/(1-s)\big|.
%	\end{split}
%	\end{equation*}
\end{thm}
\begin{proof} Let $m=[\alpha]+1$ and $\nu=\alpha-m.$  We derive from a direct calculation that
\begin{equation}\label{dxPalphalnx}
	\begin{split}
	u^{(k)}(x)= ( {\rm sgn }(x-\theta))^k|x-\theta|^{\alpha-k}\big( d_\alpha^k\ln|x-\theta|+f_\alpha^k), \quad k\ge 0,
			\end{split}
		\end{equation}
	where $d_\alpha^k$ is the same as in \eqref{dkalpha}, and
	$$f_\alpha^k:= \sum_{j=1}^k\frac{(-1)^{j-1} \Gamma(k+1) \Gamma(\alpha+1)}{j\Gamma(k-j+1)\Gamma(\alpha-k+j+1)}.$$	
We see that $u\in L^1(\Omega)$ and  $u, \cdots, u^{(m-1)}\in {\rm AC}(\Omega).$ Next, using Lemma \ref{LmmForLogSingularity},  we obtain that for $x\in (\theta,1),$
	\begin{equation*}
\begin{split}
 {}_{\theta}I_{x}^{1-s} u^{(m)} =&d_\alpha^m
{}_{\theta}I_{x}^{1-s}\big\{  (x-\theta)^{\alpha-m}\ln (x-\theta)\big\}+f_\alpha^m{}_{\theta}I_{x}^{1-s} \big\{(x-\theta)^{\alpha-m}\big\}
 \\
 =&d_\alpha^m \frac{\Gamma(\nu+1)}{\Gamma(\nu+2-s)}
 \ln (x-\theta) (x-\theta)^{\nu+1-s} \\
 &+\frac{\Gamma(\nu+1)}{\Gamma(\nu+2-s)}
 \big(\psi(\nu+1) -\psi(\nu+2-s)+f_\alpha^m\big) (x-\theta)^{\nu+1-s}.
	\end{split}
\end{equation*}
Thus, if $\nu+1-s>0,$ i.e.,  $s<\alpha+1-m,$ then
${}_{\theta}I_{x}^{1-s} u^{(m)}  \in{\rm BV}(\Omega_\theta^+).$
Similarly, under the same condition, we have
	${}_{x}I_{\theta}^{1-s}  u^{(m)}\in{\rm BV}(\Omega_\theta^-).$
By the definition \eqref{fracA}, we obtain $u\in {\mathbb W}^{\mu}_\theta (\Omega),$
where $\mu=m+s<\alpha+1.$ This implies \eqref{uexpsolu}.
The bound in \eqref{BoundForUnLogSing} follows from \eqref{BoundForUnRealB} straightforwardly.

 If $\theta=0,$ then $u(x)$ is an even function, so $\hat u_{2k+1}^C=0.$
  It is known that
	\begin{equation}\label{LogFun}
	\ln z= \lim_{\varepsilon\to 0}\frac{z^\varepsilon-1}{\varepsilon},\quad z>0.
	\end{equation}
		Using    \eqref{LogFun}, we derive from  	\eqref{SpecialCasetheta=01} that
		\begin{equation}\label{UncLog}
		\begin{split}
		\hat u_{2k}^C&= \frac 2\pi  \int_{-1}^1 \Big\{\lim_{\varepsilon\to 0}\frac{|x|^{\varepsilon+\alpha}-|x|^\alpha}{\varepsilon}\Big\}\frac{T_{2k}(x)} {\sqrt{1-x^2}} dx \\
		&=(-1)^{k}\lim_{\varepsilon\to 0}\frac{1}{\varepsilon}\Big\{ \sin \frac{(\varepsilon+\alpha) \pi} 2
		\frac{\Gamma(\varepsilon+\alpha+1)} {2^{\varepsilon+\alpha-1}\pi}\frac{\Gamma(k-(\varepsilon+\alpha)/2)} {\Gamma(k+(\varepsilon+\alpha)/2+1)}
		\\
		&	
		\quad -\sin \frac{\alpha \pi} 2
		\frac{\Gamma(\alpha+1)} {2^{\alpha-1}\pi}\frac{\Gamma(k-\alpha/2)} {\Gamma(k+\alpha/2+1)}\Big\}.
		\end{split}
		\end{equation}
		Noting that
			\begin{equation*}
		\begin{split}
		\frac{d}{d\varepsilon} \Big\{ & \sin \frac{(\varepsilon+\alpha) \pi} 2
		\frac{\Gamma(\varepsilon+\alpha+1)} {2^{\varepsilon}}\frac{\Gamma(k-(\varepsilon+\alpha)/2)} {\Gamma(k+(\varepsilon+\alpha)/2+1)}\Big\}\\
		=&
		\frac{\Gamma(\varepsilon+\alpha+1)} {2^{\varepsilon}}\frac{\Gamma(k-(\varepsilon+\alpha)/2)} {\Gamma(k+(\varepsilon+\alpha)/2+1)}
		\Big\{ \frac{\pi}{2}\cos \frac{(\varepsilon+\alpha) \pi} 2
		+ \sin \frac{(\varepsilon+\alpha) \pi} 2\\
		&\times\big(\psi(\varepsilon+\alpha+1)
		-\ln 2-\psi(k-(\varepsilon+\alpha)/2)/2-\psi(k+(\varepsilon+\alpha)/2+1)/2 \big)\Big\},
		\end{split}
		\end{equation*}
 we obtain \eqref{UncLog-0} from \eqref{UncLog} and   the L'Hospital's rule immediately.
		
		Taking $z=k-\alpha/2$ in \eqref{psizfun}, we obtain
		\begin{equation*}%\label{ShaperBoundPsi0}
		\ln(k-\alpha/2)-\frac 1 {2k-\alpha}<\psi(k-\alpha/2)<\ln (k-\alpha/2)-\frac 1 {k-\alpha/2},
		\end{equation*}
		%	and
		%	\begin{equation}\label{ShaperBoundPsi1}
		%	1/(n+\alpha)< \ln (n/2+\alpha/2)-\psi(n/2+\alpha/2) < 2/(n+\alpha),
		%	\end{equation}
		which implies that for $k\gg 1,$
	\begin{equation}\label{ShaperBoundPsi0-2}
		\psi(k-\alpha/2) = \ln k + O(k^{-1});\quad  		\psi(k+\alpha/2+1)= \ln k + O(k^{-1}).
		\end{equation}
%		Similarly, we have
%	\begin{equation}\label{ShaperBoundPsi0-3}
%		\psi(k+\alpha/2+1)= \ln k + O(k^{-1}),\;\;k\gg 1.
%	\end{equation}
	Using \eqref{gamratioA} leads to
			\begin{equation}\label{ShaperBoundPsi0-4}
		\frac{\Gamma(k-\alpha/2)} {\Gamma(k+\alpha/2+1)}= k^{-\alpha-1}\big(1+ O(k^{-2})\big),\;\;k\gg 1.
		\end{equation}
From \eqref{UncLog-0} and \eqref{ShaperBoundPsi0-2}-\eqref{ShaperBoundPsi0-4}, we obtain
		\eqref{UncLogasymptotic-0}.
\end{proof}

\begin{rem}\label{alphasingul} {\em Consider the Chebyshev expansion of $u=|x|^\alpha \ln |x|,$
we  observe from \eqref{UncLogasymptotic-0} that for $n\gg 1,$ $|\hat u_{n}^C|\le C(\ln n)n^{-(\alpha+1)}.$ Therefore, we obtain directly the optimal estimates:
\begin{equation}\label{modeeqnA}
\begin{split}
&\|u-\pi_N^C u\|_{L^\infty(\Omega)}\le \sum_{n=N+1}^\infty |\hat u_n^C| \le C(\ln N) N^{-\alpha};  \quad
 \|u-\pi_N^C u\|_{L^2_\omega(\Omega)} \le C(\ln N) N^{-\alpha-1/2}.
\end{split}
\end{equation}
However, we find from \eqref{uexpsolu} that the space ${\mathbb W}^{\alpha+1-\epsilon}_\theta(\Omega)$ is suboptimal to characterize this type of singularity.  Indeed, by Theorem \ref{TruncatedChebyshevLinf}, we only have $\|u-\pi_N^C u\|_{L^\infty(\Omega)}=O(N^{\epsilon-\alpha})$ and $ \|u-\pi_N^C u\|_{L^2_\omega(\Omega)}=O(N^{\epsilon-\alpha-1/2}).$
The situation is reminiscent of the Besov framework in \cite{Bab.G00}, where the spaces of Type-I and Type-II are defined through different space interpolation.    The question of how to modify the fractional space  to best characterise  Type-II singularity in our setting appears nontrivial and is still open.
%
%Compared with Type-I, the singular function of Type-II possesses a slightly stronger singularity with an $\epsilon$-difference in this fractional framework. It should be pointed out that the bound \eqref{BoundForUnLogSing} is nearly optimal, as one would expect $O(n^{-\alpha-1}\ln n)$ (against $O(n^{\epsilon-(\alpha+1)})$ in the above).  Below, we just justify  this for $\theta=0,$
%but for $\theta\not=0,$  it seems  the derivation of the exact formula is much more involved.
%Interestingly, if $\alpha$ is an even integer, we observe from \eqref{UncLogasymptotic-0} that the convergence rate is
%$O(k^{-\alpha-1}),$ rather than $O(k^{-\alpha-1}\ln k)$ for other values of $\alpha.$  \qed
}\end{rem}

%Indeed,  from the  study of the special case with $\theta=0$ below, we expect $|\hat u_n^C|=O(n^{-(\alpha+1)}\ln n),$
%% but  Proposition \ref{Wspacproof} implies $|\hat u_n^C|=O(n^{\epsilon-(\alpha+1)}).$ 	
%	\begin{prop}\label{RemforLog}  Consider  $u(x)=|x|^\alpha\ln|x|$. Then   $\hat u_{2k+1}^C=0,$ and
%		 \begin{equation}\label{UncLog-0}
%		 \begin{split}
%		 \hat u_{2k}^C=&\frac{\Gamma(\alpha+1)} {2^{\alpha-2}\pi}\frac{\Gamma(k-\alpha/2)} {\Gamma(k+\alpha/2+1)}
%		 \Big\{	\pi\cos \frac{\alpha \pi} 2	+\sin \frac{\alpha \pi} 2\big(2\psi(\alpha+1) \\
%		 &
%		 -2\ln 2- \psi(k-\alpha/2)-\psi(k+\alpha/2+1)\big)\Big\},\quad \forall\,  k\in {\mathbb N}_0.
%		 \end{split}
%		 \end{equation}
%		Moreover,   we have  the asymptotic behaviour
%		\begin{equation}\label{UncLogasymptotic-0}
%\begin{split}
%\hat u_{2k}^C=&\frac{\Gamma(\alpha+1)} {2^{\alpha-3}\pi}k^{-\alpha-1}
%\Big\{\frac{\pi}2 \cos \frac{\alpha \pi} 2+\sin \frac{\alpha \pi} 2\big(\psi(\alpha+1)-\ln 2-\ln k\big)\Big\}\\
%&+O(k^{-\alpha-3}\ln k)\,\sin \frac{\alpha \pi} 2+O(k^{-\alpha-3}),\quad k\gg 1.
%\end{split}
%\end{equation}	
%	\end{prop}
%	\begin{proof}			%\cite[Table 1]{Boyd2009JEM}
%		\end{proof}		

\section{Improving existing results}\label{sect:existingest}
\setcounter{equation}{0}
\setcounter{lmm}{0}
\setcounter{thm}{0}
\setcounter{cor}{0}

In this section,  we show that the previous estimates with $s\to 1$    improve   the existing results  on Chebyshev approximations (see, e.g.,  \cite{Trefethen2008SIREV,Xiang2010NM,Trefethen2013Book,Majidian2017ANM}).
 %of functions in  ${\mathbb W}^{m+1}(\Omega)$ (with $m\in {\mathbb N}_0$) defined in  \eqref{intergerASobB}.

\subsection{Existing estimates}
%Define the Sobolev-type space (cf. \cite{Trefethen2008SIREV}):
%\begin{equation}\label{intergerASob}
%\begin{split}
%{\mathbb  W}^{m+1}_T(\Omega):= \big\{ &u\in L^1(\Omega)\, :\,  u,   u',\cdots, u^{(m-1)}\in {\rm AC}(\Omega), \; u^{(m)}\in  {\rm BV}(\Omega)\big\},
%	\end{split}
%\end{equation}
%where $m\in {\mathbb N}_0$ and $\omega=(1-x^2)^{-1/2}$ is the Chebyshev weight function.
As in \cite{Trefethen2008SIREV}, let $\|\cdot\|_T$ be the Chebyshev-weighted $1$-norm: % defined by
\begin{equation}\label{cheby1norm}
\|u\|_T=\Big\|\frac{u'(x)}{\sqrt{1-x^2}}\Big\|_1, %\int_{-1}^{1}{\frac{|u'(x)|}{\sqrt{1-x^2}}}dx, %=\|u'\|_{L^1_{\omega}(\Omega)}.
\end{equation}
which  is defined via a Stieltjes integral for any $u$ of bounded variation.
%The following estimate of the decay rate of $\{a_n\}$ can be found in \cite{Trefethen2008SIREV}.
\begin{lemma}\label{trenstha}{\rm\bf(see  \cite[Thms 4.2-4.3]{Trefethen2008SIREV}).}
If $u, u',\cdots, u^{(m-1)}$ are absolutely continuous on $[-1,1],$ and if $\|u^{(m)}\|_T=V_T<\infty$ with integer $m\ge 0,$
then for  each $n\ge m+1, $
\begin{equation}\label{TrestA}
\big|\hat u_n^C\big|\le \frac{2\,V_T}{\pi n(n-1)\cdots (n-m)},
\end{equation}
and for integer $m\ge 1,$ and integer  $N\ge m+1,$
\begin{equation}\label{estTrestA}
\big\|u-\pi_N^C u\big\|_{L^\infty(\Omega)}\le \frac{2\,V_T}{\pi m\, (N-m)^m}.
\end{equation}
%where $V_T=\|u^{(m)}\|_T.$ %=\|u^{(m+1)}\|_{L^1_{\omega}(\Omega)}.$
\end{lemma}
We remark that the Chebyshev weight is removed in  Trefethen \cite[Thms 7.1-7.2]{Trefethen2013Book}, i.e.,
$V_T$ is replaced by the total variation of $u^{(m)}.$
%where the proof is different (using the   contour integral representation of $\hat u_n^C$).

 Following the argument of summation by  certain telescoping series in \cite{Xiang2010NM},  Majidian (cf. \cite[Thm 2.1]{Majidian2017ANM}) derived sharper bounds. For comparison, we quote the estimates therein below.
\begin{lemma}\label{trensthaImp}{\rm\bf(see \cite[Thm 2.1]{Majidian2017ANM}).}
If $u, u',\cdots, u^{(m-1)}$ are absolutely continuous on $[-1,1],$ and if $\|u^{(m)}\|_T=V_T<\infty$ with integer $m\ge 0,$
then for  each $n\ge m+1, $
\begin{equation}\label{TrestAimp0}
\big|\hat u_n^C\big|\le  \frac{2\,V_T}{\pi}
\displaystyle\prod\limits_{j=0}^m \dfrac 1 {n-m+2j}\,.
\end{equation}
%where $V_T=\|u^{(m+1)}\|_{L^1_{\omega}(\Omega)}.$
\end{lemma}

%\begin{rem}\label{imprvos} It should be pointed out that the original bounds  in \cite{Majidian2017ANM} were stated as
%\begin{equation}\label{TrestAimp}
%\big|\hat u_n^C\big|\le  \frac{2\,V_T}{\pi}\begin{cases}
%\displaystyle\prod\limits_{l=-k}^k \dfrac 1 {n+2l}\,,\quad & {\rm if}\;\; m=2k,\\[10pt]
%\displaystyle \prod\limits_{l=-k}^{k+1} \dfrac 1 {n+2l-1}\,,\quad & {\rm if}\;\; m=2k+1.
%\end{cases}
%\end{equation}
%In fact, letting $j=l+k,$ we can rewrite the bounds in \eqref{TrestAimp} as \eqref{TrestAimp0} simply. \qed
%\end{rem}
\subsection{Improved estimates}
\begin{theorem}\label{DecayRate}
	  %Let $u\in {\mathbb  W}^{m+1}(\Omega)$ with integer $m\ge 0,$ and denote $V_L:=\|u^{(m+1)}\|_{L^1(\Omega)}.$
%then for  each $n\ge m+1, $
	  % Then we have the following estimates.
Suppose that for integer $m\ge 0,$  $u, u',\cdots, u^{(m-1)}$ are absolutely continuous on $[-1,1],$ and  $u^{(m)}$ is of bounded variation with the total variation denoted by $V_L^{(m)}.$
	\begin{itemize}
		\item[(i)] If $n\ge m+1$  and  $n-m$ is odd,   then
		\begin{equation} \label{anbound0}
			\big|\hat u_n^{C}\big| \le  %\frac 2 {\pi} \frac{V_L} {(n-m)\cdot (n-m+2)\cdots  (n+m)}.
 \frac {2\,  V_L^{(m)}}\pi \prod_{j=0}^m \frac{1}{n-m+2j} \,.        %\frac{(2p-1)!!}{(2m+2p+1)!!} \,  {V_L}.
		\end{equation}
		\item[(ii)]If $n\ge m+1$  and  $n-m$ is even,   then
		\begin{equation} \label{anbound00}
			\big|\hat u_n^{C}\big| \le
 %\frac  2\pi  \frac 1 {\sqrt{n^2-m^2}} \frac{V_L} {(n-m+1)\cdot (n-m+3)\cdots (n+m-1)}.
			  \frac {2\, V_L^{(m)}} { \pi \sqrt{n^2-m^2}}  \prod_{j=0}^{m-1} \frac{1}{n-m+2j-1} \,.   %  \frac{p+1/2}{\sqrt{(p+1)(m+p+1)}} \frac{\,(2p-1)!!}{(2m+2p+1)!!} \,  {V_L}.
		\end{equation}
		\item[(iii)] If $0\le n\le m+1,$ then
		\begin{equation} \label{anbound1}
	|\hat u_n^{C}|\le   \frac {2\, V_L^{(n)}}  {\pi (2n-1)!!}\,.
		\end{equation}
	\end{itemize}
\end{theorem}
\begin{proof}   %For any $u\in {\mathbb  W}^{m+1}(\Omega), $
We find from    \eqref{HatUnCaseC} (or \eqref{byparts0} with one more step of integration by parts) that  for $n\ge m+1,$
	\begin{equation}\label{HatUnCaseA}
	\begin{split}
	\hat u_n^{C}&= \frac 1 {(2m+1)!!}  \frac 2\pi \int_{-1}^{1} \, G_{n-m-1}^{(m+1)}(x) \omega_{m+1}(x) \,d u^{(m)}(x).
	\end{split}
	\end{equation}
Thus, by  \eqref{HatUnCaseA},
\begin{equation}\label{byparts001}
	\begin{split}
	\big|\hat u_n^{C}\big|& \le  \frac {V_L^{(m)}} {(2m+1)!!}  \frac 2\pi  \max_{|x|\le 1}\big\{\omega_{m+1}(x) |G_{n-m-1}^{(m+1)}(x) |\big\}.
\end{split}
	\end{equation}
If $n=m+2p+1$ with $p\in {\mathbb N}_0$, we derive from \eqref{polynomialBnd} with $l=p$ and $\lambda=m+1$ that
\begin{equation}\label{BoundGeg-2}
\begin{split}
\max_{|x|\le 1}\big\{\omega_{m+1}(x)  |G_{n-m-1}^{(m+1)}(x) |\big\}&\le \frac{\Gamma(m+3/2)\Gamma(p+1/2)}{\sqrt \pi\, \Gamma(m+p+3/2)}
=\frac{(2m+1)!!\,(2p-1)!!}{(2m+2p+1)!!}.
%=  \frac{(2m+1)!!\,(n-m-2)!!}{(n+m)!!},
%= \pi^{-1/2}\Gamma(m+\frac 1 2)
%			\frac{\Gamma(\frac {n} 2-\frac {m} 2+\frac 1 2)}{\Gamma(\frac {n} 2+\frac {m} 2+\frac 1 2)}	
\end{split}
\end{equation}
Consequently, for  $n=m+2p+1$ with $p\in {\mathbb N}_n$,  we obtain from \eqref{byparts001}-\eqref{BoundGeg-2} that
		\begin{equation} \label{anbound0p}
\begin{split}
			\big|\hat u_n^{C}\big| &\le    \frac 2\pi     \frac{(2p-1)!! \,  {V_L^{(m)}}}{(2p+2m+1)!!} =
\frac 2 {\pi} \frac{V_L^{(m)}} {(2p+1)\cdot (2p+3)\cdots (2p+2m+1)}\\
&= \frac 2 {\pi} \frac{V_L^{(m)}} {(n-m)\cdot (n-m+2)\cdots  (n+m)},
 \end{split}
		\end{equation}
which implies  \eqref{anbound0}.

Similarly, if $n=m+2p+2$ with $p\in {\mathbb N}_0$, we derive from \eqref{polynomialBnd2} with $l=p$ and $\lambda=m+1$ that
\begin{equation}\label{BoundGeg-22}
\begin{split}
\max_{|x|\le 1}\big\{\omega_{m+1}(x)  |G_{n-m-1}^{(m+1)}(x) |\big\}&\le \frac{1}{\sqrt{(2p+2)(2m+2p+1)}} \frac{(2m+1)!!\,(2p+1)!!}{(2m+2p+1)!!},
%= \pi^{-1/2}\Gamma(m+\frac 1 2)
%			\frac{\Gamma(\frac {n} 2-\frac {m} 2+\frac 1 2)}{\Gamma(\frac {n} 2+\frac {m} 2+\frac 1 2)}	
\end{split}
\end{equation}
so by  \eqref{byparts001}, we have
\begin{equation} \label{anbound0pp}
\begin{split}
\big|\hat u_n^{C}\big| &\le
 %\frac 2\pi    \frac{V_L}{\sqrt{(2p+2)(2m+2p+1)}} \frac{(2p+1)!!}{(2m+2p+1)!!}    =
\frac 2\pi \frac 1{\sqrt{(2p+2)(2m+2p+1)}} \frac {V_L} {(2p+3)\cdot (2p+5)\cdots (2p+2m+1)}\\
&= \frac  2\pi  \frac 1 {\sqrt{n^2-m^2}} \frac{V_L} {(n-m+1)\cdot (n-m+3)\cdots (n+m-1)}.
\end{split}
		\end{equation}
This leads to \eqref{anbound00}.

  In case of  $0\le n\le m+1$,  we derive  from
\eqref{HatUnCaseA} (with $n=m+1$) and the factor $G_{0}^{(n)}(x)\equiv 1$ that
\begin{equation}\label{DecayRate-4}
\begin{split}
\hat u_n^{C}
%=&\frac{1}{2^{n-1}\Gamma(n+1/2)\sqrt{\pi}}\int_{-1}^{1} u^{(n)}(x)\, G_{0}^{(n)}(x) \omega_{n}(x) \,dx\\
=&\frac 1 {(2n-1)!!}  \frac 2\pi \int_{-1}^{1} \, \omega_{n}(x)\,du^{(n-1)}(x).
\end{split}
\end{equation}
Then  we obtain \eqref{anbound1} immediately.
	\end{proof}

Next, we unify the bounds in (i)-(ii) of Theorem \ref{DecayRate}  without loss of the rate of convergence. In fact, this relaxation  leads to the estimate
\eqref{TrestAimp0} in \cite[Thm 2.1]{Majidian2017ANM}, but with $V_L$ in place of $V_T.$ In other words,  the bounds in Theorem \ref{DecayRate} indeed improve the best available results.
\begin{cor}\label{unChat}  Under the same conditions as in Theorem \ref{DecayRate},
%Let $u\in {\mathbb  W}^{m+1}(\Omega)$ with integer $m\ge 0,$ and denote $V_L=\|u^{(m+1)}\|_{L^1(\Omega)}.$
 we have  that  for all $n\ge m+1,$
	\begin{equation}\label{SharpBoundUn}
	|\hat u_n^{C}| \le  \frac {2\,  V_L^{(m)}}\pi \prod_{j=0}^m \frac{1}{n-m+2j}.
 %\frac 2 {\pi}	\frac{V_L} {(n-m) (n-m+2)\cdots  (n+m)}.
	\end{equation}
\end{cor}	
\begin{proof}
It is evident that by \eqref{anbound0}-\eqref{anbound00}, we only need to  prove this bound for  $n-m$ being even.  One verifies readily  the fundamental inequality:
	$$n^2-(p-1)^2\ge\sqrt{(n^2-p^2)(n^2-(p-2)^2)}, \quad{\rm for}\;\;\;  2\le p\le n.$$
If $m$ is even, we can pair up the factors and use the above inequality with $p=m,m-2,\cdots, 2$ to derive
	\begin{equation}\label{termssq}
	\begin{split}
	 (n- & m+1)  (n-m+3) \cdots (n-1)(n+1)  \cdots (n+m-3) (n+m-1)\\
	 & =(n-(m-1)^2) (n^2-(m-3)^2)\cdots (n^2-1) \\
	&\ge \sqrt{n^2-m^2}\, \sqrt{n^2-(m-2)^2}\, \sqrt{n^2-(m-2)^2}\, \sqrt{n^2-(m-4)^2}\, \cdots
	\\& =\sqrt{n^2-m^2}\,  (n-m+2)\cdots  (n+m-2).
	\end{split}
	\end{equation}
Similarly, if $m$ is odd, we remain the middle most factor intact and pair up the factors to derive the above.
Therefore,  multiplying  both sides of \eqref{termssq} by $\sqrt{n^2-m^2},$  we obtain
	\begin{equation}\label{Compare-1}
	%\frac{1}{\sqrt{n^2-m^2}\, (n-m+1) (n-m+3)\cdots (n+m-1)}\le \frac{1}{(n-m)(n-m+2)\cdots  (n+m)}.
\frac{1}{\sqrt{n^2-m^2}}\, \prod_{j=0}^{m-1} \frac{1}{n-m+2j-1}\le \prod_{j=0}^m \frac{1}{n-m+2j}.	
\end{equation}
Then \eqref{SharpBoundUn}  follows from   \eqref{Compare-1} and  (i)-(ii) of Theorem \ref{DecayRate} directly.
\end{proof}

To show the sharpness of our improved bounds, we consider $u= |x-\theta|$, $\theta\in (-1,1)$ to compare upper bounds of $\hat u_n^{C}$.
In this case, we have $m=1$,
$u''=2\delta(x-\theta)$, $V_L^{(1)}=2$ and $V_T=2(1-\theta^2)^{-1/2}.$ Let  ${\rm Ratio}_1$ and ${\rm Ratio}_2$ be the ratios of upper bounds in \cite{Trefethen2013Book,Majidian2017ANM} (cf. \eqref{TrestA} with $V_T$ being replaced by the bounded variation of $u',$ and the bound in  \eqref{TrestAimp0}) and our improved bound in Theorem \ref{DecayRate}, respectively. In Figure \ref{FigForCompare}, we depict two ratios against various $n$ for two values of $\theta.$
We see that the improve bound is sharper than the existing ones, and the removal of the Chebyshev weight in $V_T$ is also significant for  the sharpness of the bounds.
\vskip 5pt
\begin{figure}[!ht]
	\begin{center}
		{~}\hspace*{-20pt}	 \includegraphics[width=0.42\textwidth,height=0.25\textwidth]{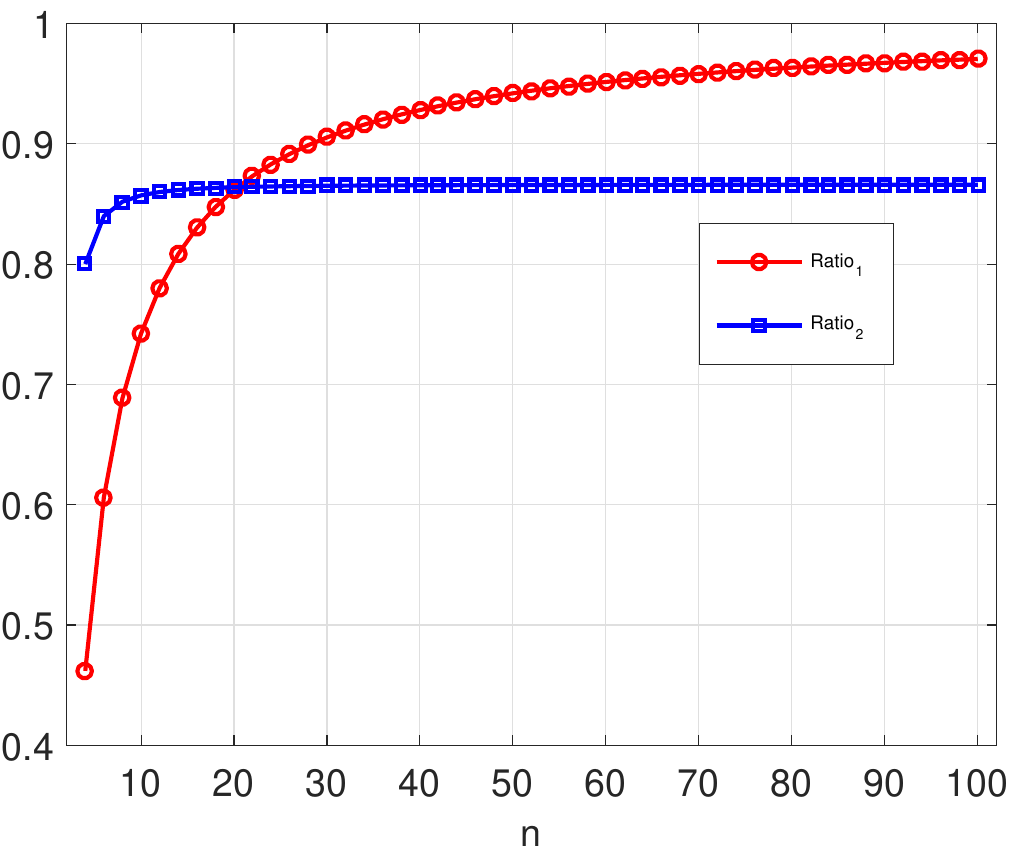}\qquad
		\includegraphics[width=0.42\textwidth,height=0.25\textwidth]{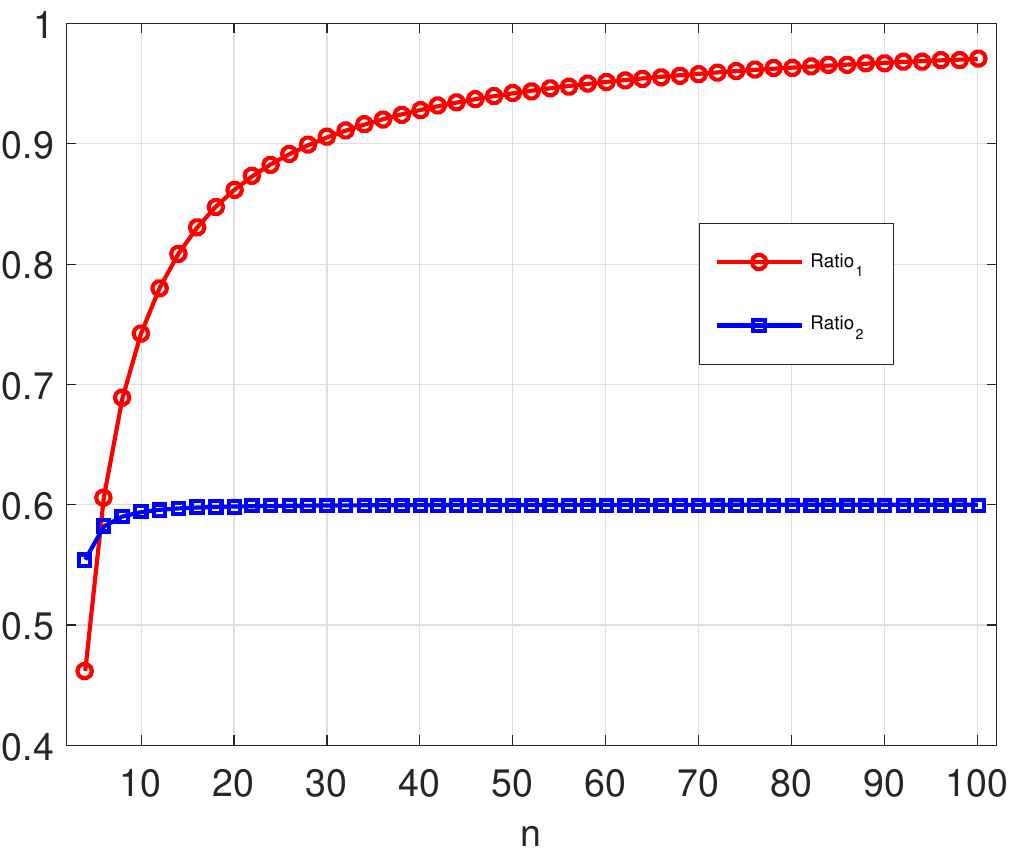}
		\caption{Ratios of the existing bounds and improved bound herein for  $u= |x-\theta|$ and  $\theta\in (-1,1).$      Left: $\theta=1/2.$ Right: $\theta=4/5$.}
		\label{FigForCompare}
	\end{center}
\end{figure}

To conclude this section, we  state below the improved $L^\infty$-estimates, and remark on the improvements  in Remark \ref{differentNote} below.
\begin{thm}\label{TruncatedChebyshev}  Let $u\in {\mathbb  W}^{m+1}(\Omega)$ with integer $m\ge 0.$
		\begin{itemize}
\item[(i)] If $1\le m\le N,$ then
		\begin{equation} \label{ProjectionErrorA}
		\|u-\pi_N^C u\|_{L^\infty(\Omega)}\le \frac {2 }  {m\pi }\bigg( \prod_{j=1}^m \frac {1} {N-m+2j-1}\bigg) V_L^{(m)}.
%\big \|u^{(m+1)}\big\|_{L^1(\Omega)}.
		\end{equation}
		\item[(ii)] If $m=0,$ then for all integer $N\ge 1,$
		\begin{equation}\label{upiNm0}
\big\|u-\pi_N^C u\big\|_{L^\infty(\Omega)}\le  V_L^{(0)}.  %\big \|u'\big\|_{L^1(\Omega)}.
\end{equation}
	\item[(iii)] If $m\ge N+1,$ then
		\begin{equation} \label{ProjectionErrorB}
		\|u-\pi_N^C u\|_{L^\infty(\Omega)} \le \frac {2}  { (2N+1)!!\,\pi} \sum_{n=N}^{m}   c_n\frac{(2N+1)!! } {(2n+1)!!} V_L^{(n)}, %\|u^{(n+1)}\|_{L^1(\Omega)}, 				
		\end{equation}
		where $c_n=1$ for all $N\le n\le m-1$ and $c_m=2.$
	\end{itemize}
\end{thm}
\begin{proof}  From  Theorem \ref{TruncatedChebyshevLinf} with $s\to 1$  and \eqref{gammn12},   we obtain that for $1\le m\le N+1,$
\begin{equation}\label{FracLinfA00}
		\begin{split}
		\|u-\pi_N^C u\|_{L^\infty (\Omega)}&\le \frac {1 }{2^{m-1} m \pi}\frac{ \Gamma( ({N-m+1})/2)}{\Gamma( ({N+m+1})/2)} V_L^{(m)} %\big \|u^{(m+1)}\big\|_{L^1(\Omega)}
\\
		&= \frac {2 }{ m \pi}\frac{ (N-m-1)!!}{(N+m+1)!!} V_L^{(m)}  %\big \|u^{(m+1)}\big\|_{L^1(\Omega)}
		=\frac {2 }  {m\pi }\bigg( \prod_{j=1}^m \frac {1} {N-m+2j-1}\bigg) V_L^{(m)}. % \big \|u^{(m+1)}\big\|_{L^1(\Omega)}.
		\end{split}
		\end{equation}
This gives \eqref{ProjectionErrorA}.
We now  prove \eqref{upiNm0}.
Using integration part parts leads to
\begin{equation}\begin{split}\label{unform=0}
&(u-\pi_N^C u)(x)  =\sum_{n=N+1}^{\infty}
\hat u_n^{C}T_n(x)=\sum_{n=N+1}^{\infty}\! \Big(\int_{0}^{\pi}u(\cos\varphi)\cos(n\varphi)d\varphi\Big) \cos(n\theta)\\
& =\frac{2}{\pi}\sum_{n=N+1}^{\infty} \Big(\int_{0}^{\pi}u'(\cos\varphi)\sin(n\varphi)\sin \varphi d\varphi \Big) \frac{\cos(n\theta)}{n} =\frac{2}{\pi}\int_{0}^{\pi}u'(\cos\varphi)\sin(n\varphi)\, \Psi^{\infty}_{N}(\varphi,\theta)\, d\varphi,
\end{split}\end{equation}
so we have
\begin{equation}\begin{split}\label{unform=0+}
\big|(u-\pi_N^C u)(x)\big|\le&
\frac{2}{\pi} \max_{\varphi\in [0,\pi]} \big|\Psi^{\infty}_{N}(\varphi,\theta)\big|\int_0^\pi |u'(\cos\varphi)| \sin \varphi \,  d\varphi
=\frac{2}{\pi} \max_{\varphi\in [0,\pi]}\big|\Psi^{\infty}_{N}(\varphi,\theta)\big|\, V_L^{(0)},  %\|u'\|_{L^1(\Omega)},
\end{split}\end{equation}
for $x=\cos \theta,$  $\theta\in (0,\pi),$  where
\begin{equation}\label{psiinfty}
\Psi^{\infty}_{N}(\varphi,\theta)=\sum_{n=N+1}^{\infty}\!\!\frac{\sin(n\varphi) \cos(n\theta)}{n}
=\sum\limits_{n=N+1}^{\infty}\!\!\frac{\sin(n(\varphi+\theta))+\sin (n(\varphi-\theta))}{2n}.
% \Psi^{M}_{N}(\varphi,\theta)=\sum_{n=N+1}^{M}\frac{1}{n}\sin(n\varphi) \cos(n\theta).
%=\sum_{k=1}^{M}\frac{1}{N+k}\sin\big((N+k)\vartheta\big) \cos\big((N+k)\theta\big).
\end{equation}

We next show that  for  $\vartheta\in {\mathbb R},$
\begin{equation}\label{ForSumSin-0}
\bigg|\sum\limits_{n=N+1}^{\infty}\!\!\frac{\sin(n\vartheta)}{n}\bigg| \le \frac{\pi}{2}.	
\end{equation}
In fact,  it suffices to derive this bound for $\vartheta\in (0,\pi),$ as the series  defines an odd, $2\pi$-periodic  function which vanishes at $\vartheta=0,\pi.$
It is known that
\begin{equation}\label{convfseries}
\sum _{n=1}^\infty\frac{\sin (n\vartheta)}{n}=\frac{\pi-\vartheta}{2}, \quad \vartheta\in (0,\pi).
\end{equation}
According to \cite{Alzer2003MPCPS},  we have that for $N\ge 2, $
\begin{equation}\label{convfseries2}
0< \sum _{n=1}^N \frac{\sin (n\vartheta)}{n}\le  \alpha (\pi-\vartheta), \quad \vartheta\in (0,\pi),
\end{equation}
with the best possible constant $\alpha=0.66395\cdots.$  As a direct consequence of \eqref{convfseries}-\eqref{convfseries2},  we have
\begin{equation}\label{nN1vars}
\Big(\frac 1 2 -\alpha\Big)(\pi-\vartheta)\le \sum\limits_{n=N+1}^{\infty}\!\!\frac{\sin(n\vartheta)}{n}  <  \frac{\pi-\vartheta}{2};\quad
\bigg|\sum\limits_{n=N+1}^{\infty}\!\!\frac{\sin(n\vartheta)}{n} \bigg| <  \frac{\pi-\vartheta}{2},
\end{equation}
%and
%\begin{equation}\label{nN1vars}
%\bigg|\sum\limits_{n=N+1}^{\infty}\!\!\frac{\sin(n\vartheta)}{n} \bigg| <  \frac{\pi-\vartheta}{2},
%\end{equation}
for $N\ge 2$ and $\vartheta\in (0,\pi).$ In fact,  the bound \eqref{nN1vars} also holds for  $N=1,$   as by \eqref{convfseries},
$$
\sum _{n=2}^\infty\frac{\sin (n\vartheta)}{n}=\frac{\pi-\vartheta}{2}-\sin\vartheta< \frac{\pi-\vartheta}{2}.
$$
Thus, we complete the proof of \eqref{ForSumSin-0}.  The estimate \eqref{upiNm0} is a direct consequence of \eqref{unform=0+}-\eqref{ForSumSin-0}.

	Finally, we turn to the proof of the estimate \eqref{ProjectionErrorB}.    For $m\ge N+1$,  we use  \eqref{anbound1} to bound
	$\{\hat u_n^C\}_{n=N+1}^m,$ and use  \eqref{ProjectionErrorA} (with $N\to m$)  to derive
	\begin{equation}\label{SumUncNPlusOneToInf-2}
	\begin{split}
\big|u(x)&-\pi_N^C u(x)\big|\le \big|\pi_m^C u(x) - \pi_N^C u(x)\big| + \big|u(x)-\pi_m^C u(x)\big|\\
&
\le    \sum_{n=N+1}^{m} \frac {2}  {\pi (2n-1)!!}\,V_L^{(n-1)} %\|u^{(n)}\|_{L^1(\Omega)}
 + \frac{2} {m(2m-1)!!\pi} V_L^{(m)} %\|u^{(m+1)}\|_{L^1(\Omega)}
\\& \le  \frac {2}  {\pi (2N+1)!!} \bigg\{\sum_{n=N+1}^{m}  \frac{(2N+1)!!} {(2n-1)!!} V_L^{(n-1)} %\|u^{(n)}\|_{L^1(\Omega)}
+ \frac{2m+1} m   \frac{(2N+1)!!} {(2m+1)!!} V_L^{(m)} % \|u^{(m+1)}\|_{L^1(\Omega)}
\bigg\}\\
&\le  	 \frac {2}  {\pi (2N+1)!!} \sum_{n=N}^{m}   \frac{c_n(2N+1)!! } {(2n+1)!!} V_L^{(n)}, % \|u^{(n+1)}\|_{L^1(\Omega)}, 				
\end{split}\end{equation}
 where $c_n=1$ for all $N\le n\le m-1$ and $c_m=2.$
 \end{proof}

\begin{rem}\label{differentNote} {\em Taking a different route,  we improve the existing bounds in the following senses.
\begin{itemize}
\item[(i)] The Chebyshev-weighted $1$-norm in Lemma \ref{trensthaImp}  is replaced  by the Legendre-weighted $1$-norm.
%, i.e., the  usual $L^1$-norm.
\item[(ii)] Sharper  bound is obtained than the best one in \cite[Thm 2.1]{Majidian2017ANM}.
\item[(iii)] We obtain the ``stability" result, that is, $m=0$  in   \eqref{estTrestA}, and  the estimate for the case $n\le m+1$ in
\eqref{anbound1}, which are new.
\end{itemize} }
\end{rem}

%\vskip 10pt

\section{Analysis of interpolation, quadrature and endpoint singularities}\label{Sect6Analysis}
\setcounter{equation}{0}
\setcounter{lmm}{0}
\setcounter{thm}{0}
\setcounter{cor}{0}

%Section \ref{mainsect:ms}

In this section, we discuss the extension of the main results to the error estimates of the related interpolation, quadratures and also special types of functions with endpoint singularities. We then conclude the paper with some final remarks.

\subsection{Analysis of   interpolations and quadrature}

As remarked in  \cite{Xiang2010NM,Trefethen2013Book,Majidian2017ANM}, the error analysis of several  widely-used interpolations and quadrature boils down to estimating the coefficients $\{\hat u_n^C\}$ and their partial sums.  We refer to
\cite{Majidian2017ANM} for a list of more than six examples. Here, we just examine two cases  and present  sharp bounds by using our
new estimates on  the  decay of the expansion coefficients.
\begin{itemize}
\item[(i)] Interpolation and quadrature at Chebyshev-Gauss  (CG) points $\{x_j\}_{j=0}^N$, i.e., zeros of $T_{N+1}(x):$
\begin{equation}\label{qnqnCG}
\begin{split}
& ({\mathcal I}_N^{C}u)(x)=\sum_{n=0}^N{'} b_n T_n(x),
\quad  b_n={\displaystyle
	\frac{2}{N+1}\sum_{j=0}^Nu(x_j)T_n(x_j)},
\end{split}	
\end{equation}
and
\begin{equation}\label{quadrCC}
 \int^1_{-1}u(x ) (1-x^2)^{-1/2}dx=\frac \pi {N+1}\sum_{j=0}^Nu(x _j) +\mathcal R_N^C[u].
\end{equation}
Then we have (cf. \cite{Xiang2010NM} and \cite[(6)]{Riess1969SINUM})
\begin{equation}\label{caseiA}
\| {\mathcal I}_N^{C}u-u\|_{L^\infty(\Omega)}\le 2\sum_{n=N+1}^\infty |\hat u_n^C|; \quad
\mathcal R_N^C[u]= \pi \sum_{k=1}^{\infty}(-1)^k\hat u_{2k(N+1)}^{C}.
\end{equation}
\item[(ii)] Legendre-Gauss  quadrature rule at the zeros $\{x_j\}_{j=0}^N$ of the Legendre polynomial $P_{N+1}(x)$ and   with quadrature weights $\{\omega_j\}_{j=0}^N$ (cf.  \cite[P. 96]{ShenTangWang2011,Majidian2017ANM}):
\begin{equation}\label{quadrGL}
 \int^1_{-1}u(x ) dx=\sum_{j=0}^N u(x _j)\omega_j +\mathcal R_N^L[u].
\end{equation}
Then we have (cf.  \cite{Trefethen2008SIREV,Majidian2017ANM}):
\begin{equation}\label{RnLu}
\big|\mathcal R_N^L[u]\big|\le \frac {32}{15} \sum_{n=N+1}^\infty |\hat u_{2n}^C|.
\end{equation}

\end{itemize}
%for the Chebyshev-Gauss interpolation and quadrat

Using Theorem \ref{IdentityForUn}  and the argument similar to Theorem \ref{TruncatedChebyshevLinf} (also see Remark \ref{expbounds}), we can obtain the following estimates. % Here, we omit the details of the proof.
\begin{theorem}\label{InterquadLGC} Given $\theta\in (-1,1),$ if  $u\in {\mathbb  W}^{m+s}_{\theta}(\Omega) $ with
$s\in (0,1)$ and integer  $m\ge 0$, then for $m+s>1,$ we have
\begin{equation}\label{UnUnc}
\begin{split}
& \|u-{\mathcal I}_N^C u\|_{L^\infty (\Omega)}\le CN^{1-m-s} U_\theta^{m,s}; \quad  \|u-{\mathcal  I}_N^C u\|_{L^2_\omega (\Omega)}\le CN^{\frac 1 2-m-s} U_\theta^{m,s},
\end{split}
\end{equation}
and
\begin{equation}\label{CGInterpolant-7}
\begin{split}
&|\mathcal R_N^C[u]|\le C N^{-(m+s)} U_\theta^{m,s};\quad  |\mathcal R_N^L[u]|\le C N^{-(m+s)} U_\theta^{m,s},
\end{split}
\end{equation}
where $C$ is  a positive constant independent of $N$ and $u.$
\end{theorem}
\begin{proof} We just provide the proof of the $L^2_\omega$-error of the CG interpolation, since the others can be proved by summing up the bounds of $\{\hat u_n^C\}$ in Theorem \ref{IdentityForUn} and Remark \ref{expbounds}.
Note that
\begin{align}\label{CGInterpolant-2}
{\mathcal  I}_N^C u (x)-u(x)={\mathcal  I}_N^C u (x)-\pi_N^C u+\pi_N^C u-u
%=\sum_{n=0}^N{'}(b_n-\hat u_n^{C})T_n(x)-\sum_{n=N+1}^\infty \hat u_n^{C}T_n(x)
=
\sum_{n=0}^N{'}(b_n-\hat u_n^{C})T_n(x)+\pi_N^C u-u.
\end{align}
Hence, we obtain
%\begin{equation}\label{CGInterpolant-3}
%\begin{split}
%\|u-{\mathbb  I}_N u\|_{L^\infty (\Omega)}\le \sum_{n=0}^N{'}|b_n-\hat u_n^{C}|+\|u-\pi_N^C u\|_{L^\infty (\Omega)},
%\end{split}
%\end{equation}
%and
\begin{equation}\label{CGInterpolant-4}
\begin{split}
\|u-{\mathcal  I}_N^C u\|_{L^2_{\omega}(\Omega)}^2\le\frac{\pi}{2}\sum_{n=0}^N{'}|b_n-\hat u_n^{C}|^2+\|u-\pi_N^C u\|_{L^2_{\omega}(\Omega)}^2.
\end{split}
\end{equation}
Recall that (cf. \cite[(4.56)]{Boyd2000Book}):
\begin{align}\label{CGInterpolant-5}
b_n-\hat u_n^{C}=\sum_{k=1}^{\infty}(-1)^k(\hat u_{2k(N+1)-n}^{C}+\hat u_{2k(N+1)+n}^{C}), \quad n=0,\cdots,N.
\end{align}
Using \eqref{BoundForUnRealB} and \eqref{gamratioA}, we find that  for $N\gg 1,$  $\sigma=m+s>1$ and $n=0,\cdots,N,$
\begin{equation}\label{CGInterpolant-6}
\begin{split}
|b_n-\hat u_n^{C}|&\le \sum_{k=1}^{\infty}\big\{|\hat u_{2k(N+1)-n}^{C}|+|\hat u_{2k(N+1)+n}^{C}|\big\}
\le  \frac{U_\theta^{m,s}}{2^{\sigma-1} \pi}2\sum_{k=1}^{\infty}\frac{ \Gamma( ({2k(N+1)-n-\sigma}+1)/ 2)}{\Gamma( ({2k(N+1)-n+\sigma}+1)/2)}\\
&\le  \frac{U_\theta^{m,s}}{2^{\sigma-1} \pi}2\sum_{k=1}^{\infty}\frac{ \Gamma( ({2k(N+1)-N-\sigma}+1)/ 2)}{\Gamma( ({2k(N+1)-N+\sigma}+1)/2)}\le  CN^{-\sigma}U_\theta^{m,s}.
\end{split}
\end{equation}
From \eqref{CGInterpolant-4}, we obtain from   a direct calculation and Remark \ref{expbounds}  the $L^2_{\omega}$-estimate.
%\begin{equation}\label{CGInterpolant-7}
%\begin{split}
%\|u-{\mathbb  I}_N u\|_{L^\infty (\Omega)}\le CN^{1-\sigma} U_\theta^{m,s};\;\; \|u-{\mathbb  I}_N u\|_{L^2_\omega (\Omega)}\le CN^{\frac 1 2-\sigma} U_\theta^{m,s}.
%\end{split}
%\end{equation}
%
\end{proof}

\subsection{Analysis of endpoint singularities}\label{endptsingula}
The previous discussions were centred around the Chebyshev expansions and approximation of singular functions with interior singularities.
In what follows, we extend the results  to the cases with $\theta=\pm 1$, and study endpoint singularities.  To fix the idea, we shall focus on
the exact formulas and decay rate of the Chebyshev expansion coefficients, since it is the basis to derive many other related error bounds.

Let ${\mathbb W}_\pm^{m+s}(\Omega)$ be the fractional Sobolev-type spaces defined in  \eqref{resultpm1}. The following representation of $\hat u_n^C$ is a direct consequence of Theorem \ref{IdentityForUn}.
\begin{theorem}\label{Bndsingular} If  $u\in {\mathbb  W}^{\sigma}_{+}(\Omega) $ with $\sigma:=m+s,$ $s\in (0,1)$ and   $m\in  {\mathbb N}_0$, then for $n\ge \sigma>1/2,$
	\begin{equation}\label{HatUnCaseC0wP}
	\begin{split}
	&\hat u_n^{C}=(-1)^{ n+[ n-s]} C_{\sigma}\Big\{\int_{-1}^1   {}_{x}^R D_{1}^{s} \, u^{(m)}(x) \,  {}^{l}G_{n-\sigma}^{(\sigma)}(x) \, \omega_{\sigma}(x)\, dx\\
&\quad\quad\quad
	+ \big\{{}_{x}I_{1}^{1-s}u^{(m)}  (x) \, {}^{l}G_{n-\sigma}^{(\sigma)}(x)\, \omega_{\sigma}(x)\big\}\big|_{x=1}\Big\}.
	\end{split}
	\end{equation}
Similarly, if  $u\in {\mathbb  W}^{\sigma}_{-}(\Omega) $ with $\sigma:=m+s,$ $s\in (0,1)$ and   $m\in  {\mathbb N}_0$, then for $n\ge \sigma>1/2,$ 	
	\begin{equation}\label{HatUnCaseC0wQ}
	\begin{split}
	&\hat u_n^{C}=C_{\sigma}\Big\{\int_{-1}^1  \!\! {}_{-1}^{~~R\!\!} D_{x}^{s} \, u^{(m)}(x) \,  {}^{r\!}G_{n-\sigma}^{(\sigma)}(x) \, \omega_{\sigma}(x)\, dx
	+ \big\{{}_{-1}I_{x}^{1-s}u^{(m)}  (x) \, {}^{r\!}G_{n-\sigma}^{(\sigma)}(x)\, \omega_{\sigma}(x)\big\}\big|_{x=-1}\Big\}.
	\end{split}
	\end{equation}
Here, $\omega_{\lambda}(x)=(1-x^2)^{\lambda-1/2}$ and $C_{\sigma}:= ({\sqrt{\pi}\,2^{\sigma-1}\Gamma(\sigma+1/2)})^{-1}.$
%\begin{equation}\label{constMn}
%C_{\sigma}:= \frac{1}{\sqrt{\pi}\,2^{\sigma-1}\Gamma(\sigma+1/2)},\quad \omega_{\lambda}(x)=(1-x^2)^{\lambda-1/2}.
%\end{equation}
\end{theorem}
%\vskip 6pt

We next apply the formulas to several typical types of singular functions.   We first  consider  $u(x)=(1+x)^\alpha$ with $\alpha>-1/2$ and $\alpha\not \in {\mathbb N}_0$ (see, e.g., \cite{Tuan1972MC,Gui1986NM}).  Following the proof of Proposition \ref{Wspacproof},
 we  have  $u\in {\mathbb  W}^{\alpha+1}_{-}(\Omega).$
 Then   using \eqref{HatUnCaseC0wQ},  one  obtains the exact formula of
the Chebyshev expansion coefficient. Equivalently, one can derive it by  taking $\theta\to -1+$ in \eqref{SpecialCase}.  More precisely,
by \eqref{obsvers} and \eqref{SpecialCase},
\begin{equation}\label{SpecialCaseAB}
	\begin{split}
	\hat u_n^{C} &=\frac{\Gamma(\alpha+1)}{2^{\alpha}\Gamma(\alpha+3/2)\sqrt{\pi}}\lim_{\theta\to -1^+}\big\{{}^{r\!}G_{n-\alpha-1}^{(\alpha+1)}(\theta) \omega_{\alpha+1}(\theta)
-(-1)^{n+[n-\alpha]}  \,{}^{l}G_{n-\alpha-1}^{(\alpha+1)}(\theta) \omega_{\alpha+1}(\theta) \big\}\\
&=\frac{\Gamma(\alpha+1)}{2^{\alpha}\Gamma(\alpha+3/2)\sqrt{\pi}}\lim_{\theta\to -1^+}{}^{r\!}G_{n-\alpha-1}^{(\alpha+1)}(\theta) \omega_{\alpha+1}(\theta).
	\end{split}
	\end{equation}
Using    \eqref{defBS0B} leads to that for $\lambda>1/2,$
\begin{equation}\label{weithWithGfunat-1}
\begin{split}
\lim_{x\to -1^+} \omega_{\lambda}(x) \,  {}^{r\!}G_\nu^{(\lambda)}(x)&=-2^{2\lambda-1}
\frac{\sin(\nu\pi)}\pi
\frac{\Gamma(\lambda-1/2)\Gamma(\lambda+1/2)\Gamma(\nu+1)}{\Gamma(\nu+2\lambda)}.
\end{split}	
\end{equation}
Therefore,  from \eqref{DuplicationFormula}  and \eqref{SpecialCaseAB}-\eqref{weithWithGfunat-1}, we obtain the formula:
\begin{equation}\label{thetaeq-1}
\begin{split}
\hat u_n^{C}
&=\frac{(-1)^{n+1}\sin(\pi \alpha)\Gamma(2\alpha+1)}{2^{\alpha-1}\pi}
\frac{\Gamma(n-\alpha)}{\Gamma(n+\alpha+1)}, \quad n\ge \alpha+1,
\end{split}
\end{equation}
and for large $n,$ we have $|\hat u_n^{C}|=O(n^{-2\alpha-1}).$

\vskip 4pt
With the aid of \eqref{thetaeq-1}, we next consider a  more general case: $u(x)=(1+x)^\alpha g(x)$ with  $g(x)$ being a sufficiently smooth function.
Here, we need to use the formula of $\hat u_n^C$ for $(1+x)^\alpha$ with  $n<\alpha+1$. Taking $m=n$ in \eqref{IdentityForUn-1} and using the property  of the Beta function, yields
	\begin{equation}\label{nsmall}
\begin{split}
\hat u_n^{C}&=  \frac{1}{\sqrt{\pi}\,2^{n-1}\Gamma(n+1/2)} \int_{-1}^{1} u^{(n)}(x)\, G_{0}^{(n)}(x) \omega_{n}(x)\,dx
\\&=  \frac{1}{\sqrt{\pi}\,2^{n-1}\Gamma(n+1/2)}  \frac{\Gamma(\alpha+1)}{\Gamma(\alpha-n+1)} \int_{-1}^{1}(1+x)^{\alpha-n}(1-x^2)^{n-1/2}\,dx
\\&=  \frac{2^{\alpha+1}\Gamma(\alpha+1)\Gamma(\alpha+1/2)}{\sqrt{\pi}\Gamma(\alpha-n+1)\Gamma(\alpha+n+1)}.
\end{split}
\end{equation}
%By the Taylor's formula, we obtain
%$$u(x)=\sum_{l=0}^\infty\frac{ g^{(l)}(-1)}{l!}(1+x)^{\alpha+l}.$$
Using the Taylor expansion of $g(x)$ at $x=-1,$ we obtain from \eqref{thetaeq-1}-\eqref{nsmall} that
\begin{equation}\label{anymptotic}
\begin{split}
\hat u_n^{C}=&  \sum _{l=0}^{[n-1-\alpha]}\frac{ g^{(l)}(-1)}{l!}\frac{(-1)^{n+1}\sin(\pi (\alpha+l))\Gamma(2\alpha+2l+1)}{2^{\alpha+l-1}\pi}
\frac{\Gamma(n-\alpha-l)}{\Gamma(n+\alpha+l+1)}\\
&+\sum _{l=[n-\alpha]}^{\infty}\frac{ g^{(l)}(-1)}{l!} \frac{2^{\alpha+l+1}\Gamma(\alpha+l+1)\Gamma(\alpha+l+1/2)}{\sqrt{\pi}\Gamma(\alpha+l-n+1)\Gamma(\alpha+l+n+1)}\\
=&\frac{(-1)^{n+1}g(-1)\sin(\pi \alpha)\Gamma(2\alpha+1)}{2^{\alpha-1}\pi}
n^{-2\alpha-1}+O(n^{-2\alpha-3}),
\end{split}
\end{equation}
where we used \eqref{gamratioA}.
%%It is noteworthy that one can find this formula
%\begin{rem}\label{neeaddform}  {\em One can find asymptotic formulas similar to \eqref{thetaeq-1} and \eqref{anymptotic}
%in  \cite{Tuan1972MC}, which were obtained from  a very different approach. }
%%Observe from Proposition \ref{Wspacproof} and the above formulas that  an order $O(n^{-\alpha})$ is gained for the Cheyshev  expansion coefficient when the singular point is at $x=\pm 1.$ \qed
% \end{rem}

Finally, we consider the singular function:    $u(x)=(1+x)^\alpha\ln (1+x)$. Using    \eqref{LogFun}, we derive from  	\eqref{SpecialCasetheta=01} and the L'Hospital's rule that
	\begin{equation*}
	\begin{split}
	\hat u_{n}^C&= \frac 2\pi  \int_{-1}^1 \Big\{\lim_{\varepsilon\to 0}\frac{(1+x)^{\varepsilon+\alpha}-(1+x)^\alpha}{\varepsilon}\Big\}\frac{T_{n}(x)} {\sqrt{1-x^2}} dx \\
	&=(-1)^{n+1}\frac{2}{\pi}\lim_{\varepsilon\to 0}\frac{1}{\varepsilon}\Big\{ \frac{\sin(\pi (\alpha+\varepsilon))\Gamma(2\alpha+2\varepsilon+1)\Gamma(n-\alpha-\varepsilon)}{2^{\alpha+\varepsilon}\Gamma(n+\alpha+\varepsilon+1)}
	 -\frac{\sin(\pi \alpha)\Gamma(2\alpha+1)\Gamma(n-\alpha)}{2^{\alpha}\Gamma(n+\alpha+1)}\Big\}\\
	&=\frac{(-1)^{n+1}\Gamma(2\alpha+1)\Gamma(n-\alpha)}{\pi 2^{\alpha-1}\Gamma(n+\alpha+1)}	 \big\{\pi\cos(\alpha\pi)
	+ \sin(\alpha\pi)
	\big(2\psi(2\alpha+1)\\
	&\quad
	-\ln 2-\psi(n+\alpha+1)-\psi(n-\alpha)\big)\big\}.
	\end{split}
	\end{equation*}
	%Similar to Remark \ref{RemforLog},
	In view of \eqref{psizfun}, we can  the asymptotic behaviour
	\begin{equation*}
	\begin{split}
	\hat u_{n}^C
	&=\frac{(-1)^{n+1}\Gamma(2\alpha+1)}{\pi 2^{\alpha-1}}n^{-2\alpha-1}	 \big\{\pi\cos(\alpha\pi)
	+ \sin(\alpha\pi)
	\big(2\psi(2\alpha+1)
	-\ln 2\\
	&\quad -2\ln n \big)\big\}+O(n^{-2\alpha-3}\ln n)\,\sin (\alpha\pi)+O(n^{-2\alpha-3}).
	\end{split}
	\end{equation*}

\begin{rem}\label{estACheb}{\em  With the above analysis of the expansion coefficients, we can then obtain directly the optimal estimates for  the Chebyshev approximation to these specific singular functions. More precisely,   for $u(x)=(1+x)^\alpha g(x)$ with  $g(x)$ being a sufficiently smooth function, we have
\begin{equation}\label{smothhfuncA}
\|u-\pi_N^C u\|_{L^\infty(\Omega)}\le CN^{-2\alpha},\quad \|u-\pi_N^C u\|_{L^2_\omega(\Omega)}\le CN^{-2\alpha-1/2},
\end{equation}
and for  $u(x)=(1+x)^\alpha\ln (1+x),$ we have
\begin{equation}\label{smothhfuncB}
\|u-\pi_N^C u\|_{L^\infty(\Omega)}\le C(\ln N)N^{-2\alpha},\quad \|u-\pi_N^C u\|_{L^2_\omega(\Omega)}\le C(\ln N) N^{-2\alpha-1/2}.
\end{equation}
Compared with the interior singularities {\rm(}see \eqref{newbndaddbus} and  \eqref{modeeqnA}{\rm)}, a higher  convergence order $O(N^{-\alpha})$ is observed which is as expected.
}\end{rem}

\vskip 6pt

\subsection{Concluding remarks} Broadly speaking, we position this work as our first attempt to show how the RL  fractional calculus   can  alter the fundamental polynomial approximation theory.  Some estimates and bounds  herein are completely new,  or significantly improve the existing results.

More precisely,  we introduce a new theoretical framework of fractional Sobolev-type spaces  for orthogonal polynomial approximations to functions with limited regularities (or interior/endpoint singularities).
The proposed  spaces are  naturally arisen from the analytic representations of the expansion coefficients involving RL fractional integrals/derivatives and GGF-Fs.
We present a collection of notable properties of the new family of GGF-Fs, and  derive optimal estimates of Chebyshev approximations in various norms for a wide class of singular functions.     The analysis techniques can be extended to general Jacobi approximations.
%In late parts, we  apply this framework to multiple dimensions and problems with corner singularities.    With the explicit information in the fractional spaces,
 We are confident that this study,  together with our follow-up works,  will have far-reaching impact on  numerical analysis of $p$-version and $hp$-version for singular problems.
 % can be a superior alternative to the limited existing tools for characterising singular functions.

%-------------------------------
 %\bibliography{RefFractional,ref,refForAddNew}

\begin{thebibliography}{10}
 	
 	\bibitem{Alzer1997MC}
 	H.~Alzer.
 	\newblock {On some inequalities for the Gamma and Psi functions}.
 	\newblock {\em Math. Comput.}, 66(217):373--389, 1997.
 	
 	\bibitem{Alzer2003MPCPS}
 	H.~Alzer and S.~Koumandos.
 	\newblock {Sharp inequalities for trigonometric sums}.
 	\newblock {\em Math. Proc. Camb. Phil. Soc.}, 139:139--152, 2003.
 	
 	\bibitem{Andrews1999Book}
 	G.E. Andrews, R.~Askey, and R.~Roy.
 	\newblock {\em {Special Functions, Encyclopedia of Mathematics and its
 			Applications, Vol. 71}}.
 	\newblock Cambridge University Press, Cambridge, 1999.
 	
 	\bibitem{Bab.G00}
 	I.~Babu{\v{s}}ka and B.Q. Guo.
 	\newblock Optimal estimates for lower and upper bounds of approximation errors
 	in the $p$-version of the finite element method in two dimensions.
 	\newblock {\em Numer. Math.}, 85(2):219--255, 2000.
 	
 	\bibitem{Bab.G01}
 	I.~Babu{\v{s}}ka and B.Q. Guo.
 	\newblock Direct and inverse approximation theorems for the $p$-version of the
 	finite element method in the framework of weighted {B}esov spaces. {Part I}:
 	{A}pproximability of functions in the weighted {B}esov spaces.
 	\newblock {\em SIAM J. Numer. Anal.}, 39(5):1512--1538, 2001.
 	
 	 	\bibitem{Babuska2002MMMAS}
 	{I. Babu\v{s}ka} and {B.Q. Guo}.
 	\newblock {Direct and inverse approximation theorems for the {$p$}-version of
 		the finite element method in the framework of weighted {B}esov spaces, {Part
 			II}: {O}ptimal rate of convergence of the {$p$}-version finite element
 		solutions}.
 	\newblock {\em Math. Models Methods Appl. Sci.}, 12(5):689--719, 2002.
 	
 	
 	\bibitem{Berg2016}
 	M.~Bergounioux, A.~Leaci, G.~Nardi, and F.~Tomarelli.
 	\newblock {Fractional Sobolev spaces and functions of bounded variation}.
 	\newblock {\em Fract. Calc. Appl. Anal.}, 20(4):936--962, 2017.
 	
 	\bibitem{MR1470226}
 	C.~Bernardi and Y.~Maday.
 	\newblock {Spectral Methods}.
 	\newblock In {P.G. Ciarlet} and {J.L. Lions}, editors, {\em {Handbook of
 			Numerical Analysis, {V}ol. {V}, {Part 2}}}, pages 209--485. North-Holland,
 	Amsterdam, 1997.
 	
 	\bibitem{Bourdin2015ADE}
 	L.~Bourdin and D.~Idczak.
 	\newblock A fractional fundamental lemma and a fractional integration by parts
 	formula-applications to critical points of {Bolza} functionals and to linear
 	boundary value problems.
 	\newblock {\em Adv. Differential Equations}, 20(3--4):213--232, 2015.
 	
 	\bibitem{Boyd1989AMC}
 	J.P. Boyd.
 	\newblock {The asymptotic Chebyshev coefficients for functions with logarithmic
 		endpoint singularities}.
 	\newblock {\em Appl. Math. Comput.}, 29:49--67, 1989.
 	
 	\bibitem{Boyd2000Book}
 	J.P. Boyd.
 	\newblock {\em {Chebyshev and Fourier Spectral Methods, 2nd Ed.}}
 	\newblock Dover, New York, 2001.
 	
 	\bibitem{Bustoz1986MC}
 	J.~Bustoz and M.E.H. Ismail.
 	\newblock {On Gamma function inequalities}.
 	\newblock {\em Math. Comput.}, 47(176):659--667, 1986.
 	
 	\bibitem{CHQZ06}
 	C.~Canuto, M.Y. Hussaini, A.~Quarteroni, and T.A. Zang.
 	\newblock {\em {Spectral Methods: Fundamentals in Single Domains}}.
 	\newblock Springer, Berlin, 2006.
 	
 	\bibitem{Castillo2002MC}
 	P.~Castillo, B.~Cockburn, D.~Sch\"otzau, and C.~Schwab.
 	\newblock Optimal a priori error estimates for the {$hp$}-version of the local
 	discontinuous {G}alerkin method for convection-diffusion problems.
 	\newblock {\em Math. Comp.}, 71(238):455--478, 2002.
 	
 	\bibitem{Chen.SW2016}
 	S.~Chen, J.~Shen, and L.L. Wang.
 	\newblock Generalized {J}acobi functions and their applications to fractional
 	differential equations.
 	\newblock {\em Math. Comp.}, 85(300):1603--1638, 2016.
 	
 	\bibitem{Funa92}
 	D.~Funaro.
 	\newblock {\em Polynomial Approxiamtions of Differential Equations}.
 	\newblock Springer-Verlag, Berlin, 1992.
 	
 	
 	 	\bibitem{Gui1986NM}
 	{W. Gui} and {I. Babu\v{s}ka}.
 	\newblock {The $h,~ p$ and $h$-$p$ versions of the finite element method in 1
 		dimension, Part I: The error analysis of the $p$-version}.
 	\newblock {\em Numer. Math.}, 49:205--612, 1986.
 	
 	\bibitem{Guo.SW06}
 	B.Y. Guo, J.~Shen, and L.L. Wang.
 	\newblock Optimal spectral-{G}alerkin methods using generalized {J}acobi
 	polynomials.
 	\newblock {\em J. Sci. Comput.}, 27(1-3):305--322, 2006.
 	
 	\bibitem{Guo.SW09}
 	B.Y. Guo, J.~Shen, and L.L. Wang.
 	\newblock Generalized {J}acobi polynomials/functions and their applications.
 	\newblock {\em Appl. Numer. Math.}, 59(5):1011--1028, 2009.
 	
 	\bibitem{Guo.W04}
 	B.Y. Guo and L.L. Wang.
 	\newblock {Jacobi approximations in non-uniformly Jacobi-weighted Sobolev
 		spaces}.
 	\newblock {\em J. Approx. Theory}, 128(1):1--41, 2004.
 	
 	\bibitem{HGG07}
 	J.~Hesthaven, S.~Gottlieb, and D.~Gottlieb.
 	\newblock {\em Spectral Methods for Time-Dependent Problems}.
 	\newblock Cambridge University Press, Cambridge, 2007.
 	

 	\bibitem{Klebaner2005Book}
 	F.C. Klebaner.
 	\newblock {\em {Introduction to Stochastic Calculus with Applications, 2nd
 			Ed}}.
 	\newblock Imperial College Press, London, 2005.
 	
 	\bibitem{Leoni2009Book}
 	G.~Leoni.
 	\newblock {\em {A First Course in Sobolev Spaces}}.
 	\newblock Amer. Math. Soc., Providence, RI, 2009.
 	
 	\bibitem{Majidian2017ANM}
 	H.~Majidian.
 	\newblock {On the decay rate of Chebyshev coefficients}.
 	\newblock {\em Appl. Numer. Math.}, 113:44--53, 2017.
 	
 	\bibitem{Nevai1994SIMA}
 	P.~Nevai, {T. Erd\'elyi}, and {A.P. Magnus}.
 	\newblock Generalized {J}acobi weights, {C}hristoffel functions, and {J}acobi
 	polynomials.
 	\newblock {\em SIAM J. Math. Anal.}, 25(2):602--614, 1994.
 	
 	\bibitem{Nezza2012BSM}
 	E.~Di Nezza, G.~Palatucci, and E.~Valdinoci.
 	\newblock {Hitchhiker’s guide to the fractional Sobolev spaces}.
 	\newblock {\em Bull. Sci. Math.}, 136:521--573, 2012.
 	
 	\bibitem{Olver1974Book}
 	F.W.J. Olver.
 	\newblock {\em {Asymptotics and Special Functions}}.
 	\newblock Academic Press, New York, 1974.
 	
 	\bibitem{Olver2010Book}
 	F.W.J. Olver, D.W. Lozier, R.F. Boisvert, and C.W. Clark.
 	\newblock {\em {NIST Handbook of Mathematical Functions}}.
 	\newblock Cambridge University Press, New York, 2010.
 	
 	\bibitem{Riess1969SINUM}
 	{R.D. Riess} and {L.W. Johnson}.
 	\newblock Estimating {G}auss-{C}hebyshev quadrature errors.
 	\newblock {\em SIAM J. Numer. Anal.}, 6:557--559, 1969.
 	
 	\bibitem{Samko1993Book}
 	S.G. Samko, A.A. Kilbas, and O.I. Marichev.
 	\newblock {\em {Fractional Integrals and Derivatives, Theory and
 			Applications}}.
 	\newblock Gordan and Breach Science Publisher, New York, 1993.
 	
 	\bibitem{Schwab1998Book}
 	C.~Schwab.
 	\newblock {\em {$p$- and $hp$-FEM. Theory and Application to Solid and Fluid
 			Mechanics}}.
 	\newblock Oxford University Press, New York, 1998.
 	
 	\bibitem{ShenTangWang2011}
 	J.~Shen, T.~Tang, and L.L. Wang.
 	\newblock {\em {Spectral Methods: Algorithms, Analysis and Applications}},
 	volume~41 of {\em Series in Computational Mathematics}.
 	\newblock Springer-Verlag, Berlin, Heidelberg, 2011.
 	
 	\bibitem{szeg75}
 	G.~Szeg\"o.
 	\newblock {\em Orthogonal Polynomials, 4th Ed}.
 	\newblock Amer. Math. Soc., Providence, RI, 1975.
 	
 	\bibitem{Trefethen2008SIREV}
 	L.N. Trefethen.
 	\newblock Is {Gauss} quadrature better than {Clenshaw-Curtis}?
 	\newblock {\em SIAM Rev.}, 51(1):67--87, 2008.
 	
 	\bibitem{Trefethen2013Book}
 	L.N. Trefethen.
 	\newblock {\em {Approximation Theory and Approximation Practice}}.
 	\newblock SIAM, Philadelphia, 2013.
 	
 	\bibitem{Tuan1972MC}
 	P.D. Tuan and D.~Elliott.
 	\newblock {Coefficients in series expansions for certain classes of functions}.
 	\newblock {\em Math. Comp.}, 26:213--232, 1972.
 	

 	
 	\bibitem{Wang2014arXiv}
 	H.Y. Wang.
 	\newblock On the convergence rate of {C}lenshaw-{C}urtis quadrature for
              integrals with algebraic endpoint singularities.
 	\newblock {\em J. Comput. Appl. Math.}, 333:87--98, 2018.
 	
 	\bibitem{Xiang2010NM}
 	S.H. Xiang, X.J. Chen, and H.Y. Wang.
 	\newblock {Error bounds for approximation in Chebyshev points}.
 	\newblock {\em Numer. Math.}, 116:463--491, 2010.
 	
 	\bibitem{zayernouri2013fractional}
 	M.~Zayernouri and G.E. Karniadakis.
 	\newblock Fractional {S}turm-{L}iouville eigen-problems: theory and numerical
 	approximation.
 	\newblock {\em J. Comput. Phys.}, 252:495--517, 2013.
 	
 \end{thebibliography}

\end{document}